\documentclass[11]{amsart}
\usepackage[height=8.5in,  width=6in, centering]{geometry}
\usepackage{amsmath,amsfonts,amsxtra}
\usepackage[psamsfonts]{amssymb}
\usepackage{mathtools}
\usepackage[mathscr]{eucal}
\usepackage{mathrsfs}
\usepackage[all]{xy}
\usepackage{hyperref}
\usepackage{enumerate}
\usepackage{color}
\numberwithin{equation}{section}

\theoremstyle{plain}
        \newtheorem{theorem}{Theorem}[section]
        \newtheorem{lemma}[theorem]{Lemma}
        \newtheorem{proposition}[theorem]{Proposition}
        \newtheorem{corollary}[theorem]{Corollary}
        \newtheorem{conjecture}[theorem]{Conjecture}
        \newtheorem{propdef}[theorem]{Proposition and Definition}
\theoremstyle{definition}
        \newtheorem{definition}[theorem]{Definition}
        \newtheorem{remark}[theorem]{Remark}
        \newtheorem{example}{Example}[section]

\DeclareMathAlphabet{\mathboondoxcal}{U}{BOONDOX-cal}{m}{n}
\SetMathAlphabet{\mathboondoxcal}{bold}{U}{BOONDOX-cal}{b}{n}
\DeclareMathAlphabet{\mathbboondoxcal} {U}{BOONDOX-cal}{b}{n}

\newcommand{\orbO}{\mathboondoxcal{O}}
\newcommand{\smorbO}{{\scriptstyle\mathboondoxcal{O}}}

\newcommand{\smorbQ}{{\scriptstyle\mathboondoxcal{Q}}}
        
\newcommand{\N}{{\mathbb N}}

\newcommand{\C}{{\mathbb C}}
\newcommand{\R}{{\mathbb R}}

\newcommand{\sfG}{\mathsf{G}}

\newcommand{\sfH}{\mathsf{H}}

\newcommand{\End}{\operatorname{End}}

\newcommand{\Tot}{\operatorname{Tot}}

\newcommand{\supp}{\operatorname{supp}}

\newcommand{\id}{\operatorname{id}}
\newcommand{\prtwo}{\operatorname{pr}_2}
\newcommand{\colim}{\operatornamewithlimits{colim}}

\newcommand{\calD}{\mathcal{D}}
\newcommand{\calA}{\mathcal{A}}
\newcommand{\wcalA}{\widetilde{\mathcal{A}}}
\newcommand{\wA}{\widetilde{A}}
\newcommand{\calB}{\mathcal{B}}

\newcommand{\calL}{\mathcal{L}}

\newcommand{\calJ}{\mathcal{J}}
\newcommand{\calS}{\mathcal{S}}

\newcommand{\calC}{\mathcal{C}}
\newcommand{\calO}{\mathcal{O}}
\newcommand{\calK}{\mathcal{K}}

\newcommand{\scrB}{\mathscr{B}}
\newcommand{\scrC}{\mathscr{C}}
\newcommand{\scrD}{\mathscr{D}}

\newcommand{\frakg}{\mathfrak{g}}

\newcommand{\frakk}{\mathfrak{k}}
\newcommand{\frakm}{\mathfrak{m}}
\newcommand{\Ad}{\operatorname{Ad}}

\newcommand{\sheafC}{\hat{\mathscr{C}}}
\newcommand{\HHsheaf}{\mathscr{H}\mathscr{H}}
\newcommand{\im}{\operatorname{im}}
\newcommand{\hatotimes}{\hat{\otimes}}
\newcommand{\ltwist}[2]{{^{#2}}{#1}}
\newcommand{\rOmega}[2]{\Omega^{#1}_{\textup{rel},#2}}
\newcommand{\hrOmega}[2]{\Omega^{#1}_{\textup{hrel},#2}}
\newcommand{\brOmega}[2]{\Omega^{#1}_{\textup{brel},#2}}

\newcommand{\tildeU}{\widetilde{U}}

\newcommand{\doublegroupoid}[4]{\xymatrix{#1  \ar@<2pt>[d] \ar@<-2pt>[d] & #2 \ar@<2pt>[l] 
\ar@<-2pt>[l] \ar@<2pt>[d] \ar@<-2pt>[d] \\ #3  & #4 \ar@<2pt>[l] \ar@<-2pt>[l]}}

\title{On the Hochschild homology of convolution algebras of proper Lie groupoids}

\author{M.J.~Pflaum, H. Posthuma,~\textrm{and} X.~Tang}
\address{\newline
Markus J. Pflaum, {\tt markus.pflaum@colorado.edu}\newline
         \indent {\rm Department of Mathematics, University of Colorado,
         Boulder, USA}
         \newline
        Hessel Posthuma, {\tt H.B.Posthuma@uva.nl}\newline
         \indent {\rm Korteweg-de Vries Institute for Mathematics,
        University of Amsterdam,
         The Netherlands} 
         \newline
        Xiang Tang, {\tt xtang@math.wustl.edu}   \newline
         \indent {\rm  Department of Mathematics, Washington University,
         St.~Louis, USA}}
\begin{document}
\maketitle
\begin{abstract}
We study the Hochschild homology of the convolution algebra of a proper Lie groupoid by introducing a convolution sheaf over the space of orbits. We develop a localization result for the associated Hochschild homology sheaf, and prove that the Hochschild homology sheaf at each stalk is quasi-isomorphic to the stalk at the origin of the Hochschild homology of the convolution algebra of its linearization, which is the transformation groupoid of a  linear action of a compact isotropy group on a vector space.  We then explain Brylinski's ansatz  to compute the Hochschild homology of the transformation groupoid of a compact group action on a manifold. We verify Brylinski's conjecture for the case of smooth circle actions that the Hochschild homology is given by basic relative forms on the associated inertia space.
\end{abstract}

\section*{Introduction}
\label{Sec:Intro}
Let $M$ be a smooth manifold, and $\calC^\infty(M)$ be the algebra of smooth functions on $M$. Connes' version  \cite{ConNDG} of the seminal Hochschild-Kostant-Rosenberg theorem \cite{HocKosRosDFRAA} states that the Hochschild homology of $\calC^\infty(M)$ is isomorphic  to the graded vector space of differential forms on $M$. In this paper, we aim to establish tools for a general Hochschild-Kostant-Rosenberg type theorem for proper Lie groupoids. 

Recall that a Lie groupoid $\sfG\rightrightarrows M$ is proper if the map $\sfG \to M\times M$, $g\mapsto (s(g), t(g))$ is a proper map, where $s(g)$ and $t(g)$ are the source and target of $g\in \sfG$. When the source and target maps are both local diffeomorphisms, the groupoid $\sfG\rightrightarrows M$ is called \'etale. In efforts by many authors, e.g. \cite{BroDavNis, BryNis,  ConNG, CrainicCyc,FeiTsy, Ponge, Wassermann}, the Hochschild and cyclic homology theory of \'etale Lie groupoids has been unvealed. The Hochschild and cyclic homology of a proper \'etale Lie groupoid
was explicitly computed by Brylinksi and Nistor \cite{BryNis}. Let us explain this result in the case of a finite group $\Gamma$ action on a smooth manifold $M$, the transformation groupoid $\Gamma\ltimes M\rightrightarrows M$ for a finite group 
$\Gamma$ action on $M$. 

The convolution groupoid algebra associated to the transformation groupoid $\Gamma\ltimes M\rightrightarrows M$ is the crossed product algebra $\calC^\infty(M)\rtimes \Gamma$, which consists of $\calC^\infty(M)$-valued functions on $\Gamma$ equipped with the convolution product, e.g.\ for $f,g\in \calC^\infty(M)\rtimes \Gamma$, 
\[
f\ast g(\gamma)=\sum_{\alpha\beta=\gamma} \beta^*\big(f(\alpha)\big)\cdot  g(\beta). 
\]
The algebra $\calC^\infty(M)\rtimes \Gamma$ is naturally a Fr\'echet algebra. The Hochschild homology of the algebra $\calC^\infty(M)\rtimes \Gamma$ as a bornological algebra is given by the following formula the  proof of which is recalled in Corollary \ref{cor:finitegroup}.
\begin{equation*}\label{eq:finitegroup}
HH_\bullet\big( \calC^\infty(M)\rtimes \Gamma  \big) \cong  \left(\bigoplus_{\gamma \in \Gamma} \Omega^\bullet (M^\gamma) \right)^\Gamma,
\end{equation*}
where $M^\gamma$ is the $\gamma$-fixed point submanifold, and $\Gamma$ acts on the disjoint union  $\coprod_{\gamma \in \Gamma} M^\gamma$ by $\gamma'(\gamma,x)=(\gamma' \gamma(\gamma')^{-1}, \gamma' x)$. Recall that the so called loop space $\Lambda_0(\Gamma, M)$ of the transformation groupoid $\Gamma \ltimes M\rightrightarrows M$ is defined as 
\[
\Lambda_0(\Gamma, M):=\coprod_{\gamma \in \Gamma} M^\gamma,
\]
equipped with the same action of $\Gamma$ as above. In other words, the Hochschild homology of $\calC^\infty(M)\rtimes \Gamma$ is the space of differential forms on the quotient $\Lambda_0(\Gamma, M)/\Gamma$, which is called the associated inertia orbifold.  We would like to remark that just as the classical Hochschild-Kostant-Rosenberg theorem, the above identification can be realized as  an isomorphism of sheaves over the quotient $M/\Gamma$.  This makes Hochschild and cylic homology of $\calC^\infty(M)\rtimes \Gamma$ the right object to work with in the study of orbifold index theory, see e.g.\ \cite{PflPosTanCyc}.

Our goal in this project is to extend the study of Hochschild homology of proper \'etale groupoids to general proper Lie groupoids, which are natural generalizations of transformation groupoids for proper Lie group actions. The key new challenge from the study of (proper) \'etale groupoids is that orbits of a general proper Lie groupoid have different dimensions. This turns the orbit space of a proper Lie groupoid into a stratified space with a significantly more complicated singularity structure than an orbifold.  

Our main result is to introduce a sheaf $ \HHsheaf_\bullet$ on the orbit space $X:=M/\sfG$ of a proper Lie groupoid $\sfG\rightrightarrows M$, whose space of global sections computes the Hochschild homology of the convolution algebra of $\sfG$. To achieve this, we start with introducing a sheaf $\calA$ of convolution algebras on the orbit space $X$ in Definition \ref{propdefn:sheaf}. Using the localization method from \cite{BraPflHAWFSS} we introduce the Hochschild homology sheaf $\HHsheaf_\bullet(\calA)$ for $\calA$ as a sheaf of bornological algebras over $X$. Moreover, we prove the following sheafification theorem for the Hochschild homology of the convolution algebra $\calA$ of the groupoid $\sfG$. \vspace{2mm}

\noindent\textbf{Theorem \ref{thm:hochschild-homology-global-sections}.}
\textit{Let $\calA$ be the convolution sheaf of a proper Lie groupoid $\sfG$. 
  Then the natural map in Hochschild homology
  \[
    HH_\bullet \big( \calA (X)\big) \to \HHsheaf_\bullet (\calA) (X)
    = \Gamma \big(X,  \HHsheaf_\bullet (\calA) \big) 
  \]
  is an isomorphism.}\vspace{2mm}

To determine the homology sheaf $\HHsheaf_\bullet(\calA)$, we study its stalk at an orbit $\calO\in X$. Using the linearization result of proper Lie groupoid developed by Weinstein and Zung (c.f. \cite{CraStrLTPLG, delHFerRMLG,PflPosTanGOSPLG, WeiLRPG, ZunPGMMLAC}), we obtain a linear model of the stalk $\HHsheaf_{\bullet, \calO}(\calA)$ in Proposition \ref{prop:local-model} as a linear compact group action on a vector space.  This result leads us to focus on the  Hochschild homology of the convolution algebra
$\calC^\infty (M) \rtimes G $ associated to a compact Lie group action on a smooth manifold $M$ in the second part of this article. 

The Hochschild homology of compact Lie group actions was studied by several authors, e.g. \cite{BloGetECHED}, \cite{BryAAGAH,BryCHET}. Brylinski \cite{BryAAGAH,BryCHET} proposed a geometric model of basic relative forms along the idea of the Grauert-Grothendieck forms to compute the Hochschild homology. However, a major part of the proof is missing in  \cite{BryAAGAH,BryCHET}. We decided to turn this result into the main conjecture of this paper in Section \ref{sec:actioncase}. \vspace{2mm}

\noindent\textbf{Conjecture \ref{BryConjBHF}.}
\textit{The Hochschild homology of the crossed product algebra $\calC^\infty(M)\rtimes G$ associated to a compact Lie group
  action on a smooth manifold $M$ is isomorphic to the space of basic relative forms on the loop space
  $\Lambda_0 (G\ltimes M) =\{(g,p)\in G\times M\mid gp = p\}$.}\vspace{2mm}

Block and Getzler \cite{BloGetECHED} introduced an interesting Cartan model for the cyclic homology of the crossed product algebra $\calC^\infty(M)\rtimes G$. However, the Block-Getzler model is not a sheaf on the orbit space $M/G$, but a sheaf on the space of conjugacy classes of $G$. This makes it impossible to localize the sheaf to an orbit of the group action in the orbit space. It is worth pointing out that the truncation of the Block-Getzler Cartan model at $E^1$-page provides a complex to compute the Hochschild homology of $\calC^\infty(M)\rtimes G$. However, the differential $\iota$ introduced in \cite[Section 1]{BloGetECHED} is nontrivial, and  makes it challenging to explicitly identify the Hochschild homology of $\calC^\infty(M)\rtimes G$ as the space of basic relative forms. We refer the reader to Remark \ref{rem:block-getzler} for a more detail discussion about the Block-Getzler model. 

In the last part of this paper, we prove Conjecture \ref{BryConjBHF} in the case where the group $G$  is $S^1$; see Proposition \ref{prop:equivariant-koszul}.
Our proof relies on a careful study of the stratification of the loop space $\Lambda_0(S^1\ltimes M)\subset S^1\times M$. The crucial property we use in our computation is that at its singular point, $\Lambda_0(S^1\ltimes M)$ locally looks like the union of the hyperplane $\{x_0=0\}$ and the line $\{x_1=\cdots =x_n=0\}$ in $\mathbb{R}^{n+1}$, which are transverse to each other. The loop space $\Lambda_0(G\ltimes M)$ for a general
$G$-manifold $M$ is much more complicated to describe. This has stopped us from extending our result for $S^1$-actions to more general compact group actions. It is foreseeable that some combinatorial structures describing the stratifications of the loop spaces and real algebraic geometry tools characterizing basic relative forms on the loop spaces are needed to solve Conjecture \ref{BryConjBHF} in full generality. We plan to come back to this problem in the near future. 

As is mentioned above, the study of Hochschild and cyclic homology of the convolution algebra of a proper Lie groupoid is closely related to the study of the groupoid index theory, e.g. \cite{PflPosTanCyc}, \cite{PflPosTanLLIT}. We expect that the study of the Hochschild homology and the generalized Hochschild-Kostant-Rosenberg theorem will eventually lead to the correct definition of basic relative forms for proper Lie groupoids, where the right index theorem will be established. \\ 

\noindent{\bf Acknowledgements}: We would like to thank Marius Crainic, Ralf Meyer, Rapha\"el Ponge and Michael Puschnigg for inspiring discussions. Pflaum's research is partially supported by Simons Foundation award number 359389
and NSF award OAC 1934725. Tang's research is partially supported by the NSF awards  DMS 1800666, 1952551. 

%
%
\section{The convolution sheaf of a proper Lie groupoid}
Throughout this paper, $\sfG\rightrightarrows M$ denotes a Lie groupoid over a 
base manifold $M$. Elements of $M$ are called points of the groupoid, those of $\sfG$ its arrows.
The symbols $s,t:\sfG\to M$ denote the source and target map, respectively, and 
$u :M \rightarrow \sfG$ the unit map. 
By definition of a Lie groupoid, $s$ and $t$ are assumed to be smooth submersions. This implies that
the space of $k$-tuples of composable arrows
\[
  \sfG_k:=\{(g_1,\ldots,g_k)\in\sfG^k \mid s(g_i)=t(g_{i+1}) 
  \text{ for $i=1,\ldots,k-1$}\}
\]
is a smooth manifold, and multiplication of arrows 
\[
  m: \sfG_2\to\sfG, \: (g_1,g_2) \mapsto g_1 \, g_2
\]
a smooth map.

If $g\in \sfG$ is an arrow with $s(g)=x $ and $t(g)=y$, 
we denote such an arrow sometimes by $g:y \leftarrow x$, and write $\sfG (y,x)$ for the space of arrows 
with source $x$ and target $y$. 
The $s$-fiber over $x$, i.e.~the manifold $s^{-1} (x)$, will be denoted 
by $\sfG ( - , x)$, the $t$-fiber  over  $y$ by $\sfG (y, -)$. 
Note that for each object $x\in M$ multiplication of arrows induces
on $\sfG(x,x)$ a group structure. This group is called the \emph{isotropy 
group} of $x$ and is denoted by $\sfG_x$. The union of all isotropy groups 
\[
  \Lambda_0 \sfG := \bigcup_{x\in M} \sfG_x = \{ g\in \sfG \mid s(g)=t(g) \}
\]
will be called the \emph{loop space} of $\sfG$.

Given a Lie groupoid $\sfG\rightrightarrows M$ two points 
$x,y \in M$ are said to lie in the same orbit if there is an arrow $g: y \leftarrow x$. 
In the following, we will always write $\calO_x$ for the orbit containing $x$, and 
$M/\sfG$ for the space of orbits of the groupoid $\sfG$.
We assume further that the orbit space always carries the quotient topology with respect to the
canonical map $\pi: M \to M/\sfG$.
Note that $M/\sfG$ need not be Hausdorff unless $\sfG$ is a proper Lie groupoid, which  means 
that the map $(s,t):\sfG \rightarrow M\times M$ is a  proper map. 

Sometimes, we need to specify to which groupoid a particular structure map belongs to. In such 
a situation we will write $s_\sfG$, $m_\sfG$, $\pi_\sfG$ and so on.

In the following, we will define a sheaf of algebras $\calA$ on $M/\sfG$ in such a way that the 
algebra $\calA_\textup{c}(M/\sfG)$ of compactly supported global sections of $\calA$ coincides with the 
smooth convolution algebra of the groupoid. To this end, we use a smooth left Haar measure on $\sfG$.

Recall that by a smooth left Haar measure on $\sfG$ one understands a family of 
measures $(\lambda^x)_{x\in M}$  such that the following properties hold true:
\begin{enumerate}[(H1)]
\item For every $x \in \sfG_0$, $\lambda^x$ is a positive measure on $\sfG(x,-)$ with $\supp \lambda^x = \sfG(x,-)$.
\item For every $g\in \sfG$, the family $(\lambda^x)_{x\in M}$ is invariant under left 
      multiplication 
      \[ L_g : \sfG(s(g),-) \to \sfG(t(g),-), \: h \mapsto gh \]
      or in other words 
      \[ 
          \int\limits_{\sfG (s(g),-) } \! u(gh) \, d\lambda^{s(g)} (h) = 
          \int\limits_{\sfG (t(g),-) } \! u(h) \, d\lambda^{t(g)} (h) \quad
          \text{for all $ u\in \calC^\infty_\textup{c} (\sfG)$}.
      \]
\item The system is smooth in the sense that for every $u \in \calC^\infty_\textup{c} (\sfG)$ the 
      map
      \[
          M \to \C, \:  x \mapsto \int\limits_{\sfG (x,-) } \! u (h) \, d\lambda^x (h) 
      \]
      is smooth. 
\end{enumerate}
Let us fix a smooth left Haar measure $(\lambda^x)_{x\in M}$ on $\sfG$.
Given an open set $U \subset M/\sfG$ we first put  
\begin{equation}
\label{eq:DefPreImgsOpen}
  U_0 := \pi^{-1}(U), \quad U_1 := s^{-1} (U_0) \subset \sfG_1
  \quad\text{and} \quad
  U_{k+1} := \bigcap_{i=1}^k \sigma_i^{-1} ( U_k) \subset \sfG_{k+1}
  \text{ for all $k\in \N^*$},
\end{equation}
where $\sigma_i :\sfG_{k+1} \rightarrow  \sfG_k$, $(g_1,\ldots , g_{k+1}) \mapsto 
      (g_1,\ldots , g_ig_{i+1}, \ldots  , g_k)$.
Then we define 
\begin{equation}\label{eq:calAU}
  \calA(U):= \big\{ f \in \calC^\infty \big( U_1 \big) 
             \mid \supp f \text{ is longitudinally compact} \big\} \ .
\end{equation}
Hereby, a subset $K \subset \sfG$ is called \emph{longitudinally compact}, 
if for every compact subset $C \subset M/\sfG$ the intersection 
$K \cap  s^{-1} \pi^{-1} (C)$ is compact.
Obviously, every $\calA(U)$ is a linear space, and the map which assigns 
to an open $U\subset M/\sfG$ the space $\calA(U)$ forms a sheaf on $M/\sfG$ which in 
the following will be denoted by $\calA$ or by $\calA_\sfG$ if we want to
emphasize the underlying groupoid. The section space 
$\calA(U)$ over $U\subset M/\sfG$ open becomes an associative algebra with the \emph{convolution product}
\begin{equation}
\label{eq:convolution-product}
 f_1 \ast f_2 \, (g) := \int\limits_{\sfG (t(g),-)} \! f_1 (h) \, f_2 (h^{-1}g) \, d\lambda^{t(g)} (h) \ ,
 \quad f_1,f_2 \in  \calA(U), \: g \in  \sfG \ .
\end{equation}
The convolution product is compatible with the restriction maps, hence
$\calA$ becomes a sheaf of algebras on $M/\sfG$. 

Let us assume from now on that the groupoid $\sfG$ is proper. Recall from 
\cite{PflPosTanGOSPLG} that then the orbit space $M/\sfG$ carries the structure
of a differentiable stratified space in a canonical way. The structure sheaf 
$\calC^\infty_{M/\sfG}$  coincides with the sheaf of continuous functions $\varphi :U \to \R$ 
with $U \subset M/\sfG$ open such that $\varphi \circ \pi \in \calC^\infty ( U_1 )$. 
Now observe that the action
\[ 
  \calC^\infty_{M/\sfG} (U) \times \calA (U) \rightarrow  \calA (U), \: 
  (\varphi , f ) \mapsto \varphi f := \Big( U_1 \ni g \mapsto 
  \varphi \big(\pi s(g)\big) f (g)  \in \R \Big)
\]
commutes with the convolution product, and turns  $\calA$ into  a 
$\calC^\infty_{M/\sfG}$-module sheaf. 

\begin{propdef}\label{propdefn:sheaf}
  Given a proper Lie groupoid $\sfG \rightrightarrows M$, the associated sheaf 
  $\calA$ is a fine sheaf of algebras over the orbit space $M/\sfG$ which in 
  addition carries the structure of a $\calC^\infty_{M/\sfG}$-module sheaf. 
  The space $\calA_\textup{c}(M/\sfG)$ of global sections of $\calA$ with compact 
  support coincides with the \emph{smooth convolution algebra} of $\sfG$. 
  We  call $\calA$ the \emph{convolution sheaf} of $\sfG$.
\end{propdef}

For later purposes, we equip the spaces $\calA (U)$ with a locally convex topology 
and a convex bornology. To this end, observe first that for every longitudinally 
compact subset $K\subset U_1$ the space   
\[
  \calA (M/\sfG;K) := \big\{ f \in \calC^\infty (\sfG) \mid \supp f \subset K \big\}
\]
inherits from $\calC^\infty (\sfG)$ the structure of a Fr\'echet space. 
Moreover, since $\calC^\infty (\sfG)$ is nuclear, $\calA (M/\sfG;K)$ has to be nuclear
as well by \cite[Prop.~50.1]{TreTVSDK}.
By separability of $U$ there exists a (countable) exhaustion of $U_1$ by longitudinally 
compact sets, i.e.~a family $(K_n)_{n\in \N}$ of longitudinally compact subset of $U_1$ such 
that $K_n \subset K_{n+1}^\circ$ for all $n\in \N$,
and such that $\bigcup_{n\in \N} K_n = U_1$. The space $\calA (U)$ can then be identified with the inductive limit 
of the strict inductive system  of nuclear Fr\'echet spaces $\big( \calA (M/\sfG;K_n) \big)_{n\in \N}$. 
It is straightforward to check that 
the resulting inductive limit topology on  $\calA (U)$ does not depend on the particular choice of 
the exhaustion $(K_n)_{n\in \N}$. Thus, $\calA (U)$ becomes a nuclear LF-space, where nuclearity 
follows from \cite[Prop.~50.1]{TreTVSDK}. As an LF-space, $\calA (U)$ carries a natural bornology given by 
the von Neumann bounded sets, i.e.~by the sets 
$S \subset \calA (U)$ which are absorbed by each neighborhood of $0$. 
In other words, a subset $S \subset \calA (U)$ is bounded if all $f\in S$ are 
supported in a fixed longitudinally compact subset $K\subset U_1$, and if the 
set of functions $D(S)$ is uniformly bounded for every compactly supported 
differential operator $D$ on $U_1$. 

The bornological point of view is particularly convenient when considering tensor 
products. In particular one has the following fundamental property.
\begin{proposition}
\label{Prop:ProdConvSheaves}
 Let $\sfG \rightrightarrows M$, and  $\sfH\rightrightarrows N$ be proper Lie groupoids.
 Denote by $M/\sfG$ and $N/\sfH$ their respective orbit spaces. 
 Then $M/\sfG\times N/\sfH$ is diffeomorphic as a differentiable stratified space to the orbit 
 space of the product groupoid $\sfG \times \sfH \rightrightarrows M\times N$. 
 Moreover, there is a natural  isomorphism   
\begin{equation}
  \label{eq:tprodconvalg}
  \calA_\sfG (U) \hat{\otimes} \calA_\sfH (V) \cong \calA_{\sfG\times \sfH} (U \times V)
\end{equation}
 for any two open sets $U \subset M/\sfG$ and $V\subset N/\sfH$. 
\end{proposition}
\begin{proof}
  The first claim is a consequence of the fact that  two elements
  $(x,y), (x',y') \in M\times N$ lie in the same $(\sfG \times \sfH)$-orbit
  if and only if $x$ and $x'$ lie in the same $\sfG$-orbit and
  $y$ and $y'$ lie in the same $\sfH$-orbit.
  Let us prove the second claim. 
  Let $(K_n)_{n\in \N}$ be an exhaustion of $U_1 := s_\sfG^{-1} \pi_\sfG^{-1} (U) $ by longitudinally compact subsets and $(L_m)_{m\in \N}$
  an exhaustion of $V_1:= s_\sfH^{-1} \pi_\sfH^{-1} (V)$ by such sets. Since $\calA_\sfG (U)$ coincides with the inductive limit
  $\colim\limits_{n\in \N}\calA_\sfG (M/\sfG;K_n)$ and $\calA_\sfH (V)$  with
  $\colim\limits_{m\in \N}\calA_\sfH (N/\sfH;L_m)$, \cite[Cor.~2.30]{MeyACH} entails that 
  \begin{equation}
    \label{eq:IndLimitTensorProduct}
     \calA_\sfG(U) \hat{\otimes} \calA_\sfH (V) \cong 
     \colim\limits_{n\in \N} \calA_\sfG (M/\sfG;K_n) \hat{\otimes} \calA_\sfH (N/\sfH;L_n) .
  \end{equation}
  Now observe that $\calA_\sfG (M/\sfG;K_n) \hat{\otimes} \calA_\sfH (N/\sfH;L_m) \cong 
  \calA_{\sfG\times \sfH}(M/\sfG\times N/\sfH; K_n \times L_m)$ by \cite[Prop.~51.6]{TreTVSDK}, 
  and that $( K_n \times L_n)_{n\in \N}$  is an exhaustion of $U \times V$ by longitudinally compact subsets. 
  Together with Eq.~\eqref{eq:IndLimitTensorProduct} this proves the claim.
\end{proof}

%
%
\section{Localization of the Hochschild chain complex}
\label{Sec:LHCH}
In this section, we apply the localization method in  Hochschild  
homology theory, partially following \cite{BraPflHAWFSS}, to the Hochschild chain complex
of the convolution algebra. 

\subsection{Sheaves of bornological algebras over a differentiable space}
We start with a  (reduced separated second countable) differentiable space
$(X, \calC^\infty)$ 
and assume that $\calA$ is a sheaf of $\R$-algebras on $X$. We will denote by 
$A=\calA (X)$ its space of global sections. We assume further that $\calA$ is a 
$\calC^\infty_X$-module sheaf and that every section space $\calA (U)$
with $U\subset X$ open carries the structure of a nuclear LF-space such that 
each of the restriction maps $\calA(U) \to  \calA (V)$ is continuous and 
multiplication in $\calA (U)$ is separately continuous.
Finally, it is assumed that the action $\calC^\infty (U) \times \calA (U) \to \calA(U)$
is continuous.

As a consequence of our assumptions, each of the spaces $\calA(U)$ carries a natural 
bornology, namely the one consisting of all von Neumann bounded subsets, i.e.~of all subsets 
$B \subset \calA(U)$ which are absorbed by every neighborhood of the origin.
Moreover, by \cite[Lemma 1.30]{MeyLACH}, separate continuity of multiplication in $\calA(U)$
entails that the product map is a jointly bounded map, hence induces a bounded map 
$\calA(U) \hat{\otimes} \calA(U) \to \calA (U)$ on the complete (projective) bornological tensor product 
of $\calA(U)$ with itself.  

\begin{remark}
   \begin{enumerate}
   \item 
     We refer to Appendix \ref{AppCyHomBornAlg} for basic definitions and to 
     \cite{MeyLACH}  for further details on 
     bornological vector spaces, their (complete projective) tensor products,
     and the use of these concepts within cyclic homology theory. 
     We always assume the bornologies in this paper to be convex vector bornologies.
   \item
     In this paper, we  will often silently make use of the fact, that for two nuclear 
     LF-spaces 
     $V_1$ and $V_2$ their complete bornological tensor product $V_1 \hat{\otimes} V_2$ 
     naturally concides (up to natural equivalence) with the complete  inductive tensor product 
     $V_1 \hat{\otimes}_\iota V_2 $ endowed with the bornology of von Neumann bounded sets.
     Moreover, $V_1 \hat{\otimes}_\iota V_2 $ is again a nuclear  LF-space. We refer to \cite[A.1.4]{MeyACH}
     for a proof of these propositions. 
     Note that for Fr\'echet spaces the  projective and inductive topological 
     tensor product coincide. 
   \end{enumerate}
\end{remark}

\begin{definition}
  A sheaf of algebras $\calA$ defined over a differentiable space 
  $(X,\calC_X^\infty)$ such that the above assumptions are fulfilled will be called 
  a \emph{sheaf of bornological algebras over} $(X,\calC_X^\infty)$.  
  If all $\calA (U)$  are unital and the restriction maps 
  $\calA (U) \rightarrow \calA (V)$ are unital homomorphisms, we say that $\calA$ is a 
  \emph{sheaf of unital bornological algebras} or just that  $\calA (U)$ is \emph{unital}.
  If every section space $\calA (U)$  is an H-unital algebra, we
  call $\calA$ a \emph{sheaf of H-unital bornological algebras} or briefly  \emph{H-unital}.
  Finally, we call  $\calA$ an \emph{admissible sheaf of bornological algebras}
  if $\calA$ is H-unital and if for each $k\in \N^*$ the 
  presheaf assigning to an open $U\subset X$ the $k$-times complete bornological 
  tensor product $\calA(U)^{\hat{\otimes}k}$ is even a sheaf on $X$.
\end{definition}

\begin{example}
  \begin{enumerate}
  \item The structure sheaf $\calC_X^\infty$ of a differentiable space 
        $(X,\calC_X^\infty)$ is an example of an admissible sheaf of unital 
        bornological algebras over $(X,\calC_X^\infty)$. 
  \item Given a proper Lie groupoid $\sfG$, the convolution sheaf $\calA$ is an admissible
        sheaf of bornological algebras over the orbit space $(X,\calC_X^\infty)$ of the 
        groupoid. This follows by construction of $\calA$, 
        Prop.~\ref{Prop:ProdConvSheaves} and \cite[Prop.~2]{CraMoeFGCH}, 
        which entails H-unitality of each of the section spaces $\calA(U)$.
  \end{enumerate}
\end{example}

\subsection{The Hochschild homology sheaf}

Assume that $\calA$ is a sheaf of bornological algebras over 
the differentiable space $(X,\calC_X^\infty)$. We will construct the Hochschild homology sheaf
$\HHsheaf_\bullet (\calA)$ associated to $\calA$  as a generalization of Hochschild homology for algebras;
see \cite{LodCH} for the latter and Appendix \ref{AppCyHomBornAlg} for basic definitions and notation used.

For each $k\in \N^*$ let $\scrC_k (\calA)$ denote the presheaf on $X$ which assigns to an open $U\subset X$ 
the $(k+)$-times complete bornological tensor product $\calA(U)^{\hat{\otimes}(k+1)}$. 
Note that in general,  $\scrC_k (\calA)$ is not a sheaf. We denote by 
$\sheafC_k (\calA)$ the sheafification of $\scrC_k (\calA)$. Observe that 
for $V\subset U \subset X$ open the Hochschild boundary 
\[
  b : \scrC_k (\calA) (U)  \to \scrC_{k-1} (\calA) (U)
\]
commutes with the restriction maps $r^U_V: \scrC_k (\calA) (U) \to \scrC_k (\calA) (V)$,
hence we obtain a complex of presheaves $\big( \scrC_\bullet (\calA) , b \big)$
and by the universal property of the sheafification a sheaf complex 
$\big( \sheafC_\bullet (\calA) , b \big)$.
The Hochschild homology sheaf $\HHsheaf_\bullet (\calA)$ is now defined as the 
homology sheaf of $\big( \sheafC_\bullet (\calA) , b \big)$ that means
\[
  \HHsheaf_k (\calA) := 
  \ker \big( b : \sheafC_k (\calA) \to \sheafC_{k-1} (\calA) \big) / 
  \im \big( b : \sheafC_{k+1} (\calA) \to \sheafC_k (\calA) \big) .
\]
By construction,  the stalk $\HHsheaf_k (\calA)_\smorbO$, $ \smorbO \in X$
coincides with the $k$-th Hochschild homology $HH_k (\calA_\smorbO )$ of the stalk
$\calA_\smorbO$. On the other hand, $HH_k (\calA (X) )$ need in general not coincide
with the space  $\HHsheaf_k (\calA) (X)$ of global sections of the $k$-th 
Hochschild homology sheaf. The main goal of this section is to prove 
the following result which is crucial for our study of the Hochschild homology 
of the convolution algebra of a proper Lie groupoid, but also might be intersting by its own.  
Its proof will cover the remainder of Section \ref{Sec:LHCH}.

\begin{theorem}
\label{thm:hochschild-homology-global-sections}
  Let $\calA$ be the convolution sheaf of a proper Lie groupoid $\sfG$. 
  Then the natural map in Hochschild homology
  \[
    HH_\bullet \big( \calA (X)\big) \to \HHsheaf_\bullet (\calA) (X)
    = \Gamma \big(X,  \HHsheaf_\bullet (\calA) \big) 
  \]
  is an isomorphism.
\end{theorem}
Before we can spell out the proof we need several auxiliary tools and results.
\subsection{The localization homotopies}\label{sec:localization-homotopies}

Throughout this paragraph we assume that $\calA(X)$ is an admissible sheaf of bornological 
algebras over the differentiable space $(X,\calC^\infty_X)$. 

To construct the localization morphisms, observe that the complex $C_\bullet (A)$
inherits from $A=\calA(X)$ the structure of a $\calC^\infty (X)$-module. More precisely, 
the corresponding action is given by 
\begin{equation}
  \label{eq:DefActionSmoothFctChains1}
  \calC^\infty (X) \times C_k(A) \rightarrow C_k(A), \quad 
  (\varphi , a_0 \otimes \ldots \otimes a_k) \mapsto ( \varphi a_0 ) \otimes a_1 \otimes \ldots \otimes a_k \: .
\end{equation}
By definition, it is immediate that the $\calC^\infty (X)$-action commutes with the 
operators $b$ and  $b'$ hence induces a chain map
$ \calC^\infty (X) \times C_\bullet (A) \to C_\bullet (A)$. 
In a similar fashion we define an action of 
$\calC^\infty (X^{k+1}) \cong \big( \calC^\infty (X)\big)^{\hat{\otimes} (k+1)}$ on $C_k(A)$ by
\begin{equation}
  \label{eq:DefActionSmoothFctChains2}
  (\varphi_0 \otimes \ldots \otimes \varphi_k , a_0 \otimes \ldots \otimes a_k) \mapsto 
  ( \varphi_0 a_0 ) \otimes \ldots \otimes (\varphi_k a_k ).
\end{equation}
This allows us to speak of the \emph{support} of a chain $c \in C_k(A)$. It is defined 
as the complement of the largest open subset $U$ in $ X^{k+1}$ such that 
$\varphi \cdot c = 0$ for all $\varphi \in \calC^\infty (X)$ with 
$\supp \varphi \subset U$. 

Next choose a  metric $d: X \times X \to \R$  such that the 
function $d^2$ lies in $\calC^\infty (X\times X)$.
Such a metric exists by Corollary \ref{Cor:SmoothMetricAppendix}.
Then fix a smooth function 
$\varrho : \R \rightarrow [0,1]$ which has support in $(-\infty , \frac 34]$ and  
satisfies $\varrho (r) =1$ for $r\leq \frac 12$. For $\varepsilon >0$ we denote by 
$\varrho_\varepsilon$ the rescaled function $r \mapsto \varrho (\frac{s}{\varepsilon^2} )$. 
Now define functions $\Psi_{k,i,\varepsilon} \in \calC^\infty (X^{k+1})$ for $k\in \N$ and 
$i= 0, \ldots, k$ by
\begin{equation}
  \Psi_{k,i,\varepsilon} (x_0, \ldots, x_k) =  \prod_{j=0}^{i-1} 
  \varrho_\varepsilon \big( d^2(  x_j  , x_{j+1}) \big),
  \quad \text{where $x_0,\ldots , x_k \in X$ and $x_{k+1}:= x_0$} \: . 
\end{equation} 
Moreover, put $\Psi_{k,\varepsilon} := \Psi_{k,k+1,\varepsilon}$. Using the 
$\calC^\infty (X^{k+1})$-action on $C_k(A)$ we obtain for each $\varepsilon >0$ a graded 
map of degree $0$
\[
  \Psi_\varepsilon : C_\bullet (A) \to C_\bullet(A), \: C_k(A) \ni c \mapsto 
  \Psi_{k,\varepsilon} c \: .
\]
One immediately checks that $\Psi_\varepsilon$ commutes with the face maps $b_i$ and the cyclic operator $t_k$.
Hence, $\Psi_\varepsilon$ is a chain map. One  even has more. 

\begin{lemma}\label{lem:localization-homotopies}
  Let $\calA$ be an admissible sheaf of bornological algebras over the differentiable 
  space $(X,\calC^\infty)$, and put $A:= \calA(X)$. Let $d$ be a metric on $X$ such that
  $d^2$ is smooth and fix a smooth map $\varrho:\R\to [0,1]$ with support in $(-\infty,\frac 34]$
  such that  $\varrho|_{(\infty,\frac12]}=1$. 
  Then, for each $\varepsilon >0$, the chain map $\Psi_\varepsilon: C_\bullet (A) \to C_\bullet(A)$ is 
  homotopic to the identity morphism on $C_\bullet (A)$.
\end{lemma}
\begin{proof}
 Let us first consider the case, where $\calA$ is a sheaf of unital algebras. The
 Hochschild chain complex then is a simplicial module with face maps $b_i$ and
 the degeneracy maps
 \[
   s_{k,i} : C_k (A) \rightarrow C_{k+1} (A), \; a_0 \otimes \ldots \otimes a_k \mapsto 
   a_0 \otimes \ldots \otimes a_i \otimes 1 \otimes a_{i+1} \otimes \ldots \otimes a_k  
   \ ,
 \]
 where $k\in \N$, $i = 0, \ldots , k$.
 Define $\calC^\infty (X)$-module maps $\eta_{k,i,\varepsilon} : C_k (A) \rightarrow C_{k+1} (A)$ 
 for $k\in \N$, $i=1,\cdots,k+2$ and $\varepsilon >0$ by 
  \begin{equation}
    \eta_{k,i,\varepsilon} (c) :=
    \begin{cases}
      \Psi_{k+1,i,\varepsilon}\cdot ( s_{k,i-1} c  )& \text{for $i\leq k+1$},\\
      0 &\text{for $i=k+2$}.
    \end{cases}
  \end{equation}
  Moreover, put $C_{-1} (A) := \{ 0\}$ and let  $\eta_{-1,1,\varepsilon}: C_{-1} (A) \to C_0 (A)$
  be the $0$-map. For $k\geq 1$ and $i = 2,\cdots, k$ one then computes 
\begin{equation*}
  \begin{split}
    ( b \eta_{k,i,\varepsilon} +  \eta_{k-1,i,\varepsilon} b) c  = \, &
    (-1)^{i-1} \Psi_{k,i-1,\varepsilon}  c \, +  \Psi_{k,i-1,\varepsilon} \sum_{j=0}^{i-2} \,
       (-1)^j \,  s_{k-1,i-2} b_{k,j} c \, + \\
       & + (-1)^i \Psi_{k,i,\varepsilon} c \, + \Psi_{k,i,\varepsilon} \sum_{j=0}^{i-1} \,
       (-1)^j \, s_{k-1,i-1} b_{k,j} c  \ .
  \end{split}
\end{equation*}
For the case $i=1$ one obtains
\begin{equation*} 
  ( b \eta_{k,1,\varepsilon} + \eta_{k-1,1,\varepsilon} b) c
     = c  \, - \, \Psi_{k,1,\varepsilon} c \, + \, \Psi_{k,1,\varepsilon} s_{k-1,0} b_{k,0} c \ ,
\end{equation*}
and for $i=k+1$  
\begin{equation*}
  ( b \eta_{k,k+1,\varepsilon} +  \eta_{k-1,k+1,\varepsilon} b ) c  = 
   \Psi_{k,k,\varepsilon} (-1)^k c \, + \, \Psi_{k,k,\varepsilon} \sum_{j=0}^{k-1}
      \, (-1)^j \, s_{k-1,k-1} b_{k,j} c  \, + \,
      (-1)^{k+1} \Psi_{k,\varepsilon} \, c  . 
\end{equation*}
Finally, one checks for $k=0$ and $i=1$
\begin{equation*}
  ( b \eta_{0,1,\varepsilon} + \eta_{-1,1,\varepsilon} b) c = b \eta_{0,1,\varepsilon} c = 0 \ .
\end{equation*}
These formulas immediately entail that the maps
  \begin{displaymath}
    \begin{split}
      H_{k,\varepsilon} & = \sum_{i=1}^{k+1} \, (-1)^{i+1} \, \eta_{k,i,\varepsilon} : 
      C_k (A)\rightarrow C_{k+1} (A) 
    \end{split}
  \end{displaymath}
  form a homotopy between the identity  and the localization
  morphism $\Psi_\varepsilon$. More precisely,
  \begin{align}
  \label{EqHomHom}
      \big( b H_{k,\varepsilon} + H_{k-1,\varepsilon} b \big) c & = c - \Psi_\varepsilon \, c
      \quad \text{for all } k\in \N \text{ and } c\in C_k (A) \ .
  \end{align}     
  This finishes the proof of the claim in the unital case.

  Now let us consider the general case, where   $\calA$ is assumed to be a sheaf of H-unital but not necessarily 
  unital algebras. Consider the  direct sum of sheaves $\calA \oplus \calC^\infty_X$,
  denote it by $\wcalA$, and put $\wA :=\wcalA(X)$. 
  We turn $\wcalA$ into a sheaf of unital bornological algebras by defining 
  the product of $(f_1,h_1), (f_2,h_2) \in \wcalA(U)$ as 
  \begin{equation}
  \label{eq:product}
    (f_1,h_1)\cdot (f_2,h_2) := (h_1f_1 + h_2f_1 + f_1\, f_2 , h_1\, h_2) .   
  \end{equation}
  One obtains a split short exact sequence in the category of bornological algebras  
  \begin{displaymath}
  \xymatrix{
    0 \ar[r]  & A  \ar[r]  & \wA  \ar[r]^q  & \calC^\infty (X) \ar@<1ex>@{-->}[l]^i \ar[r] & 0 
    } \ .
  \end{displaymath}
  This gives rise to a diagram of chain complexes and chain maps
  \begin{equation}
  \label{diag}
  \xymatrix{
    0 \ar[r] & \ker_\bullet q_* \ar@{^{(}->}[r] \ar@<1ex>@{-->}[d]^{\kappa} &  C_\bullet (\wA) \ar[r]^{q_*} & C_\bullet (\calC^\infty (X)) \ar@<1ex>@{-->}[l]^{i_*} \ar[r] & 0\\
    &  C_\bullet (A) \ar[u]^{\iota} ,
  }
  \end{equation}
  where the row is split exact, and $\iota$ denotes the canonical embedding. 
  Since $A$ is H-unital, $\iota$ is a quasi-isomorphsm. Because the chain complexes 
  $\ker_\bullet q_*$ and $C_\bullet (A)$ are bounded from below, there exists a chain map $\kappa$ 
  which is left inverse to $\iota$.
  Note that the components $\kappa_k$  need not be bounded maps between bornological spaces. 
  By construction, $\Psi_\varepsilon$ acts on each of the chain complexes
  within the diagram, and all chain maps (besides possibly $\kappa$) commute with this action. 
  By the first part of the proof we have an algebraic homotopy 
  $H : C_\bullet (\wA) \to C_{\bullet+1} (\wA)$ such that 
  \[
      \id - \Psi_\varepsilon = b H + H b \: .
  \] 
  Define  $F : C_\bullet (A) \to C_{\bullet+1} (A)$ by $F := \kappa (\id -i_*q_*) H \iota$.
  Note that $F$ is well-defined indeed, since $q_*(\id -i_*q_*) =0$. 
  Now compute for $c \in C_k (A)$ 
  \[
    (bF +Fb)c = \kappa (\id -i_*q_*) (bH + H b) \iota c =
    \kappa (\id -i_*q_*) (\iota c -\Psi_\varepsilon \iota c) = c -\Psi_\varepsilon \, c \: .
  \]
  Hence $F$ is a homotopy between the identity and $\Psi_\varepsilon$ and the claim is proved.
\end{proof} 

\begin{lemma}
  \label{lem:homotopy-subordinate-partition-unity}
   Let $\calA$ be an admissible sheaf of bornological algebras over the differentiable 
   space $(X,\calC^\infty)$, put $A:= \calA(X)$, and let the metric $d$ and the cut-off function
   $\varrho$ as in the preceding lemma. Assume that 
   $(\varphi_l)_{l\in \N}$ is a smooth locally finite partition of unity and that $(\varepsilon_l)_{l\in \N}$ 
   a sequence of positive real numbers. Then 
   \begin{equation}
     \label{eq:def-localization-morphism}
     \Psi: C_\bullet (A) \to C_\bullet (A),\enspace
     C_k (A)\ni c \mapsto \sum_{l\in \N} \varphi_l\Psi_{\varepsilon_l}c \ .   
   \end{equation}
   is a chain map  and there exists a homotopy between the identity
   on $C_\bullet (A)$ and $\Psi$. 
\end{lemma}

\begin{proof}
  Recall that the action of $\calC^\infty(X)$ commutes with the Hochschild boundary and that each
  $\Psi_{\varepsilon_l}$ is a chain map. Since $(\varphi_l)_{l\in \N}$ is a locally finite smooth partition
  of unity, $\Psi$ then has to be a chain map by construction.  
  
  Now assume that  $\calA$ is a sheaf of unital algebras. Let
  $H_{\bullet,\varepsilon_l}:C_\bullet (A)\to C_{\bullet+1}$ be the homotopy from the preceding lemma which
  fulfills Equation \eqref{EqHomHom} with $\varepsilon = \varepsilon_l$. 
  For all $k\in \N$ let $H_k$ be the map
  \[ 
     H_k :  C_k (A) \to  C_{k+1} (A), \enspace c \mapsto \sum_{l\in \N} H_{k,\varepsilon_l} \varphi_l \, c \ . 
  \]
  Then
  \begin{align}
  \label{EqHomHom2}
    \big( b H_k + H_{k-1} b \big) c & =
    \sum_{li\in \N} \left( \varphi_l c -  \Psi_{\varepsilon_l} \varphi_l \, c \right) = c - \Psi \, c
    \quad \text{for all } k\in \N \text{ and } c\in C_k (A) \ .
  \end{align}   
  Hence $H$ is a homotopy  between the identity  and  $\Psi$ which proves the claim in the unital case. 

  In the non-unital case define the unitalizations $\wcalA$ and $\wA$ as before and let $q_*$, $i_*$, $\iota$,
  $\kappa$ denote the chain maps as in  Diagram \eqref{diag}. 
  Let $H : C_\bullet (\wA) \to C_{\bullet+1} (\wA)$ be the algebraic homotopy constructed
  for the unital case. In particular this means that 
  \[
      \id - \Psi = b H + H b \: .
  \] 
  Defining $F : C_\bullet (A) \to C_{\bullet+1} (A)$ by $F := \kappa (\id -i_*q_*) H \iota$
  then gives a homotopy between the identity on $C_\bullet (A)$ and $\Psi$. 
\end{proof}

\begin{lemma}
\label{lem:cycle-not-meeting-diagonal}
   Let $\calA$ be an admissible sheaf of bornological algebras over the differentiable 
   space $(X,\calC^\infty)$, put $A:= \calA(X)$ and let $c \in C_k(A)$ be  a Hochschild cycle.
  If the support of $c$ does not meet the diagonal,  then $c$ is a Hochschild boundary. 
\end{lemma}
\begin{proof}
  Assume that the support of the Hochschild cycle $c$ does not meet the diagonal and
  let $U = X^{k+1} \setminus \supp c$.   
  Then $U$ is an open neighborhood of the diagonal. By 
  Corollary \ref{Cor:SmoothMetricAppendix}  there exists a complete 
  metric $d:X \times X \to \R$ such that $d^2 \in \calC^\infty (X \times X)$. 
  Choose a compact exhaustion $(K_n)_{n\in \N}$ of $X$ which means that each $K_n$
  is compact, $K_n \subset K_{n+1}^\circ$ for all $n\in \N$ and
  $\bigcup_{n\in \N}K_n =X$. For each $n\in \N$ there then exists an $\varepsilon_n >0$
  such that all $(x_0,\ldots,x_k) \in K_n^{k+1}$ are in $U$ whenever
  $d(x_j,x_{j+1}) < \varepsilon_n$ for $j=0, \ldots , k$ and $x_{k+1}:=x_0$.
  Choose a locally finite smooth partition of unity $(\varphi_l)_{l\in \N}$
  subordinate to the open covering $(K_n^\circ)_{n\in \N}$ and let
  $\Psi: C_\bullet (A) \to C_\bullet (A)$ be the associated chain map defined by 
  \eqref{eq:def-localization-morphism}. According to 
  Lemma \ref{lem:homotopy-subordinate-partition-unity} there then exists a chain
  homotopy $H$ between the identity on $C_\bullet (A)$ and $\Psi$. Since the support
  of $c$ does not meet $U$ one obtains  
  \[
    c = c - \Psi_\varepsilon c = bH (c) \ , 
  \]
  so $c$ is a Hochschild boundary indeed. 
\end{proof}

\begin{proposition}
\label{prop:isomorphism-global-section-space}
  Assume to be given a proper Lie groupoid with orbit space $X$ and 
  convolution sheaf $\calA$. Let $A = \calA (X)$  and $\sheafC_\bullet (\calA )$ be the
  sheaf complex of Hochschild chains. Denote for each $\smorbO\in X$ and each chain
  $c \in C_\bullet \big(\calA (U) \big)$ defined on a  neighborhood $U\subset X$ of
  $\smorbO$ by
  $[c]_{\smorbO}$ the germ of $c$ at $\smorbO$ that is the image of $c$ in the stalk
  $\sheafC_{\bullet,\smorbO} (\calA) = \colim\limits_{V \in \mathcal{N}(\smorbO)} C_\bullet (\calA (V))$,
  where $\mathcal{N} (\smorbO)$ denotes the filter basis of open neighborhoods of $\smorbO$.
  Then the chain map 
  \[
     \eta : C_\bullet ( A ) \to \Gamma \big( X , \sheafC_\bullet (\calA)\big), 
     \enspace c \mapsto \big( [c]_{\smorbO}\big)_{\smorbO\in X} 
  \]
  is a quasi-isomorphism.
\end{proposition}

\begin{proof}
  Consider a section $s \in \Gamma \big( X , \sheafC_k (\calA)\big)$.
  Then there exists a (countable) open covering $(U_i)_{i\in I}$ of the orbit space $X$
  and a family $(c_i)_{i\in I}$ of $k$-chains $c_i \in  C_k \big(\calA (U_i) \big)$ such
  that $[c_i]_{\smorbO} = s(\smorbO)$ for all $i\in I$ and $\smorbO \in U_i$.
  After possibly passing to a finer (still countable) and locally finite covering one
  can assume that there exists a partition of unity $(\varphi_i)_{i\in I}$ 
  by functions $\varphi_i \in \calC^\infty (X)$ such that $\supp \varphi_i \subset \subset U_i$ 
  for all $i\in I$. If $s$ is a cycle, then we can achieve after possible passing to an even finer
  locally finite covering that each $c_i$ is a Hochschild cycle as well. 
  Choose a metric $d : X \times X \to \R$ such that $d^2 \in \calC^\infty (X\times X)$.
  For each $i$ there then exists $\varepsilon_i >0$ such that the space of all
  $\smorbO\in X$ with $d (\smorbO, \supp \varphi_i) \leq (k+1) \varepsilon_i$
  is a compact subset of $U_i$. 
  The chain $\Psi_{\varepsilon_i} (\varphi_i c_i)$ then has compact support in
  $U_i^{k+1}$. Extend it by $0$ to a smooth function on $X^{k+1}$ and denote the
  thus obtained $k$-chain also by  $\Psi_{\varepsilon_i} (\varphi_i c_i)$. Now put 
  \begin{equation}
    \label{eq:definition-lifting-chain}
     c := \sum_{i\in I} \Psi_{\varepsilon_i} (\varphi_i c_i) \ .
  \end{equation}
  Then $c \in C_k ( A) $  is well-defined since the sum in the definition of 
  $c$ is locally finite. For every $\smorbO \in X$ now choose an open neigborhood
  $W_\smorbO$ meeting only finitely many of the elements of the covering
  $(U_i)_{i\in I}$.  Denote by $I_\smorbO$ the set of indices $i\in I$ such that
  $U_i \cap W_\smorbO\neq \emptyset$. Then each $I_\smorbO$ is finite. 
  Next let $H_i : C_\bullet \big( \calA (U_i)\big) \to C_{\bullet+1} \big( \calA (U_i)\big)$ 
  be the homotopy operator constructed in the proof of Lemma
  \ref{lem:localization-homotopies} such that 
  \[
    bH_i +H_i b = \id - \Psi_{\varepsilon_i}  \: .
  \]
  Let $e_i = H_i (\varphi_i c_i)$ for $i \in I_\smorbO$ and put
  $e_\smorbO = \sum_{i\in I_\smorbO} e_i|_{W_\smorbO^{k+2}}$.
  Then $e_\smorbO \in C_{k+1} \big( \calA (W_\smorbO)\big)$. Now compute for
  $\smorbQ \in W_\smorbO$
  \begin{equation*}
    \begin{split}
      s(\smorbQ) - [c]_\smorbQ \, & = \sum_{i\in I_\smorbO}
      \left[\varphi_i c_i \right](\smorbQ)-\left[\Psi_{\varepsilon_i}(\varphi_i c_i)\right]_\smorbQ = 
      \sum_{i\in I_\smorbO}\left[ b e_i \right]_\smorbQ+\left[ H_i(\varphi_i b  c_i )\right]_\smorbQ=\\ 
      & = \left[ b e_\smorbO \right]_\smorbQ+\sum_{i\in I_\smorbO}\left[ H_i (\varphi_i b c_i)\right]_\smorbQ \ .
    \end{split}
  \end{equation*}
  Hence one obtains, whenever $s$ is a cycle,
  \[
    s(\smorbQ) - [c]_\smorbQ  = \left[ b e_\smorbO \right]_\smorbQ  \quad \text{for all }
    \smorbO \in X, \: \smorbQ \in W_\smorbO\ .
  \]
  This means that $s$  and $\eta (c)$  define the same homology class. 
  So the induced morphism between homologies 
  $H_\bullet\eta : HH_\bullet(A)\to H_\bullet\big(\Gamma\big(X,\sheafC_\bullet (\calA)\big)\big)$ is surjective. 
  It remains to show  that $H_\bullet\eta$ is injective. To this end assume that $e \in C_k (A)$  
  is a cycle such that $H_\bullet\eta (e) = 0$. Then $\eta (e) = bs $ for some 
  $s \in  \Gamma \big( X , \sheafC_{k+1} (\calA)\big)$. As before, associate to $s$ a
  sufficiently fine locally finite open cover $(U_i)_{i\in I}$ together with a
  subordinate smooth partition of unity $(\varphi_i)_{i\in I}$ and
  $c_i \in  C_{k+1}(\calA(U_i))$ such that $[c_i]_\smorbO = s(\smorbO)$ for all
  $\smorbO \in U_i$. Let $W_\smorbO$ and $I_\smorbO$ also be as above. Define
  $c \in C_{k+1}(A)$ by Eq.~\eqref{eq:definition-lifting-chain}. 
  Now compute for $\smorbQ\in W_\smorbO$
  \begin{equation*}
    \begin{split}
      [bc - e ]_\smorbQ \,&=
      \sum_{i\in I_\smorbO}\left[ b \Psi_{\varepsilon_i}(\varphi_i c_i)\right]_\smorbQ
      - [\varphi_i e]_\smorbQ 
      =\sum_{i\in I_\smorbO}\left[\Psi_{\varepsilon_i} (\varphi_i b c_i)  \right]_\smorbQ -
      [\varphi_i e]_\smorbQ = \\ & = 
      \sum_{i\in I_\smorbO}  \left[ \varphi_i b c_i \right]_\smorbQ - [\varphi_i e]_\smorbQ =
      \sum_{i\in I_\smorbO}  (\varphi_i b s)(\smorbQ) - (\varphi_i bs)(\smorbQ) = 0 \ .
    \end{split}
  \end{equation*}
  Therefore, $bc - e\in C_k(A)$ is a $k$-cycle such that its support does not meet the
  diagonal. By Lemma \ref{lem:cycle-not-meeting-diagonal},
  $bc - e$ is a boundary which means that the homology of $e$ is trivial.
  Hence $H_\bullet\eta$ is an isomorphism. 
\end{proof}

Now we have all the tools to verify our main localization result.

\begin{proof}[Proof of Theorem \ref{thm:hochschild-homology-global-sections}]
  First note that we can regard every chain complex of sheaves  $\calD_\bullet$ as a cochain complex 
  of sheaves under the duality $\calD^n := \calD_{-n}$ for all integers $n$. 
  We therefore have the hypercohomology 
  $\mathbb{H}_{n} (X, \calD_\bullet) := \mathbb{H}^{-n} (X, \calD^{\bullet})$;
  see \cite[Appendix]{WeiCHS}, where the  case of cochain complexes of sheaves not 
  necessarily bounded below as we have it here is considered. 
  Observe that $\big( \sheafC_\bullet (\calA) , b \big)$ and  
  $\big( \HHsheaf_\bullet (\calA),0\big)$ are quasi-isomorphic sheaf complexes,
  hence their hypercohomologies coincide. 
  Recall that for a cochain complex of fine sheaves $\calD^\bullet$
  \[
    \mathbb{H}^{n} (X, \calD^{\bullet}) = 
    H^{n} \big( \Gamma (X , \calD^{\bullet})\big) 
    \ .
  \]
  Since both $\sheafC_\bullet (\calA)$ and $\HHsheaf_\bullet (\calA)$ are 
  complexes of fine sheaves, these observations together with Proposition  \ref{prop:isomorphism-global-section-space}
  now entails for natural $n$
  \begin{equation*}
    \begin{split}
      HH_n \big( \calA (X)\big) \, & =  H_n \big( \Gamma (X , \sheafC_\bullet (\calA)\big) =  
    \mathbb{H}_{n} (X, \sheafC_\bullet (\calA)) = \\ & = \mathbb{H}_{n} (X,  \HHsheaf_{\bullet} (\calA) ) =
    H_n \big( \Gamma (X , \HHsheaf_{\bullet})\big) =
    \Gamma \big(X,  \HHsheaf_n (\calA) \big) \ .
    \end{split}
  \end{equation*}
  This is the claim. 
\end{proof}



\section{Computation at a stalk}
Recall that $\sfG\rightrightarrows M$ is a proper Lie groupoid, $X$ is its orbit space, and
$\calA_G$ is the convolution sheaf of $\sfG$ (Definition \ref{propdefn:sheaf}).
Given  an orbit $\orbO\in X$ of $\sfG$, we introduce in this section a linear
model of the groupoid around the stalk and use it in
Proposition \ref{prop:local-model} to construct a quasi-isomorphism
between the stalk complex $\sheafC_{\bullet,\smorbO}(\calA_\sfG)$ and the
corresponding of the linear model. We divide the construction into two steps.

\subsection{Reduction to the linear model}
Let us recall the linearization result for the groupoid $\sfG$ around an orbit $\orbO$.  
Let $N\orbO\to \orbO$ be the normal bundle of the closed submanifold $\orbO$ in $M$, 
and $\sfG_{|\orbO} \rightrightarrows \orbO$ be the restriction of the groupoid $\sfG$ to $\orbO$.
$\sfG_{|\orbO}$ acts on $N\orbO$ canonically.
And we use $\sfG_{|\orbO}\ltimes N\orbO\rightrightarrows N\orbO$ to denote the associated
transformation groupoid. As in Definition \ref{propdefn:sheaf}, let $\calA_{N\orbO}$ be the
sheaf of convolution algebras on $X_{N\orbO}=N\orbO/ \sfG_{|\orbO}$, the orbit space associated to
the groupoid $\sfG_{|\orbO}\ltimes N\orbO$. Accordingly, we can consider the presheaf of chain
complexes $\scrC_\bullet(\calA_{N\orbO})$ and  the associated sheaf complex
$\sheafC_\bullet(\calA_{N\orbO})$ as in Proposition \ref{prop:isomorphism-global-section-space}.
In what follows in this subsection, we will identify the stalk $\sheafC_{\bullet,\smorbO}(\calA_\sfG)$
with the linearized model $\sheafC_{\bullet,\smorbO}(\calA_{N\orbO})$, which is the stalk of the sheaf
$\sheafC_\bullet(\calA_{N\orbO})$ at the zero section of $N\orbO$. 

The main tool to identify the above two stalks is the linearization result of proper Lie groupoids developed by Weinstein \cite{WeiLRPG} and Zung \cite{ZunPGMMLAC} (See also \cite{CraStrLTPLG,PflPosTanGOSPLG,delHFerRMLG}). The particular approach we take below is from \cite{PflPosTanGOSPLG}.
Fix a transversely invariant metric $g$ on $M$. Given a function $\delta: \orbO\to \mathbb{R}_{>0}$, let $T^\delta_{\orbO, N\orbO}$ be the $\delta$-neighborhood of the zero section in $N\orbO$. According to
\cite[Theorem 4.1]{PflPosTanGOSPLG}, there exists a continuous map $\delta: \orbO\to \mathbb{R}_{>0}$ such that the exponential map $\exp_{|T^\delta_{\orbO, N\orbO}}: T^\delta_{\orbO, N\orbO}\to T^\delta_\orbO:=\exp(T^\delta_{\orbO, N\orbO})\subset M$ is a diffeomorphism. Furthermore, the exponential map $\exp_{|T^\delta_{\orbO, N\orbO}}$ lifts to an isomorphism $\Theta$ of the following groupoids
\begin{equation}\label{eq:tubular}
\Theta:\big(\sfG_{|\orbO}\ltimes N\orbO \big)_{|T^\delta_{\orbO, N\orbO}}\to \sfG_{|T^\delta_\orbO}. 
\end{equation}

\begin{lemma}\label{prop:stalk-linearization}
  For each orbit $\orbO \subset M$, the pullback map $\Theta^*$ defines a quasi-isomorphism
  $\Theta_{\bullet,\smorbO}$ from the stalk complex $\sheafC_{\bullet,\smorbO}(\calA_\sfG)$ to
  the stalk complex $\sheafC_{\bullet,\smorbO}(\calA_{N\orbO})$. 
\end{lemma}
\begin{proof}
  We explain how $\Theta_{\bullet,\smorbO}$ is defined on $\sheafC_{\bullet,\smorbO}(\calA_\sfG)$.
  Let  $[f_0\otimes \cdots \otimes f_k]\in \sheafC_{k,\smorbO}(\calA_\sfG)$ be a germ of a $k$-chain at
  $\orbO\in X$.  Let $U$ be a neighborhood of $\orbO$ in $X$ such that $f_0\otimes \cdots \otimes f_k$
  is a section of $\scrC_k(\calA(U))$ which is mapped to $[f_0\otimes \cdots \otimes f_k]$ in
  the stalk complex $\sheafC_{\bullet,\smorbO} (\calA_\sfG)$ under the canonical map $\eta$ from
  Proposition \ref{prop:isomorphism-global-section-space}.
  By (\ref{eq:calAU}), the support of each of the maps $f_0, \cdots, f_k$ is longitudinally compact.
  In particular,  $\supp(f_i)\cap s^{-1}(\orbO)$ ($i=0,\cdots, k$) is compact. Therefore,
  \[ s\big(\supp(f_i)\cap s^{-1}(\orbO)\big)=t\big(\supp(f_i)\cap s^{-1}(\orbO)\big)\]
  and the union $K_{f_0, \cdots, f_k}:=\bigcup_{i=0}^ks\big(\supp(f_i)\cap s^{-1}(\orbO)\big)$ is also compact in $\orbO$.  

Let $K$ be a precompact open subset of $\orbO$ containing $K_{f_0, \cdots, f_k}$ as a proper subset.  Observe that the closure of $K$ is compact in $\orbO$. Hence, there is a positive constant $\varepsilon$ such that the $\varepsilon$-neighborhood $T^\varepsilon_{K}$ of $K$ is contained inside the $\delta$-neighborhood $T^\delta_{\orbO}$, the range of the linearization map $\Theta$ in (\ref{eq:tubular}). Applying the homotopy map $\Psi_
\varepsilon$ defined in Lemma \ref{lem:localization-homotopies} to $f_0\otimes \cdots \otimes f_k$, we may assume without loss of generality that the support of $f_0,\cdots, f_n$ is contained inside $T^\varepsilon_{K}$, and therefore inside the $\delta$-neighborhood $T^\delta_{\orbO}$. 
Accordingly, the pullback function $\Theta^*(f_0\otimes \cdots \otimes f_k)$ is well defined and supported in 
\[
\big(\sfG_{|\orbO}\ltimes N\orbO\big)_{|\Theta^{-1}(T^\varepsilon_{K})}\times \cdots \times \big(\sfG_{|\orbO}\ltimes N\orbO\big)_{|\Theta^{-1}(T^\varepsilon_{K})}.
\] 
Let $U^\varepsilon_{\orbO}$ be the $\varepsilon$-neighborhood of $\orbO$ in $N\orbO/\sfG_{|\orbO}$. By the definition of $\Theta$, it is not difficult to check that $\Theta^*(f_i)$ is supported inside $\big(\sfG|_\orbO \rtimes N\orbO\big)|_{\Theta^{-1}(T^\varepsilon_{K})}$ for $i=0,\cdots, k$ and therefore $\Theta^*(f_0\otimes \cdots \otimes f_k )$ is a well defined $k$-chain in $\scrC_k\big(\calA_{N\orbO}(U^\varepsilon_{\orbO})\big).$ Define
$\Theta_{\bullet,\orbO}\big( [f_0\otimes\cdots \otimes f_k]\big)\in
\sheafC_{\bullet,\smorbO}(\calA_{N\orbO})$ to be the germ of $\Theta^*(f_0\otimes \cdots \otimes f_k )$ at the point $\orbO$ in the orbit space $X_{N\orbO} = N\orbO/\sfG_{|\orbO}$.
It is worth pointing out that the construction of
$\Theta_{\bullet,\orbO} \big( [f_0\otimes\cdots \otimes f_k]\big)$ is independent of the choices of the subset $K$ and the constant $\varepsilon$.
Analogously, using the inverse map $\Theta^{-1}$, we can construct the inverse morphism
$(\Theta^{-1})_{\bullet,\orbO}$ from $\sheafC_{\bullet,\smorbO}(\calA_{N\orbO})$ to $\sheafC_{\bullet,\smorbO}(\calA_\sfG)$, and therefore prove that $\Theta_{\bullet,\orbO}$ is a quasi-isomorphism. We leave the details to the diligent reader. 
\end{proof}

\subsection{Computation of the linear model}
We compute in this subsection the cohomology of $C_\bullet(\calA_{N\orbO})$.
Our method is inspired by the work of Crainic and Moerdijk \cite{CraMoeFGCH}. 

To start with, recall that  we prove in \cite[Cor. 3.11 ]{PflPosTanGOSPLG} that for a proper Lie groupoid $\sfG\rightrightarrows M$, given $x\in M$, there is a neighborhood $U$ of $x$ in $M$  diffeomorphic to $O\times V_x$ where $O$ is an open ball in the orbit $\orbO$ through $x$ centered at $x$, and $V_x$ is a $\sfG_x$ --the isotropy group of $\sfG$ at $x$-- invariant open ball in $N_x \orbO$ centered at the origin. Under this diffeomorphism $\sfG_{|U}$ is isomorphic to the product of the pair groupoid $O\times O\rightrightarrows O$ and the transformation groupoid $\sfG_x\ltimes V_x\rightrightarrows V_x$.  Applying this result to the the transformation groupoid $\sfG_{|\orbO}\ltimes N\orbO\rightrightarrows N\orbO$, we conclude that given any $x\in \orbO$, there is an open ball $O$ of $x$ in $\orbO$ such that the restricted normal bundle $U_x:=N\orbO_{|O}$ is diffeomorphic to $N_x\orbO \times O$ and $\big(\sfG_{|\orbO}\ltimes N\orbO\big)_{|U_x}$ is isomorphic to the product of the pair groupoid $O\times O$ and the transformation groupoid $\sfG_x\ltimes N_x\orbO$.  

Following the above local description of $\sfG_{|\orbO}\ltimes N\orbO$, we choose a covering $( O_x)_{x\in \orbO}$
of the orbit $\orbO$, and therefore also a covering $\mathfrak{U}:= ( U_x)_{x\in \orbO}$,
$U_x:=O_x\times N_x\orbO$, of $N\orbO$. We choose a locally finite countable subcovering $(O_i)_{i\in I}$ of
$\orbO$ and the associated covering $(U_i)_{i\in I}$ of $N\orbO$. Choose
$\varphi_i\in \calC_\textup{c}^\infty(\orbO)$ such that $\varphi_i^2$ is a partition of unity subordinate to the
open covering $( O_i)_{i\in I}$ of $\orbO$. Lift $\varphi_i\in \calC_\textup{c}^\infty(\orbO)$ to
$\tilde{\varphi}_i\in \calC^\infty (N\orbO)$ that is let it be constant along the fiber direction. As $\varphi_i$ is compactly supported, $\tilde{\varphi}_i$ is longitudinally compactly supported and therefore belongs to $\calA_{N\orbO}$. 
Now consider the groupoid $\sfH_{\mathfrak{U}}$ over the disjoint union $\sqcup U_i$, such that arrows from $U_i$ to $U_j$ are arrows in $\sfG_{|\orbO}\ltimes N\orbO$ starting from $U_i$ and ending in $U_j$. Composition of arrows in $\sfG_{|\orbO}\ltimes N\orbO$ equips $\sfH_{\mathfrak{U}}$ with a natural Lie groupoid structure that is Morita equivalent to $\sfG_{|\orbO}\ltimes N\orbO$. 
As a consequence of this the orbit spaces of the groupoids
$\sfG_{|\orbO}\ltimes N\orbO$ and $\sfH_{\mathfrak{U}}$ are therefore naturally homeorphic,
actually even diffeomorphic in the sense of differentiable spaces. We therefore identify them.

The following lemma is essentially due to Crainic and Moerdijk \cite{CraMoeFGCH}. 
\begin{lemma}\label{lem:covering} The map
  $\Lambda: A(\sfG_{|\orbO}\ltimes N\orbO):=
  \Gamma\big(\calA_{\sfG_{|\orbO}\ltimes N\orbO}\big)\to
  A(\sfH_{\mathfrak{U}}):=\Gamma\big(\calA_{\sfH_{\mathfrak{U}}}\big)$ defined by
\[
\Lambda(f):=(\tilde{\varphi}_i f \tilde{\varphi}_j)_{i,j}
\]
is an algebra homomorphism  which induces a quasi-isomorphism $\Lambda_\bullet$ from $C_\bullet\big(A(\sfG_{|\orbO}\ltimes N\orbO)\big)$ to $C_\bullet\big(A(\sfH_{\mathfrak{U}})\big)$.
In addition, $\Lambda$ induces a quasi-isomorphism of sheaf complexes
\[
  \Lambda_\bullet: \sheafC_{\bullet} (\calA_{\sfG_{|\orbO}\ltimes }) \to
  \sheafC_{\bullet} (\calA_{\sfH_{\mathfrak{U}}})
\]
over their joint orbit space
$ N\orbO/\sfG_{|\orbO} \cong (\sfH_{\mathfrak{U}})_0/\sfH_{\mathfrak{U}}$. 
\end{lemma}
\begin{proof}
  The proof of the claim is a straightforward generalization of the one of \cite[Lemma 5]{CraMoeFGCH}.
  The slight difference here is that we work with the algebras $A(\sfG_{|\orbO}\ltimes N\orbO)$ and
  $A(\sfH_{\mathfrak{U}})$ instead of the algebra of compactly supported functions.
  We skip the proof here to avoid repetition.
\end{proof}

Next, the groupoid $\sfH_{\mathfrak{U}}$ can be described more explicitly as follows. Firstly, index the open sets in the covering $(U_i)_{i\in I}$ by natural numbers so in other words assume $I \subset\N^*$. After possibly reindexing
again, one can assume that if $k\in I$, then $l\in I$ for all $1\leq l\leq k$. 
Secondly, given $i$, write $x \in U_i$ as $(x_{\textup{v}}, x_{\textup{o}})$ where $x_{\textup{v}}\in N_{x_i}\orbO$ and $x_{\textup{o}}\in O_i$. Choose a diffeomorphism $\psi_i: O_i\to \mathbb{R}^k$, where $k=\dim(\orbO)$.  Thirdly, for any $1< i\in I$, choose an arrow $g_i\in \sfG$ from $x_1$ to $x_i$. The arrow $g_i$ induces an isomorphism between $N_{x_1}\orbO$ and $N_{x_i}\orbO$, and conjugation by $g_i$ defines an isomorphism from $\sfG_{x_i}$ to $\sfG_{x_1}$. Accordingly, $g_i$ induces a groupoid isomorphism between $\sfG_{x_1}\ltimes N_{x_1}\orbO$ and $\sfG_{x_i}\ltimes N_{x_i}\orbO$. 

\begin{lemma}\label{lem:productgroupoid}
The groupoid $\sfH_{\mathfrak{U}}$ is isomorphic to the product groupoid
\[
  \big(\sfG_{x_1}\ltimes N_{x_1}\orbO\big)\times (I\times I)
  \times (\mathbb{R}^k\times \mathbb{R}^k) \ .
\]  
\end{lemma}
\begin{proof}
We define groupoid morphisms 
\[
  \Phi:  \: \sfH_{\mathfrak{U}}\to \big(\sfG_{x_1}\ltimes N_{x_1}\orbO\big)\times (I\times I)\times (\mathbb{R}^k\times \mathbb{R}^k)
\] and
\[
  \Psi: \: \big(\sfG_{x_1}\ltimes N_{x_1}\orbO\big)\times (I\times I)\times (\mathbb{R}^k\times \mathbb{R}^k)\to \sfH_{\mathfrak{U}}. 
\]
Given an arrow $h\in \sfH_{\mathfrak{U}}$ with source in $U_i$ and target in $U_j$, we consider
$(s(h)_{\textup{o}}, x_i)\in O_i\times O_i$ and $(t(h)_{\textup{o}}, x_j)\in O_j\times O_j$. Define $h_{x_i}\in (\sfG_{x_i}\ltimes N_{x_i}\orbO)\times (O_i\times O_i)$ (and $h_{x_j}\in (\sfG_{x_j}\ltimes N_{x_j}\orbO)\times (O_i\times O_i)$) by 
$h_{x_i}=\big((\id, 0), (s(h)_{\textup{o}}, x_i)\big)$ (and $h_{x_j}=\big((\id, 0), (t(h)_{\textup{o}}, x_j)\big)$).
The arrow $g_j^{-1}h_{x_j}^{-1}h h_{x_i}g_i$ belongs to ${\sfH_\mathfrak{U}}_{|U_1}$ and its component in
$O_1\times O_1$ is $(x_1, x_1)$. The arrow $\Phi (h)$ now is defined to be
\[
\Phi(h):=\big(g_j^{-1}h_{x_j}^{-1}h h_{x_i}g_i, (i,j), (\psi(s(h_{ij}), t(h_{ij}))\big)\in \big(\sfG_{x_1}\ltimes N_{x_1}\orbO\big)\times (I\times I)\times (\mathbb{R}^k\times \mathbb{R}^k). 
\]
Similarly, given $(k, (i,j), (y_i, y_j))\in \big(\sfG_{x_1}\ltimes N_{x_1}\orbO\big)\times (I\times I)\times (\mathbb{R}^k\times \mathbb{R}^k)$, define 
\[
  h_{y_i}:=\big((\id, 0), (\psi_i^{-1}(y_i), x_i)\big)\in \sfG_{|U_i},\qquad
  h_{y_j}:=\big((\id, 0), (\psi^{-1}_j(y_j), x_j)\big)\in \sfG_{|U_j},
\] 
and $h_1:=\big(k, (x_1, x_1)\big)\in \sfG_{|U_1}$.  Notice $g_{j}h_1g_{i}^{-1}$ is an arrow in
$\sfH_{\mathfrak{U}}$ starting from $x_i$ and ending at $x_j$.
We can now define $\Psi(k,(i,j), (y_i, y_j))$  to be
\[
\Psi(k,(i,j), (y_i, y_j)):=h_{y_j}g_{j}h_1g_{i}^{-1}h_{y_i}^{-1}\in \sfH_{\mathfrak{U}}. 
\]
It is straightforward to check that $\Phi$ and $\Psi$ are groupoid morphisms and inverse
to each other.  
\end{proof}

Let $A\big((\sfG_{x_1}\ltimes N_{x_1}\orbO)\times (I\times I)\times (\mathbb{R}^k\times \mathbb{R}^k) \big)$
be the space of global sections of the convolution sheaf
$\calA_{(\sfG_{x_1}\ltimes N_{x_1}\orbO)\times (I\times I)\times (\mathbb{R}^k\times \mathbb{R}^k)}$. 
With the maps $\Phi$ and $\Psi$ introduced in Lemma \ref{lem:productgroupoid}, we have the following induced isomorphisms of chain complexes,
\[
\begin{split}
\Phi_\bullet \!:\:\, &C_\bullet \Big(A\big(\big(\sfG_{x_1}\ltimes N_{x_1}\orbO\big)\times (I\times I)\times (\mathbb{R}^k\times \mathbb{R}^k) \big)\Big)\to C_\bullet\big(A(\sfH_{\mathfrak{U}})\big),\\
\Psi_\bullet \!:\:\, &C_\bullet\big(A(\sfH_{\mathfrak{U}})\big)\to C_\bullet \Big(A\big(\big(\sfG_{x_1}\ltimes N_{x_1}\orbO\big)\times (I\times I)\times (\mathbb{R}^k\times \mathbb{R}^k) \big)\Big). 
\end{split}
\]
Since they are induced by an ismorphism of groupoids, we also obtain a pair of mutually inverse
isomorphisms of complexes of sheaves which are denoted by the same symbols,
\[
\begin{split}
  \Phi_\bullet \!:\:\, &\sheafC_{\bullet} \left(\calA_{(\sfG_{x_1}\ltimes N_{x_1}\orbO)\times (I\times I)\times (\mathbb{R}^k\times \mathbb{R}^k)}\right)\to
  \sheafC_\bullet\left(\calA_{\sfH_{\mathfrak{U}}}\right),\\
  \Psi_\bullet \!:\:\, &\sheafC_\bullet\left(\calA_{\sfH_{\mathfrak{U}}}\right)\to
  \sheafC_\bullet \left(\calA_{(\sfG_{x_1}\ltimes N_{x_1}\orbO)\times (I\times I)\times (\mathbb{R}^k\times \mathbb{R}^k)}\right). 
\end{split}
\]

Observe that both  groupoids $I\times I$ and $\mathbb{R}^k\times \mathbb{R}^k$ have only one orbit. Therefore, longitudinally compactly supported functions on them are the same as compactly supported functions. Observe that $\calC^\infty(\sfG_{x_1}\ltimes N_{x_1}\orbO)$ is the algebra of longitudinally compactly supported smooth functions on $\sfG_{x_1}\ltimes N_{x_1}\orbO$. By Lemma \ref{lem:productgroupoid}, the groupoid algebra $A(\sfH_{\mathfrak{U}})$ is isomorphic to $A\big((\sfG_{x_1}\ltimes N_{x_1}\orbO)\times (I\times I)\times (\mathbb{R}^k\times \mathbb{R}^k) \big)$. The latter can be identified with
$\calC^\infty(\sfG_{x_1}\ltimes N_{x_1}\orbO)\,\hat{\otimes}\,
\R^{I\times I}\,\hat{\otimes}\,
\calC_\textup{c}^\infty(\mathbb{R}^k\times \mathbb{R}^k)$,
where $\R^{I\times I}$ is the space of finitely supported functions on
$I\times I$.
Note that $I\times I$ and $\mathbb{R}^k\times \mathbb{R}^k$ both carry the structure of a pair groupoid, so the corresponding products on $\R^{I\times I}$ and
$\calC_\textup{c}^\infty(\mathbb{R}^k\times \mathbb{R}^k)$ are given in both cases by
convolution which we denote as usual by $*$.
Let $\tau_I$ be the trace on $\R^{I\times I}$ defined by
\[
  \tau_I(d) :=\sum_i d_{ii} \ ,
  \quad d = (d_{ij})_{i,j \in I} \in \R^{I\times I}
\] 
and let $\tau_{\mathbb{R}^k}$ be the trace on
$\calC_\textup{c}^\infty (\mathbb{R}^k\times \mathbb{R}^k)$ given by
\[
  \tau_{\mathbb{R}^k}(\alpha):=\int_{\mathbb{R}^k} \alpha(x,x)dx \ ,
  \quad \alpha \in \calC_\textup{c}^\infty (\mathbb{R}^k\times \mathbb{R}^k) \ ,
\]
where $dx$ is the Lebesgue measure on $\mathbb{R}^k$. Define a map 
\[
  \tau_m: C_m\big(\calC^\infty(\sfG_{x_1}\ltimes N_{x_1}\orbO)\,\hat{\otimes}\,
  \R^{I\times I}\,\hat{\otimes}\,
  \calC_\textup{c}^\infty(\mathbb{R}^k\times \mathbb{R}^k)\big)\to
  C_m\big(\calC^\infty(\sfG_{x_1}\ltimes N_{x_1}\orbO)\big)
\]
as follows:
\[
\begin{split}
  \tau_m \, & \big( ( f_0\otimes \cdots \otimes f_m)\otimes (d_0\otimes \cdots\otimes d_m)
  \otimes (\alpha_0\otimes \cdots \otimes \alpha_m)\big)\\
  &:=\tau_I(d_0 * \cdots * d_m)
  \tau_{\mathbb{R}^k}(\alpha_0 * \cdots * \alpha_m) \, f_0\otimes \cdots \otimes f_m\ , \\
  &f_0 ,\cdots , f_m \in \calC^\infty(\sfG_{x_1}\ltimes N_{x_1}\orbO) , \:
   d_0, \cdots , d_m \in  \R^{I\times I}, \:
   \alpha_0,\cdots,\alpha_m\in\calC_\textup{c}^\infty (\mathbb{R}^k\times \mathbb{R}^k)  \ . 
\end{split}
\]
It is easy to check using the tracial property of $\tau_I$ and $\tau_{\mathbb{R}^k}$ that $\tau_\bullet$ is a chain map. Moreover, observe that the whole argument works not only for the
global section algebra $\calC^\infty(\sfG_{x_1}\ltimes N_{x_1}\orbO)$ but for any
of the section algebras $\calC^\infty(\sfG_{x_1}\ltimes V)$ with $V \subset N_x\orbO$
an open $\sfG_{x_1}$-invariant subspace. So eventually we obtain a morphism of sheaf complexes 
\[
    \tau_\bullet :
    \sheafC_\bullet \big(\calA_{\calC^\infty(\sfG_{x_1}\ltimes N_{x_1}\orbO)\, \hat{\otimes}\,
    \R^{I\times I} \, \hat{\otimes} \,
    \calC_\textup{c}^\infty(\mathbb{R}^k\times \mathbb{R}^k)}\big)
    \to \sheafC_\bullet \big(\calA_{\calC^\infty(\sfG_{x_1}\ltimes N_{x_1}\orbO)}\big) \ .
  \]
  over the orbit space $N_{x_1}\orbO/\sfG_{x_1}$.
\begin{lemma}\label{lem:linearmodel}
  The chain map
  \[
    \tau_\bullet : C_\bullet \big(\calC^\infty(\sfG_{x_1}\ltimes N_{x_1}\orbO)\, \hat{\otimes}\,
    \R^{I\times I} \, \hat{\otimes} \,
    \calC_\textup{c}^\infty(\mathbb{R}^k\times \mathbb{R}^k)\big)\to C_\bullet \big(\calC^\infty(\sfG_{x_1}\ltimes N_{x_1}\orbO)\big)
  \]
  is a quasi-isomorphism. More generally,
  
  \[
    \tau_\bullet :
    \sheafC_\bullet \big(\calA_{\calC^\infty(\sfG_{x_1}\ltimes N_{x_1}\orbO)\, \hat{\otimes}\,
    \R^{I\times I} \, \hat{\otimes} \,
    \calC_\textup{c}^\infty(\mathbb{R}^k\times \mathbb{R}^k)}\big)
    \to \sheafC_\bullet \big(\calA_{\calC^\infty(\sfG_{x_1}\ltimes N_{x_1}\orbO)}\big)
  \]
  is an isomorphism of complexes of sheaves.
\end{lemma}
\begin{proof}
Choose a function $\beta\in \calC_\textup{c}^\infty(\mathbb{R}^k)$ such that 
\[
\int_{\mathbb{R}^k} \beta^2(x)dx=1. 
\]
Let $\alpha\in \calC_\textup{c}^\infty(\mathbb{R}^k\times \mathbb{R}^k)$ be the function
$\beta\otimes \beta$. Define an algebra morphism 
\[
  j_\alpha: \calC^\infty(\sfG_{x_1}\ltimes N_{x_1}\orbO)\to
  \calC^\infty(\sfG_{x_1}\ltimes N_{x_1}\orbO)\,\hat{\otimes}\,\R^{I\times I} \, \hat{\otimes} \,\calC_\textup{c}^\infty(\mathbb{R}^k\times \mathbb{R}^k)
\]
by 
\[
j_\alpha(f)=f\otimes \delta_{(1,1)}\otimes \alpha \ ,
\]
where $\delta_{(1,1)}$ is the function on $I\times I$ that is $1$ on $(1,1)$
and $0$ otherwise.  $j_{\alpha,\bullet}$ is the induced map on the cochain complex. 
It is easy to check $\tau_\bullet \circ j_{\alpha,\bullet}=\id$. Applying
$j_{\alpha, \bullet}\circ \tau_\bullet$ to
\[ (f_0\otimes \cdots \otimes f_m)\otimes (d_0\otimes \cdots \otimes d_m)\otimes (\alpha_0\otimes \cdots \otimes \alpha_m)\]
gives 
\[
\tau_I (d_0 * \cdots * g_m)\tau_{\mathbb{R}^k}(\alpha_0 * \cdots * \alpha_m)\big( f_0\otimes \cdots \otimes f_m\big)\otimes \big(\delta_{1,1}\otimes\cdots \otimes \delta_{1,1}\big)\otimes \big(\alpha\otimes \cdots \otimes \alpha\big). 
\]

Following  the proof of Lemma \ref{lem:localization-homotopies}, we consider the unital
algebra $\widetilde{\calC}^\infty(\sfG_{x_1}\ltimes N_{x_1}\orbO)$ which is the direct sum of
$\calC^\infty(\sfG_{x_1}\ltimes N_{x_1}\orbO)$ with $\calC^\infty(N_{x_1}\orbO)^{\sfG_{x_1}}$
and product structure given by Eq.~\eqref{eq:product}.
We then have the following split exact sequence in the category of bornological algebras
\begin{equation}\label{eq:unitalization}
  0\rightarrow \calC^\infty(\sfG_{x_1}\ltimes N_{x_1}\orbO)\rightarrow
  \widetilde{\calC}^\infty(\sfG_{x_1}\ltimes N_{x_1}\orbO)\rightarrow
  \calC^\infty(N_{x_1})^{\sfG_{x_1}}  \rightarrow 0.
\end{equation}
It is not hard to see that the chain maps $\tau_\bullet$ and $j_{\alpha, \bullet}$ extend to
the corresponding versions of the algebras $ \widetilde{\calC}^\infty(\sfG_{x_1}\ltimes N_{x_1}\orbO)$ and $\calC^\infty(N_{x_1})^{\sfG_{x_1}}$. As both algebras are unital, the homotopy maps constructed in the proof of \cite[Lemma 6]{CraMoeFGCH} can be applied to conclude that $j_{\alpha, \bullet}\tau_\bullet$ is a quasi-isomorphism for $ \widetilde{\calC}^\infty(\sfG_{x_1}\ltimes N_{x_1}\orbO)$ and $\calC^\infty(N_{x_1})^{\sfG_{x_1}}$. As the algebra $\calC^\infty(\sfG_{x_1}\ltimes N_{x_1})$ is $H$-unital, we consider the long exact sequence associated to the short exact sequence (\ref{eq:unitalization}). As $j_{\alpha, \bullet}$ and $\tau_{\bullet}$ are quasi-isomorphisms on
$ \widetilde{\calC}^\infty(\sfG_{x_1}\ltimes N_{x_1}\orbO)$ and $\calC^\infty(N_{x_1})^{\sfG_{x_1}}$, we conclude by the five lemma that $\tau_\bullet$ and $j_{\alpha, \bullet}$ are also quasi-isomorphisms for $\calC^\infty(\sfG_{x_1}\ltimes N_{x_1}\orbO)$.
The argument generalizes immediately to the sheaf case.
\end{proof}

Summarizing Lemma \ref{prop:stalk-linearization} -- Lemma \ref{lem:linearmodel}, we thus
obtain the following local model for the stalk complex $\sheafC_{\bullet,\smorbO}(\calA_\sfG)$.
\begin{proposition}\label{prop:local-model}
  For every orbit $\orbO\in X$ the composition
  $L_{\bullet,\orbO}:=\tau_{\bullet,0} \circ \Psi_{\bullet,0}\circ \Lambda_{\bullet,0}\circ
  \Theta_{\bullet,\orbO}$, where $\tau_{\bullet,0}$, $\Psi_{\bullet,0}$, and $\Lambda_{\bullet,0}$
  denote the respective sheaf morphisms localized at the zero sections, 
  is a quasi-isomorphism,
\[
\begin{split}
L_{\bullet,\orbO}: \sheafC_{\bullet,\orbO}(\calA_{\sfG})&\stackrel{\Theta_{\bullet,\orbO}}{\longrightarrow}\sheafC_{\bullet,\orbO} \big(\calA_{\sfG_{|\orbO}\ltimes N\orbO}\big)\stackrel{\Lambda_{\bullet,0}}{\longrightarrow}\sheafC_{\bullet,0} \Big(\calA_{\sfH_{\mathfrak{U}}}\Big)
\\
&\stackrel{\Psi_{\bullet,0}}{\longrightarrow} \sheafC_{\bullet,0} \big( \calA_{(\sfG_{x_1}\ltimes N_{x_1}\orbO) \times (I\times I)\times (\mathbb{R}^k\times \mathbb{R}^k)}\big)\stackrel{\tau_{\bullet,0}}{\longrightarrow}\sheafC_{\bullet,0} \big(\calA_{\sfG_{x_1}\ltimes N_{x_1}\orbO}\big).
\end{split}
\]
\end{proposition}

\section{Basic relative forms}
Let $M$ be a smooth manifold equipped with a left action of a compact Lie group
$G$ which we write as $(g,x)\mapsto gx,$ for $g\in G, x\in M$. Associated to
this action is the Lie groupoid $G\ltimes M\rightrightarrows M$ with source map
given by the projection $(g,x)\mapsto x$ and target given by the action
$(g,x)\mapsto gx$. The {\em loop space}
$\Lambda_0(G\ltimes M)\subset G\times M$ coincides in this case with the
disjoint union of all fixed point sets $M^g\subset M$ for $g\in G$:
\[
  \Lambda_0(G\ltimes M):=\big\{(g,p)\in G\times M \mid gp=p\big\}=
  \bigcup_{g\in G}\{g\}\times M^g \ .
\]
For fixed $g\in G$, the fixed point subset $M^g\subset M$ is a closed
submanifold but it can wildly vary as $g$ varies over $G$. Therefore, the
loop space $\Lambda_0(G\ltimes M)$ is a singular subset of $G\times M$. If we
let $G$ act on $G\times M$ by 
\[
h\cdot (g,p):=(hgh^{-1},hp), \quad h\in G , \: (g,p)\in G \times M \ ,
\]
this action preserves $\Lambda_0(G\ltimes M)\subset G\times M$ sending $M^g$ to $M^{hgh^{-1}}$.
In \cite{BryAAGAH,BryCHET}, Brylinski introduces the notion of {\em basic relative forms}.
Intuitively, a basic relative $k$-form is a smooth family
$(\omega_g)_{g\in G} \in \prod_{g\in G} \Omega^k(M^g)$ of  
differential forms on fixed point subspaces which are
\begin{enumerate}[(i)]
\item\label{ite:horizontal}
  {\em horizontal} that is $i_{\xi_{M^g}}\omega_g=0$ for all
  $g\in G$ and $\xi\in {\rm Lie}(G_g)$, and
\item\label{ite:invariant}
  $G$-{\em invariant} which means that
  $h^*\omega_g=\omega_{h^{-1}gh}$ for all  $g,h \in G$.
\end{enumerate}
Here, $G_g:=Z_G(g)$ denotes the centralizer of $g\in G$, which acts on $M^g$.
Because of the singular nature of $\Lambda_0$, one needs to make sense of what is exactly
meant by a {\em smooth} family of differential forms. There are two solutions for this: 

\subsection*{(A) Sheaf theory}
In the sense of Grauert--Grothendieck and following Brylinski \cite{BryCHET}, we define the sheaf
of relative forms on $\Lambda_0(G\ltimes M)$ as the quotient sheaf
\[
  \rOmega{k}{\Lambda_0}:=\iota^{-1}\left(\Omega^k_{G\ltimes M\to G}\slash  \left( \calJ \Omega^k_{G\ltimes M\to G} + d_\textup{rel}\calJ \wedge \Omega^{k-1}_{G\ltimes M\to G} \right) \right) \ .
\]
Here, $\Omega^k_{G\ltimes M\to G}$ denotes the sheaf of $k$-forms on $G\times M$ relative to the projection 
${\rm pr}_1:G\times M\to G$ and $\iota$ the canonical injection
$\Lambda_0(G\ltimes M) \hookrightarrow G\ltimes M$.
A form $\omega\in \Omega^k_{G\ltimes M\to G} (\widetilde{U})$ for
$\widetilde{U}\subset G\ltimes M$ open is given by a smooth global section
of the vector bundle $s^*\bigwedge^kT^*M$ that is by an element
$\omega \in \Gamma^\infty (\widetilde{U},s^*\bigwedge^kT^*M)$.
The de Rham differential on $M$ defines a differential
$d_\textup{rel}: \Omega^k_{G\ltimes M\to G}\to  \Omega^{k+1}_{G\ltimes M\to G}$.
Finally, $\calJ$ denotes the vanishing ideal of smooth functions on $G\times M$ that restrict to zero on
$\Lambda_0(G\ltimes M)\subset G\times M$. Note that
$\calJ \Omega^\bullet_{G\ltimes M\to G} + d_\textup{rel}\calJ \wedge \Omega^\bullet_{G\ltimes M\to G} $ is a
differential graded ideal in the sheaf complex $\left( \Omega^k_{G\ltimes M\to G}, d_\textup{rel}\right)$,
so $\rOmega{\bullet}{\Lambda_0}$ becomes a sheaf of differential graded algebras
on the loop space. For open $U\subset \Lambda_0(G\ltimes M)$, an element of
$\rOmega{k}{\Lambda_0} (U)$ can now be understood as an equivalence class $[\omega]_{\Lambda_0}$
of forms $\omega \in \Omega^k_{G\ltimes M\to G} (\widetilde{U})$ defined
on some open $\widetilde{U}\subset G\ltimes M$ such that $U = \widetilde{U}\cap \Lambda_0(G\ltimes M)$.
This explains the definition of the sheaf complex of relative forms on the singular space
$\Lambda_0(G\ltimes M)$; confer also \cite{PflPosTanGGCDSSCB}.
Next observe that the map which associates to each $p\in M$ the conormal space 
$N^*_p := \big( T_pM / T_p\calO_p \big)^* $ is a generalized subdistribution of the cotangent bundle
$T^*M$ in the sense of Stefan-Suessmann, cf.~\cite{SteISVF,SusOFVFID,JotRatSniSRDS}.
In the language of \cite{DraLeeParRicSDFG}, $N^*$ is a cosmooth generalized distribution.  
The restriction of $N^*$ to each orbit, and even to each  stratum of $M$ of a fixed isotropy type, is
a vector bundle, cf.~\cite{PflPosTanGOSPLG}. 
Henceforth,  the pullback distribution $s^*\bigwedge^k N^*$ is naturally a cosmooth generalized subdistribution
of $\bigwedge^k T^*G\ltimes M$. We define the space $\hrOmega{k}{\Lambda_0\sfG} (U)$ of 
\emph{horizontal relative $k$-forms on the loop space} (over $U$)
as the  subspace 
\[
  \hrOmega{k}{\Lambda_0\sfG} (U) := \big\{ [\omega]_{\Lambda_0} \in \rOmega{k}{\Lambda_0\sfG} (U) \mid
  \omega_{(g,p)} \in {\bigwedge}^kN^*_p \text{ for all } (g,p)\in U \big\} 
  \, .
\]
This implements the above condition \eqref{ite:horizontal}.
Observe that the action of $G$ on $TN$ leaves the orbits invariant, hence induces also an action on the
conormal distribution $N^*$ in a canonical way \cite[Sec.~3]{PflPosTanGOSPLG}.
Call a section $[\omega]_{\Lambda_0} \in \hrOmega{k}{\Lambda_0} (U)$ \emph{invariant}, if 
\begin{equation}
  \label{eq:DefInvHorForms}
  \omega_{hgh^{-1},hp} (hv_1 , \ldots , hv_k ) =  \omega_{(g,p)} (v_1 , \ldots , v_k )
\end{equation}
for all $(g,p)\in U \subset \Lambda_0 \sfG $, $h\in G $ such that $(hgh^{-1},hp)\in U$ and
$v_1 , \ldots , v_k \in N_p$. Note that the invariance of $[\omega]_{\Lambda_0}$ does not depend on the
particular choice of the representative $\omega$ such that $\omega_p \in \bigwedge^kN^*_p$.
Condition \eqref{ite:invariant} is covered by defining the space $\brOmega{k}{\Lambda_0} (U)$ of
\emph{basic relative $k$-forms on the loop space} (over $U$) now as the space of all invariant
horizontal relative $k$-forms 
$[\omega]_{\Lambda_0} \in \hrOmega{k}{\Lambda_0\sfG} (U)$. Obviously, one thus obtains sheaves
$\hrOmega{k}{\Lambda_0}$ and $\brOmega{k}{\Lambda_0}$ on the loop space $\Lambda_0(G\ltimes M)$.
We will call the push forward $\pi_* s_* \brOmega{k}{\Lambda_0}$ by the source map $s$ and canonical projection
$\pi:M \to X =M/G$ sheaf of basic relative functions as well and denote it also 
by the symbol $\brOmega{k}{\Lambda_0}$. This will not lead to any confusion. 
The interpretion of basic relative forms as smooth families of forms on the fixed point manifolds is
still missing, but will become visible in the following approach.

\subsection*{(B) Differential Geometry}
From a more differential geometric perspective, we consider the family of vector bundles
$F\to \Lambda_0$ defined by $F_{(g,p)}:=T_p^*M^g$ for $(g,p)\in \Lambda_0(G\ltimes M)$. Of course,
this does not define a (topological) vector bundle over the inertia space
$\Lambda_0 (G\ltimes M)$ because in general the rank jumps discontinuously but it is again a
cosmooth generalized distribution. Using the canonical projection $s^*T^*M|_{\Lambda_0}\to F$ we say that
a local section $\omega\in\Gamma(U,\bigwedge^kF)$ over $U\subset\Lambda_0$ is {\em smooth} if for each
$(g,p)\in U$ there exist open neighborhoods $O\subset G$ of $g$ and $V\subset M$ of $p$
together with a \emph{locally representing} smooth $k$-form
$\omega_{O,V} \in \Gamma^\infty (O \times V, \bigwedge^k s^* T^*M)$ such that 
$(O \times V) \cap \Lambda_0 \subset U $ and $\omega_{(h,q)} = \big[\omega_{O,V}\big]_{(h,q)}$ for all  
$(h,q) \in (O \times V) \cap \Lambda_0(G\ltimes M)$.  Hence a smooth section $\omega$ can be identified
with the smooth family $(\omega_g)_{g\in \operatorname{pr}_G(U)}$  of forms 
$\omega_g \in \Omega^k \Big( s \big( U \cap (\{ g \} \times M^g)\big)\Big)$ wich are uniquely defined 
by the condition that ${\omega_g}_{|V^g} = \iota_{V^g}^* \omega_{O,V}$ for all $g \in O$  
and all pairs $(O,V)$ with locally representing forms $\omega_{O,V}$ as before.
The $\iota_{V^g} : V^g \hookrightarrow V$ hereby are  the canonical embeddings of the fixed point manifolds $V^g$. 
We denote the space of all smooth sections of $\bigwedge^k F$ over  $U$ by 
$\Gamma^\infty (U,\bigwedge^k F )$ or  $\Gamma^\infty_{\bigwedge^k F} (U)$. Obviously, $\Gamma^\infty_{\bigwedge^k F}$ 
becomes a sheaf on $\Lambda_0$. 

\begin{proposition}
\label{prop:factorization-relative-forms}
  The canonical sheaf morphism $\theta^k:\iota^{-1}\Gamma^\infty_{\bigwedge^k s^* T^*M}  \to  \Gamma^\infty_{\bigwedge^k F}$
  factors through a unique epimorphism of sheaves $\Theta^k:\rOmega{^\bullet}{\Lambda_0} \to  \Gamma^\infty_{\bigwedge^k F}$ making the following diagram commutative:
  \[
  \xymatrix{\iota^{-1}\Gamma^\infty_{\bigwedge^k s^* T^*M} \ar[rr]^{\theta^k}\ar[d]&&\Gamma^\infty_{\bigwedge^k F}\\
  \rOmega{^\bullet}{\Lambda_0}\ar[rru]_{\Theta^k}&&}
  \]
\end{proposition}

\begin{proof}  
The claim follows by showing that for open $\tildeU \subset G \times M$ 
  and $U := \tildeU \cap \Lambda_0 (G \ltimes M)$ the canonical map 
  $\theta^k_{\tildeU}: \Gamma^\infty (\tildeU , \bigwedge^k s^* T^*M) \to   \Gamma^\infty (U ,\bigwedge^k F)$, 
  $\omega \mapsto [\omega]$ is surjective and has
  \[
    \calK (\tildeU) := \calJ (\tildeU) \, \Gamma^\infty (\tildeU , {\bigwedge}^k s^* T^*M) + 
    d_\textup{rel}\calJ (\tildeU) \wedge  \Gamma^\infty (\tildeU,{\bigwedge}^{k-1} s^* T^*M)
  \]
  contained in its kernel.

  The sheaf  $\Gamma^\infty_{\bigwedge^k F}$ is a $\calC^\infty_{\Lambda_0}$-module sheaf, hence a soft sheaf. This entails surjectivity
  of $\theta^k_{\tildeU}$. Assume that $\omega \in  \Gamma^\infty (\tildeU,\bigwedge^k s^* T^*M)$ is of the form $\omega = f\, \varrho$ 
  for some $f \in \calJ (\tildeU)$ and $\varrho  \in  \Gamma^\infty (\tildeU,\bigwedge^k s^* T^*M)$. Then 
  \[
   \theta^k_{\tildeU} (\omega)_{(g,p)} = \theta^k_U(f\varrho)_{(g,p)} = f(q,p) \varrho_{(q,p)} =0 \quad 
   \text{for all } (g,p)\in U  \ . 
  \]
  Now assume $\omega = d_\textup{rel}f \wedge \varrho$ with $f$ as before and $\varrho  \in  \Gamma^\infty (\tildeU,\bigwedge^{k-1} s^* T^*M)$. 
  To prove that $ \theta^k_U(\omega) =0 $ it suffices to show that $ \iota^*_{U^g_g} \omega =0$ for all $g \in \operatorname{pr}_G(U)$. 
  Fix some $g \in \operatorname{pr}_G(U)$ and $p \in U^g_g$ and choose an open coordinate neighborhood $V \subset M$ with coordinates 
  $(x_1, \ldots , x_d): V \hookrightarrow \R^d$ such that $V \subset U_g$,
  $({x_1}_{|V^g}, \ldots , {x_k}_{|V^g}): V^g \hookrightarrow \R^k$ is  a local coordinate system of $M^g$ over $V^g$ and such that 
  $V^g$ is the zero locus of the coordinate functions $(x_{k+1},\ldots , x_d): V \hookrightarrow \R^{d-k}$.  
  After possibly shrinking $V$ there exists an open neighborhood $O$ of $g$ in $G$ such that $O \times V \subset \tildeU$. 
  Extend the coordinate functions $(x_1,\ldots,x_d)$ to smooth functions on $O\times V$ constant along the fibers 
  of the source map. 
  Then we have $d_\textup{rel}f = \sum_{l=1}^d \frac{\partial f}{\partial x_l} dx_l $. Since $\frac{\partial f}{\partial x_l} (g,p) =0$
  for $p\in V^g$ and $1 \leq l \leq k$ and since $\iota^*_{V^g} dx_l=0$ for  $k < l \leq d$ one gets
  \[
    \iota^*_{V^g}\iota^*_{U^g_g} \omega =  \iota^*_{V^g} \big( d_\textup{rel}f \wedge \varrho \big)  = 
    \sum_{l=1}^d  \Big( \iota^*_{V^g} \frac{\partial f}{\partial x_l}\Big) \, \big(\iota^*_{V^g} dx_l \big) \wedge \big(\iota^*_{V^g} \varrho \big) 
    = 0 \ ,
  \]
  where, by slight abuse of notation, we have also used the symbol $\iota_{V^g}$ for the embedding
  $V^g \hookrightarrow U$, $p \mapsto (g,p)$. 
  So $\iota^*_{U^g_g} \omega = 0$ and $\calK(\tildeU)$ is in the kernel
  of $\theta^k_{\tildeU}$. Hence $\theta^k_{\tildeU}$ factors through some linear map
  \[ \Theta^k_U: \rOmega{k}{\Lambda_0}(U) \to \Gamma^\infty (U,{\bigwedge}^k F) \ . \]
This proves the claim.
 \end{proof}
\begin{remark}
  Conjecturally, the morphism $\Theta^k$ is an isomorphism, showing that the sheaf theoretic approach (A)
  and the differential geometric approach (B) above leads to the same definition of basic relative forms.
  Below, in Section \ref{sec:circleaction}, we prove this conjecture for the case of an $S^1$-action.
  In the general case this conjecture remains open.

  Note that the image of the sheaf of horizontal relative $k$-forms under $\Theta^k$ coincides exactly
  with those families of forms  $(\omega_g)_{g\in \operatorname{pr}_G(U)}$ fulfilling condition
  \eqref{ite:horizontal} above. Since $G$ naturally acts on the generalized distribution $F$ and
  $\Theta^k$ is obviously equivariant by construction, the original conditions by Brylinski are recovered
  now also in the differential geometric picture of relative forms. 
\end{remark}
\begin{remark}
\label{rem:block-getzler}  
In \cite{BloGetECHED}, Block and Getzler define a sheaf on $G$ whose stalk at $g\in G$ is given by the space of $G_g$-equivariant 
differential forms on $M^g$. There are two differentials on this sheaf, $d$ and $\iota$, together constituting the equivariant differential 
$D:=d+\iota$, which, under an HKR-type map correspond to the Hochschild and cyclic differential on the crossed product algebra 
$G\ltimes C^\infty(M)$. Taking cohomology with respect to $\iota$ only leads to a very similar definition of basic relative forms as above, 
however notice that the basic relative forms defined above form a sheaf over the quotient $M\slash G$, not the group $G$. 
\end{remark}


%
%
\section{The group action case}\label{sec:actioncase}
In this section  we consider the action of a compact Lie group $G$ on a complete bornological algebra $A$
and then specialize to the case where $A$ is the algebra of smooth functions
on a smooth $G$-manifold $M$. The general assumption hereby is always that the action
$\alpha: G\times A$, $(g,a)\mapsto g\cdot a$ is smooth in the sense of \cite{KriMicCSGA} that is if each smooth curve in
$G\times A$ is mapped by $\alpha$ to a smooth curve in $A$. 
This is automatically guaranteed when $G$ acts by diffeomorphisms on the manifold $M$ and $A=\calC^\infty (M)$.
Under the assumptions made the associated \emph{smooth crossed product} $G\ltimes A$ is given by $\calC^\infty(G,A)$
equipped with the product
\begin{equation}
  \label{eq:convolution-product-left-action-algebra}
    (f_1\ast f_2)(g):=\int_Gf_1(h) \, (h\cdot f_2(h^{-1}g)) \,  dh \ ,\quad
    f_1,f_2\in \calC^\infty(G,A), \: g \in G \ .
\end{equation}
\subsection{The equivariant Hochschild complex}
To compute the Hochschild homology of the smooth crossed product $G\ltimes A$,
consider the bigraded vector space 
\[
  C=
  \bigoplus_{p,q\geq 0} C_{p,q},\quad
  \text{with} \quad C_{p,q}:=\calC^\infty(G^{(p+1)},A^{\otimes(q+1)}).
\]
There exists a bi-simplicial structure on $C$ given by face maps
$\delta_i^v:C_{p,q}\to C_{p,q-1}$, $0\leq i\leq q$ and
$\delta_j^h:C_{p,q}\to C_{p-1,q}$, $0\leq j\leq p$ defined as follows.
The vertical maps are given by 
\[
  \delta_i^v(F)(g_0,\ldots,g_p):=
  \begin{cases} b_i(F( g_0,\ldots,g_p))&\text{for  } 0\leq i\leq q-1,\\
    b^{(g_0\cdots g_p)^{-1}}_q(F(g_0,\ldots,g_p))&\text{for  } i\leq q,
  \end{cases}
\]
where the $b_i$ for $0\leq i\leq q-1$ are the first $q-1$ simplicial maps multiplying the
$i$'th and $i+1$'th entry in $A^{\otimes(q+1)}$ underlying the Hochschild chain complex of $A$,
and $b_q^g$ is the $g$-twisted version of the last one:
\[
  b^g_q(a_0\otimes\ldots\otimes a_q):=(g\cdot a_q)a_0\otimes a_1\otimes\ldots\otimes a_{q-1}\ ,
  \quad a_0,\ldots ,a_q \in A ,  \: g\in G \ .
\]
The horizontal maps are defined by
\[
  \delta_j^h(F)(g_0,\ldots,g_{p-1}):=
  \begin{cases}\int_GF(g_0,\ldots, h,h^{-1}g_j,\ldots g_{p-1})\, dh&\text{for  }
    0\leq j\leq p-1,\\
    \int_G h\cdot F(h^{-1}g_0,g_1,\ldots,g_{p-1},h)\, dh & \text{for  }j=p,
  \end{cases}
\]
where, in the second line, $h$ acts diagonally on $A^{\otimes(q+1)}$.
The following observations now hold true. 

\paragraph{$(i)$} The diagonal complex $\operatorname{diag}(C_{\bullet,\bullet}):=\bigoplus_{k\geq 0}C_{k,k}$
equipped with the differential
\[ d_\textup{diag}:=\sum_i(-1)^i\delta^h_i\delta^v_i \]
is isomorphic to the Hochschild complex 
$C_k(G\ltimes A)=\calC^\infty\big( G^{(k+1)},A^{\otimes(k+1)}\big)$ of the smooth crossed product
algebra $G \ltimes A$ via the isomorphism
$\overline{\phantom{F}}: \operatorname{diag}(C_{\bullet,\bullet}) \to C_\bullet(G \ltimes A))$,
$F\mapsto \overline{F}$ defined by 
\begin{equation}
\label{iso-c}
\overline{F}(g_0,\ldots,g_k):=(g_k^{-1}\cdots g_0^{-1}\otimes g_k^{-1}\cdots g_1^{-1}\otimes\ldots\otimes g_k^{-1}) \cdot F(g_0,\ldots,g_k), \quad F\in C_{k,k} \ ,
\end{equation}
where the pre-factor on the right hand side acts componentwise via the action of $G$ on $A$.
\paragraph{$(ii)$} The vertical differential $\delta^v$ in the total complex is given by a twisted version of the standard Hochschild complex of the algebra $A$. The horizontal differential $\delta^h$ in the $q$-th row can be interpreted as 
the Hochschild differential of the convolution algebra $\calC^\infty(G)$ with values in the $G$-bimodule $\calC^\infty(G,A^{\otimes (q+1)})$ with bimodule structure
\[
  (g\cdot f)(h):=g(f(g^{-1}h)),\quad (f\cdot g)(h):=f(hg),\quad f\in \calC^\infty(G,A^{\otimes (k+1)}),
  \: g,h \in G \ .
\]
The homology of this complex is isomorphic to the group homology of $G$ with values in the adjoint
module $\calC^\infty(G,A^{\otimes (k+1)})_{\rm ad}$ given by $\calC^\infty(G,A^{\otimes (k+1)})$
equipped with the diagonal action:
\[
  H_\bullet \big( \calC^\infty(G),\calC^\infty(G,A^{\otimes (q+1)}) \big)\cong
  H_\bullet^{\rm diff}\big( G , \calC^\infty(G,A^{\otimes (q+1)})_{\rm ad} \big) \ .
\]  
Because $G$ is a compact Lie group, its group homology vanishes except for the zeroth degree:
\[
  H_k^{\rm diff}\big( G,\calC^\infty(G,A^{\otimes (k+1)})_{\rm ad} \big)=
  \begin{cases}
    \calC^\infty(G,A^{\otimes (k+1)})_{\rm ad}^{\rm inv} & \text{for } k=0 , \\
    0&\text{for } k >0.
  \end{cases}
\]
\paragraph{$(iii)$} Filtering the total complex by rows, we obtain a spectral sequence with $E^2$-terms
\[
  E^2_{0,q}\cong \calC^\infty(G,A^{\otimes(q+1)})^{\rm inv},\qquad E^2_{p,q}=0 \text{ for } p\geq 1.
\]
The spectral sequence therefore collapses and the cohomology of the total complex is computed by
the complex $C^G_\bullet (A):=\calC^\infty(G,A^{\otimes(\bullet+1)})^{\rm inv}$ equipped with
the twisted Hochschild differential
\[
  (b_{\rm tw}f)(g):=\sum_{i=0}^q(-1)^ib_i(f(g))+(-1)^{q+1}b_{q+1}^{g^{-1}}(f(g))\ ,
  \quad f \in \calC^\infty(G,A^{\otimes(q +1)}) , \: g \in G \ .
\]
This complex is called the {\em equivariant Hochschild complex} in \cite{BloGetECHED}.
\paragraph{$(iv)$} By the Eilenberg--Zilber theorem, the diagonal complex is quasi-isomorphic to the total complex ${\rm Tot}(C_{\bullet,\bullet})$ with $\delta_{\rm Tot}:=\delta^h+\delta^v$ where the horizontal and vertical differentials are given by the usual formulas 
$\delta^{h,v}:=\sum_i(-1)^i\delta_i^{h,v}$. There is an explicit formula for the map $EZ: {\rm diag}(C_{\bullet,\bullet})\to {\rm Tot}(C_{\bullet,\bullet})$ implementing this quasi-isomorphism.
Combining items $(i)-(iv)$ above we conclude that the following holds.

\begin{proposition}
\label{prop:qismequivariantcplx}
Given a complete bornological algebra $A$ with a smooth left $G$-action, the composition
\[
  \widetilde{\phantom{F}}: C_\bullet(G \ltimes A )\stackrel{\overline{\phantom{F}}}{\longrightarrow}
  {\rm diag}(C)_\bullet\stackrel{EZ}{\longrightarrow}{\rm Tot}(C_{\bullet,\bullet})\longrightarrow
  C^G_\bullet(A)
\]
is a quasi-isomorphism of complexes. The explicit formula is given by mapping a chain
$F\in C_k(\calC^\infty(G,A))$ to the equivariant Hochschild chain $\widetilde{F}\in C^G_{k}(A)$ defined by
\[
\widetilde{F}(g):=\int_{G^{k}}(g^{-1}h_1\cdots h_k\otimes 1\otimes h_1\otimes \ldots\otimes h_1\cdots h_{k-1})F(h_k^{-1}\cdots h_1^{-1}g,h_1,\ldots,h_k)dh_1\cdots dh_k.
\]
\end{proposition}

\begin{remark}\label{rem:qismequivariantcplx}
 %
 %
  This result has originally been proved by Brylinski in \cite{ BryAAGAH,BryCHET}.  Observe that a right $G$-action
  $\beta$ on an algebra $A$ can be changed to a left $G$-action $\alpha$ on an algebra $A$ by
  $\alpha(g)(a):=\beta(g^{-1})(a)$. Let $A^{\operatorname{op}}$ be the opposite algebra of $A$ and assume
  that $\beta$ defines a right $G$ action on $A^{\operatorname{op}}$. Use $A^{\operatorname{op}}\rtimes_\beta G$ to
  denote the (right) crossed product algebra defined by the right $G$ action on $A^{\operatorname{op}}$.
  Define a map  $\Phi: G{_\alpha \ltimes} A\to A^{\operatorname{op}}\rtimes_\beta G$ by $\Phi(f)(g):=f(g^{-1})$.
  One directly checks the following identity, 
 \[
  \Phi (f_1\ast_{G{_\alpha \ltimes} A} f_2)= \Phi (f_2)\ast_{A^{\operatorname{op}}\rtimes_\beta G}\Phi(f_1), 
 \] 
 and concludes that the map $\Phi$ induces an isomorphism of algebras 
\[
 G{_\alpha \ltimes} A \cong \big(A^{\operatorname{op}}\rtimes_\beta G\big)^{\operatorname{op}}. 
\]
Furthermore notice that for a general algebra $\mathfrak{A}$, the algebra $\mathfrak{A}\otimes \mathfrak{A}^{\operatorname{op}}$ is naturally isomorphic to $\mathfrak{A}^{\operatorname{op}}\otimes \mathfrak{A}$ and therefore $HH_\bullet(\mathfrak{A})\cong  HH_\bullet(\mathfrak{A}^{\operatorname{op}})$ since the corresponding Bar resolutions coincide. Applying this
observation to $\big(A^{\operatorname{op}}\rtimes_\beta G\big)^{\operatorname{op}}$, one concludes that 
\[
  HH_\bullet(G{_\alpha \ltimes} A )\cong HH_\bullet\big(A^{\operatorname{op}}\rtimes_\beta G\big),
\]
and that Proposition \ref{prop:qismequivariantcplx} holds also true  for a smooth right $G$-action on an algebra $A$
meaning that there is a quasi-isomorphism of chain complexes
\[
  \widehat{\phantom{F}}: C_\bullet( A\rtimes G ) \longrightarrow
  C^G_\bullet(A^{\operatorname{op}})  \ .
\]  
Note that for a right $G$-action the convolution product on $\calC^\infty (G,A)$ is given  
by
\begin{equation}
  \label{eq:convolution-product-right-action-algebra}
    (f_1\ast f_2)(g):=\int_G ( f_1(h) \cdot (h^{-1}g)) \, f_2(h^{-1}g) \,  dh \ ,\quad
    f_1,f_2\in \calC^\infty(G,A), \: g \in G \ .
\end{equation}
Throughout this paper, as it is more natural to have a left $G$-action on a manifold $M$, we will
work with a right $G$-action on $\calC^\infty(M)$.  
\end{remark}

\subsection{The $G$-manifold case}
\label{sec:group-manifold-case}
Let $M$ be a manifold endowed with a smooth left $G$-action. Denote by $X = M/G$ the space of $G$-orbits in $M$ and
by $\pi :  M \to X$ the canonical projection. We consider the action groupoid
$\sfG = G\ltimes M\rightrightarrows M$ and the corresponding 
convolution sheaf $\calA = \calA_{G\ltimes M}$ over $X$.
It is straightforward to check that in the case of $A=\calC^\infty (M)$ the product defined by
Eq.~\eqref{eq:convolution-product-right-action-algebra} coincides with the convolution product on
$\calA(M/G) \cong \calC^\infty ( G\ltimes M) \cong  \calC^\infty ( G,A) $ given by
Eq.~\eqref{eq:convolution-product}. Hence $\calA(M/G)$ coincides with $A \rtimes G$.
According to Proposition \ref{prop:qismequivariantcplx} and Remark \ref{rem:qismequivariantcplx},
we then have for each $G$-invariant open $V \subset M$ a quasi-isomorphism between Hochschild chain complexes
\[
  \widehat{\phantom{F}}|_{V/G} : C_\bullet \big( \calA ( V/G ) \big) \to
  C_\bullet^G \big(\calC^\infty (V)\big) \cong
  C_\bullet \big(\calC^\infty (V), \calA ( V/G )\big) \ .
\]
To compute the Hochschild homology $HH_\bullet \big( \calA ( V/G ) \big)$
it therefore suffices to determine the homology of the complex
$C_\bullet \big(\calC^\infty (V), \calA ( V/G )\big)$ which we will consider
in the following. Recall that $\calA (V/G)$ is isomorphic as a bornological
vector space to the completed tensor product
$\calC^\infty (G) \hatotimes \calC^\infty (V)$ and that $\calA (V/G)$
carries the (twisted) $\calC^\infty (V)$-bimodule structure  
\[
  \calC^\infty (V) \hatotimes \calA(V/G) \hatotimes  \calC^\infty (V) \to \calA(V/G), \:
  f \otimes a \otimes f' \mapsto \Big( G \times V \ni (g,v) \mapsto f(g v) \, a(g,v) f'(v) \in \R \Big).
\]
Since the bimodule structure is compatible with restrictions $r^U_V$
for $G$-invariant open subsets $V \subset U \subset M$ 
one obtains a complex of presheaves  which assigns to every open $V/G$ with $V \subset M$
open and $G$-invariant the complex $C_\bullet (\calC^\infty (V),\calA (V/G))$.
Sheafification gives rise to a sheaf complex which we denote by
$\sheafC_\bullet  \big( \calC^\infty_M , \calA  \big)$.
Since
$C_\bullet  \big( \calC^\infty (V) , \calA (V/G)\big)\cong\calA (V/G) \hatotimes C_\bullet (\calC^\infty (V))$
for all $G$-invariant open $V\subset M$, this sheaf complex can be written as
\[
 \sheafC_\bullet  \big( \calC^\infty_M , \calA\big) = 
 \calA \hatotimes  \pi_* \sheafC_\bullet  \big( \calC^\infty_M \big) \ ,
\]
where, as before, $\sheafC_\bullet  \big( \calC^\infty_M \big)$ denotes   the Hochschild sheaf complex of
$\calC^\infty_M$. We now have the following result.

\begin{proposition}
  Assume to be given a $G$-manifold $M$, let $\calA$ be the convolution sheaf
  of the associated action groupoid $G\ltimes M \rightrightarrows M$ on the
  orbit space $X =M/G$, and put  $A = \calA(X)$. Then the chain map
  \[
  \varrho : C_\bullet \big(\calC^\infty (M), A \big) \rightarrow
  \Gamma \big( X, \sheafC_\bullet  ( \calC^\infty_M , \calA)\big), \enspace
  c \mapsto ([c]_\smorbO)_{\smorbO\in X}
  \]
  which associates to every $k$-chain $c\in  C_k\big( \calC^\infty (M), A \big) $
  the section $([c]_\smorbO)_{\smorbO\in X}$, where $[c]_\smorbO$ denotes the
  germ of $c$ in the stalk  $\sheafC_{\bullet,\smorbO}(\calC^\infty_M ,\calA)$,
  is a quasi-isomorphism.  
\end{proposition}
\begin{proof}
  Observe that the sheaves $\sheafC_k  ( \calC^\infty_M , \calA)$ are
  fine and that
  $\varrho_0 : C_0(\calC^\infty (M),A)\to\Gamma \big( X, \sheafC_0 (\calC^\infty_M ,\calA)\big)$
  is the identity morphism. Using again the homotopies from
  Section \ref{sec:localization-homotopies}, the proof is completely analogous
  to the one of Proposition \ref{prop:isomorphism-global-section-space},
  hence we skip the details. 
\end{proof}

Next, we compare the sheaf complex
$\sheafC_\bullet  ( \calC^\infty_M , \calA)\big)$ with the complex of relative
forms by constructing a morphism of sheaf complexes between them.

\begin{proposition}
\label{prop:chain-map-horizontal-relative-forms} 
  Under the assumptions of the preceding proposition
  define for each open $G$-invariant subset $V \subset M$ and $k\in \N$
  a $\calC^\infty (V/G)$-module map by 
  \[
  \begin{split}  
    \Phi_{k,V/G} : \: C_k  \big( \calC^\infty (V) , \calA (V/G)\big)
    \cong \calA(V/G) \hatotimes \, C_k \big( \calC^\infty (V) \big)
    & \to \rOmega{k}{\Lambda_0} \big( \Lambda_0 (G\ltimes V) \big), \\
     f_0 \otimes f_1 \otimes \ldots \otimes f_k & \mapsto 
    \big[ f_0 \, d (s_{G\ltimes V}^*  f_1 )\wedge  \ldots \wedge  d  (s_{G\ltimes V}^* f_k ) \big]_{\Lambda_0} \ .  
  \end{split}
  \]
  Then the $\Phi_{k,V/G}$ are the components of a morphism of sheaf complexes
  \[
    \Phi_\bullet: \sheafC_\bullet  ( \calC^\infty_M , \calA)\big)
    \to \pi_* (s_{|\Lambda_0})_* \rOmega{\bullet}{\Lambda_0} \ ,
  \] 
  where the differential on $\rOmega{\bullet}{\Lambda_0}$ is given by the
  zero differential. The image of a cycle under $\Phi_\bullet$  lies in the sheaf complex of
  horizontal relative forms $\hrOmega{\bullet}{\Lambda_0}$. 
\end{proposition}

\begin{proof}
  Let $f_0 \in \calA (V/G)$ and $f_1 , \ldots ,f_k \in \calC^\infty (V)$. Observe first that  
$\Phi_{k,V/G} ( f_0 \otimes f_1 \otimes \ldots \otimes f_k )$ is a relative form indeed since 
$d(s_{G\ltimes V}^*f) \in \Omega^1_{G\ltimes V \to G} (G\ltimes M)$ for all $f \in \calC^\infty (V)$.
Now let $(g,p) \in \Lambda_0(G\ltimes V)$ and compute:
\begin{displaymath}
\begin{split}
  & \Phi_{k-1,V/G}  \, b \, ( f_0 \otimes f_1 \otimes \ldots \otimes f_k ) (g,p) = 
  f_0 (g,p) f_1 (p) \,\big[ d(s_{G\ltimes V}^* f_2) \wedge \ldots \wedge d (s_{G\ltimes V}^* f_k)\big]_{(g,p)} +\\ 
  & \: + \sum_{i=1}^{k-1} (-1)^i f_0 (g,p) f_i(p) \big[ d(s_{G\ltimes V}^* f_1) \wedge \ldots \wedge d (s_{G\ltimes V}^* f_{i-1} ) \wedge d (s_{G\ltimes V}^* f_{i+1} ) \wedge   \ldots \wedge d (s_{G\ltimes V}^* f_k) \big]_{(g,p)}  + \\
  & \: + \sum_{i=1}^{k-1} (-1)^i f_0 (g,p) f_{i+1}(p) \big[ d(s_{G\ltimes V}^* f_1) \wedge \ldots \wedge d (s_{G\ltimes V}^* f_i ) \wedge d (s_{G\ltimes V}^* f_{i+2} ) \wedge  \ldots \wedge d(s_{G\ltimes V}^* f_k) \big]_{(g,p)}  + \\
  & \: + (-1)^k f_k (gp) f_0 (g,p) \, \big[ d(s_{G\ltimes V}^* f_1) \wedge \ldots \wedge d (s_{G\ltimes V}^* f_{k-1})\big]_{(g,p)}
  = 0 \ .
\end{split}
\end{displaymath}
Hence $\Phi_{\bullet,V/G}$ is a chain map in the sense that it intertwines the Hochschild boundary with the zero differential. 

It remains to show that the image of $\Phi_{\bullet,V/G}$ is in the space of horizontal relative forms. 
To this end assume for a moment that $V$ is a $G$-invariant open ball around the origin in some euclidean space $\R^n$ 
which is assumed to carry an orthogonal $G$-action. Consider the Connes--Koszul resolution of $\calC^\infty(V)$ provided in \eqref{eq:ResSmoothFunc}. 
A chain map   between the Connes--Koszul resolution and 
the Bar resolution of $\calC^\infty(V)$ over the identity map $\id_{\calC^\infty(V)}$ in degree $0$ is given by the family of maps 
\[
  \begin{split}
  \Psi_{k,V} : \: \Gamma^\infty (V\times V, E_k) & \to B_k \big( \calC^\infty(V) \big) = \calC^\infty(V\times V) \hatotimes \calC^\infty(V^k),\\
   \omega & \mapsto \Big( (v,w,x_1,\ldots,x_k) \mapsto \omega_{(v,w)}\big(Y(x_1,w), \ldots , Y(x_k,w)  \big) \Big) \ .
  \end{split}
\]
Tensoring the Connes--Koszul resolution of $\calC^\infty (V)$ 
with $\calA^\infty (V/G)$ results in the following complex:
\begin{equation}
  \label{eq:ConvolutionConnesKoszulChainCpl}
       \Omega^d_{G\ltimes V \to G} (V)  \overset{i_{Y_{G\ltimes V}}}{\longrightarrow} \ldots 
       \overset{i_{Y_{G\ltimes V}}}{\longrightarrow}   \Omega^1_{G\ltimes V \to G} (V)
       \overset{i_{Y_{G\ltimes V}}}{\longrightarrow}  \calC^\infty (G\ltimes V) \longrightarrow 0 \ ,
\end{equation} 
where $Y_{G\ltimes V} : G\ltimes V \to  s^*TV$ is defined by $Y_{G\ltimes V}(g,v) =v -gv$. The composition of 
$\id_{\calA^\infty (V/G)} \hatotimes \Psi_{k,V} $ with $\Phi_{k,V/G}$ then is the map which associates 
to each relative form $\omega \in  \Omega^k_{G \ltimes V \to G} (V)$ its restriction $[\omega]_{\Lambda_0}$ 
to the loop space. It therefore suffices to show that for $\omega \in  \Omega^k_{G \ltimes V \to G} (V)$ 
with $i_{Y_{G\ltimes V}} \omega =0$ the restriction to the loop space is a horizontal relative form. 
To verify this let $\xi$ be an element of the Lie algebra $\frakg$ of $G$ and again $(g,v)\in \Lambda_0(G\ltimes V)$. 
Then 
\[
  0 = \left. \frac{d}{dt} \big( i_{Y_{G\ltimes V}} \omega\big)_{(e^{-t\xi} g , v)} \right|_{t=0} = 
   \big( - i_{Y_{G\ltimes V}} i_{\xi_G} d^G\omega + i_{\xi_V} \omega \big)_{(g,v)} = \big( i_{\xi_V} \omega \big)_{(g,v)} \ ,
\]
where $d^G$ denotes the exterior differential with respect to $G$ and $\xi_G$ and $\xi_V$ are the fundamental vector fields
of $\xi$ on $G$ and $V$, respectively. So  $ i_{\xi_V} \omega \in \calJ(V) \Omega^{k-1}_{G \ltimes V \to G} (V)$,
which means that $[\omega]_{\Lambda_0} \in \hrOmega{k}{\Lambda_0}(G\ltimes V)$.

\end{proof}

\begin{proposition}\label{prop:manifold-one-isotropy-type}
  Let $M$ be a $G$-manifold with only one isotropy type and assume that the orbit space $M/G$ is connected.
  Then the following holds true.
  \begin{enumerate}[{\rm (1)}]
  \item\label{ite:quotient-manifold}
    The quotient space $M/G$ carries a unique structure of a smooth manifold 
    such that $\pi : M \to M/G$ is a submersion.
  \item\label{ite:loop-space-manifold}
    The loop space $\Lambda_0 (G\ltimes M)$ is a smooth submanifold of $G\times M$.
  \item\label{ite:isomorphism-basic-relative-forms-slice}
    Let $p\in M$  be a point and $V_p \subset M$ a slice to the orbit through $p$ that is
    \begin{enumerate}[\rm (SL1)]
    \item
      $V_p$ is a $G_p$-invariant submanifold which is transverse to the orbit $\orbO_p  := Gp$ at $p$,
    \item
      $V := GV_p$ is an open neighborhood of the orbit $\orbO_p$ and $V_p$ is closed in $V$,
    \item
      there exists a $G$-equivariant diffeomorphism $\eta: N\orbO_p \to V$
      mapping the normal space $N_p= T_pM/T_p\orbO_p$ onto $V_p$. 
    \end{enumerate}
    Then for every $k$  the map 
    \[
      \Psi_{k,V_p/G_p}:\brOmega{k}{\Lambda_0}\big(\Lambda_0(G\ltimes GV_p)\big)
      \to \brOmega{k}{\Lambda_0} \big(\Lambda_0 (G_p\ltimes V_p)\big)
      \quad \omega \mapsto \omega_{|\Lambda_0(G_p\ltimes V_p)} 
    \]  
    is an isomorphism and the space of basic relative $k$-forms    
    $\brOmega{k}{\Lambda_0} \big(\Lambda_0 (G_p\ltimes V_p)\big)$ coincides naturally with
    $\calC^\infty(G_p)^{G_p} \hatotimes \Omega^k (V_p)$.
  \item\label{ite:quism-twisted-hochschild-homology}
    The chain map 
    \[
      \Phi_{\bullet,M/G} : C_\bullet  \big( \calC^\infty (M) , \calA (M/G)\big)
      \to \hrOmega{\bullet}{\Lambda_0} \big(\Lambda_0 (G\ltimes M)\big)
    \]  
    is a quasi-isomorphism when the graded module
    $\hrOmega{\bullet}{\Lambda_0} \big(\Lambda_0 (G\ltimes M)\big)$ is endowed with the zero differential.  
  \end{enumerate}
\end{proposition}

\begin{proof}
\textit{ad} (\ref{ite:quotient-manifold}).
    It is a well known result about group actions on manifolds that under the assumptions made the
    quotient space $M/G$ carries a unique manifold structure such that
    $\pi : M \to M/G$ is a submersion; see e.g.\ \cite[Sec.~IV.3]{BreICTG} or \cite[Thm.~4.3.10]{PflAGSSS}.
\\ \textit{ad} (\ref{ite:loop-space-manifold}).
    This has been proved in \cite[Prop.~4.4]{FarPflSeaSISCLGA}. Let us outline the
    argument since we need it for the following claims, too. 
    By the assumptions made there exists a compact subgroup $K\subset G$ such that
    every point of $M$ has isotropy type $(K)$. Let $p\in M$ be a point and $G_p$
    its isotropy group. Without loss of generality we can assume that $G_p=K$. 
    Let $V_p\subset M $ be a slice to the orbit $\orbO$ through $p$. The isotropy group
    of an element $q\in V_p$ then has to coincide with $K$, so $V_p^K = V_p$.
    Therefore the map
    \[
        \tau : G/K\times V_p \to M ,\quad (gK,q) \mapsto gq
    \]
    is a $G$-equivariant diffeomorphism onto a neighborhood of $\orbO$.
    Now choose a small enough open neigborhood of $eK$ in $G/K$ and
    a smooth section $\sigma : U \to G$ of the fiber bundle $G\to G/K$.
    The map
    \[
      \widetilde{\tau} : G \times U \times V_p \to G \times \tau(U\times V_p) ,
      \quad (h, gK,q) \mapsto\big(\sigma(gK)h \sigma(gK)^{-1},\sigma (gK) q \big)
    \]
    then is a diffeomorphism onto the open set
    $G\times \tau (U\times V_p)$ of $G\times M$. One observes that
    \[
      \widetilde{\tau} (  K \times U \times V_p)  = \left( G \times \tau(U\times V_p) \right) \cap \Lambda_0 (G\ltimes M) \ ,
    \]
    which shows that $ \Lambda_0 (G\ltimes M)$ is a submanifold of $G\times M$, indeed. 
\\ \textit{ad} (\ref{ite:isomorphism-basic-relative-forms-slice}).
    Put $K=G_p$ as before, let $N = GV_p$, and denote by $\frakg$ and $\frakk$ the Lie
    algebras of $G$ and $K$, respectively. Choose an $\Ad$-invariant inner product 
    on $\frakg$ and let $\frakm$ be the orthogonal complement of $\frakk$ in $\frakg$.
    Next choose for each $q\in N$ an element $h_q \in G$ such that $h_q q \in V_p$. Then
    \[
      \pi^N : N\to \orbO_p , \quad q \mapsto h_q^{-1}p
    \]
    is an equivariant fiber bundle. Let $TN\to N$ be the tangent
    bundle of the total space and $VN \to N$ the vertical bundle. Note that $TN$ and $VN$
    inherit from $N$ the equivariant bundle structures.  
    Now put for $q\in N$
    \[
      H_qN := \operatorname{span} \left\{ \left(\Ad_{h_q^{-1}} (\xi)\right)_N (q) \in T_qN \mid
      \xi \in \frakm\right\} \ ,
    \]
    where $\xi_N$ denotes the fundamental vector field of $\xi$ on  $N$. 
    Then $HN \to N$ becomes an equivariant vector bundle complementary to $VN \to N$.
    Let $P^{\textup{v}}: TN \to VN$ be the corresponding fiberwise projection along $HN$.
    By construction, $P^{\textup{v}}$ is $G$-equivariant.     
    After these preliminary considerations let 
    $\omega \in \brOmega{k}{\Lambda_0}\big(\Lambda_0 (G\ltimes GV_p)\big)$.
    The restriction $\omega|_{\Lambda_0 (K\ltimes V_p)}$ then is a basic relative form again,
    so $\Psi_{k,V_p/K}$ is well defined. Let us show that it is surjective.
    Assume that $\varrho \in  \brOmega{k}{\Lambda_0}\big(\Lambda_0 (K \ltimes V_p)\big)$.
    We then put for  $(g,q) \in  \Lambda_0 (G \ltimes N)$ and $X_1,\ldots ,X_k \in T_q N$
    \begin{equation}
      \label{eq:relation-relform-restriction}
      \omega_{(g,q)} (X_1,\ldots ,X_k) :=
      \varrho_{(h_qgh_q^{-1},h_qq)} \big(  Th_q (P^{\textup{v}} (X_1)),\ldots ,  Th_q (P^{\textup{v}} (X_k))\big) \ ,   
    \end{equation}
    where $Th : TN \to TN$ for $h\in G$ denotes the derivative of the action of $h$ on $N$.
    Since $Tk$ for $k\in K$ acts as identity on $TV_p\subset VN$, the value $\omega_{(g,q)} (X_1,\ldots ,X_k)$
    does not depend on the particular choice of a group element $h_q$ such that
    $h_qq \in V_p$.     
    Moreover, since for fixed $q_0 \in N$ one can find a small enough neighborhood $U$ 
    and choose $h_q$ to depend smoothly on  $q\in U$, $\omega$ is actually a smooth
    differential form on $N$. By construction, it is a relative form. If $X_l \in H_qN$
    for some $l$, then $\omega_{(g,q)} (X_1,\ldots ,X_k)=0$ by definition.
    If $X_l = \big(\Ad_{h_q^{-1}} (\xi)\big)_N  (q)$ for some $\xi \in \frakk$, then $P^{\textup{v}} X_l = X_l$ and
    $Th_q X_l (q) =\xi_N (h_q q)$ which entails by \eqref{eq:relation-relform-restriction}
    that $\omega_{(g,q)} (X_1,\ldots ,X_k)=0$ again since $\varrho $ is a horizontal form. 
    So $\omega$ is a horizontal form. It remains to show that it is $G$-invariant. Let $h\in G$
    and  $(g,q)$ and $X_1,\ldots ,X_k$ as before. Then
    \begin{equation*}
      \begin{split}     
      \omega_{(hgh^{-1},hq)} (Th X_1,\ldots ,Th X_k) \, & =
      \varrho_{(h_qg h_q^{-1},h_qq)} \big(  Th_q Th^{-1} (P^{\textup{v}} (Th X_1)),\ldots ,  Th_q Th^{-1} (P^{\textup{v}} (Th X_k))\big)
      \\ & =  \omega_{(g,q)} (X_1,\ldots ,X_k) \ ,
      \end{split}
    \end{equation*}
    so $\omega$ is $G$-invariant and therefore a basic relative form. Hence $\Psi_{k,V_p/K}$ is surjective.
    To prove injectivity of $\Psi_{k,V_p/K}$ observe that if $\omega\in\brOmega{k}{\Lambda_0}\big(\Lambda_0(G\ltimes GV_p)\big)$
    and $\varrho$ is the restriction $\omega|_{\Lambda_0 (K\ltimes V_p)}$,
    then Eqn.~\eqref{eq:relation-relform-restriction} holds true since $\omega$ is $G$-invariant and horizontal.   
    But this implies that if $\omega|_{\Lambda_0 (K\ltimes V_p)}=0$, then $\omega$ must  be $0$ as well, so 
    $\Psi_{k,V_p/K}$ is injective.  
    It remains to show
    \[
      \brOmega{k}{\Lambda_0} \big(\Lambda_0 (K \ltimes V_p)\big)\cong \calC^\infty(K)^{K} \hatotimes \Omega^k (V_p) \ .
    \]
    To this end observe that $\Lambda_0 (K \ltimes V_p) = K \times V_p$ since $V_p^K =V_p$ which in other words
    means that very $K$-orbit in $V_p$ is a singleton. The claim now follows immediately.  
    \\ \textit{ad} (\ref{ite:quism-twisted-hochschild-homology}).
    By Theorem \ref{thm:hochschild-homology-global-sections} it suffices to verify the claim for the case where
    $M =GV_p$, where $p$ is a point and $V_p$ a slice to the orbit $\orbO$ through $p$. As before let $K$ be the isotropy
    $G_p$. By the slice theorem there exists a $K$-equivariant diffeomorphism
    $\varphi : V_p \to \widetilde{V}_p\subset N_p\orbO$ onto an open zero neighborhood of the normal space $N_p\orbO$.
    Choose a $K$-invariant inner product on  $N_p\orbO$ and a $G$-invariant inner product on the Lie algebra $\frakg$.
    Again as before let $\frakm$ be the orthogonal complement of the Lie algebra $\frakk$ in $\frakg$.
    The inner product on $\frakg$ induces a $G$-invariant riemannian metric on $G$ which then induces a
    $G$-invariant riemannian metric on the homogeneous space $G/K$ by the requirement that $G \to G/K$ is a
    riemannian submersion. Now observe that the map $G/K \times V_p \to M$, $(gK,v) \mapsto gv$  is a $G$-invariant
    diffeomorphism, so we can identify $M$ with $G/K \times V_p$. The chosen riemannian metrics on $G/K$ and $V_p$ then
    induce a $G$-invariant metric on $M$. Since $\C$ is faithfully flat over $\R$ we can assume without loss of generality now
    that smooth functions and forms on $M$ and $G\ltimes M$ are all complex valued, including elements of the convolution algebra.
    Let $e\in N_p\orbO \cong T_pV_p$ be a vector of unit length, and let $Z$ be the vector field on $M$ which maps every point to
    $e$ (along the canonical parallel transport). Next choose a symmetric open neighborhood $U$ of the diagonal of $G/K \times G/K$
    such that for each pair $(gK,hK) \in U$ there is a unique $\xi \in \Ad_h(\frakm)$ such that
    $gK = \exp(\xi)hK$. Denote that $\xi$ by $\exp_{hK}^{-1} (gK)$.
    Let $\chi : G/K \times G/K \to [0,1]$ be  a function with support
    contained in $U$ and such that $\chi =1 $ on a neighborhood of the diagonal.
    Now define the vector field $Y :M\times M \to \operatorname{pr}^*_2 (TM)$ by
    \[
      Y\big( (gK,v),(hK,w) \big) = \chi (gK,hK) \, \left( \exp_{hK}^{-1} (gK), v-w\right) + \sqrt{-1} \chi' (gK,hK) \, Z\big( (gK,v),(hK,w) \big) \ ,
    \]
    where $ \operatorname{pr}_2:M\times M \to M$ is projection onto the second coordinate
    and where the smooth cut-off function $\chi' :  G/K \times G/K \to [0,1]$ vanishes on a neighborhood of the diagonal and
    is identical $1$ on the locus where $\chi \neq 1$. 
    Finally put $E_k:= \operatorname{pr}^*_2 (\bigwedge^k T^*M)$. Then, by \cite[Lemma 44]{ConNDG},
    the complex 
    \[
      \Gamma^\infty (M\times M ,E_{\dim M}) \stackrel{i_Y}{\longrightarrow} \ldots  \stackrel{i_Y}{\longrightarrow}\Gamma^\infty (M\times M ,E_1) \stackrel{i_Y}{\longrightarrow} \calC^\infty (M \times M )\longrightarrow \calC^\infty (M)
    \]
    is a (topologically) projective resolution  of   $\calC^\infty (M)$ as a $\calC^\infty (M)$-bimodule. 
    Tensoring this resolution with the convolution algebra $\calA (G\ltimes M)$ gives the following complex
    of relative forms:
    \begin{equation}
      \label{eq:quasiisomorphic-equivariant-chain-cplxc}
      \Omega^{\dim M}_{G\ltimes M\to G} (G\ltimes M) \stackrel{i_{Y_G}}{\longrightarrow} \ldots  \stackrel{i_{Y_G}}{\longrightarrow}
      \Omega^1_{G\ltimes M\to G} (G\ltimes M)  \longrightarrow \calC^\infty (G\ltimes M) \ , 
    \end{equation}
    where $Y_G: G \times M \to pr^*_2 TM $ is the vector field
    \[
      (g,(hK,v))\mapsto \chi (ghK,hK) \, \left( \exp_{hK}^{-1} (ghK), 0\right)+\sqrt{-1} \chi' (ghK,hK) \, Z\big( (ghK,v),(hK,v) \big)\ .
    \]
The vector field $Y_G$ vanishes on $(g,(hK,v))$ if and only if
$g \in hKh^{-1}$ that is if and only if  $(g,(hK,v))\in \Lambda_0(G\ltimes M)$.
We will use the parametric Koszul resolution
Proposition \ref{prop:parametrizedkoszuleulervectorfield} to show that the complex
    \eqref{eq:quasiisomorphic-equivariant-chain-cplxc} is quasi-isomorphic to the complex of
    horizontal relative forms
    \begin{equation}
       \label{eq:quasiisomorphic-horizontal-relative-chain-cplxc}
      \hrOmega{\dim M}{\Lambda_0}(\Lambda_0(G\ltimes M)) \stackrel{0}{\longrightarrow} \ldots  \stackrel{0}{\longrightarrow}
      \hrOmega{1}{\Lambda_0}(\Lambda_0(G\ltimes M)) \stackrel{0}{\longrightarrow} \calC^\infty (\Lambda_0(G\ltimes M)) \ .
    \end{equation}
    This will then entail the claim.
    So it remains to show that \eqref{eq:quasiisomorphic-equivariant-chain-cplxc} and
    \eqref{eq:quasiisomorphic-horizontal-relative-chain-cplxc} are quasi-isomorphic.
    We first consider the case where $V_p$ consist just of a point. Then $M$ coincides with the
    homogeneous space $G/K$ and $Y_G$  is an Euler-like vector field on its set of zeros
    \[ S = \{ (g,hK)\in G \times G/K \mid g \in hKh^{-1} \} \subset M \ . \]
    Note that $S$ is a submanifold on $M$. That $Y_G$ is Euler-like on $S$ indeed follows from the
    equality
    \[
      \left.\frac{d}{dt}\exp_{hK}^{-1}\big( \exp(t\xi) ghK \big) \right|_{t=0}  =
      \left.\frac{d}{dt}\exp_{hK}^{-1}\big( \exp(t\xi) hK \big) \right|_{t=0}  =
      \xi 
    \]
    for all $(g,hK)\in S$, $\xi \in \Ad_{gh} (\frakm) = \Ad_h (\frakm)$. Hence, by
    Proposition  \ref{prop:parametrizedkoszuleulervectorfield}, the complex 
    \begin{equation*}
      \Omega^{\dim G/K}_{G\ltimes G/K\to G} (G\ltimes G/K) \stackrel{i_{Y_G}}{\longrightarrow} \ldots  \stackrel{i_{Y_G}}{\longrightarrow}
      \Omega^1_{G\ltimes G/K \to G} (G\ltimes G/K)  \longrightarrow \calC^\infty (G\ltimes G/K)  
    \end{equation*}
    is quasi-isomorphic to
     \[  0 \longrightarrow  \ldots \longrightarrow 0\longrightarrow \calC^\infty (S) \ . \] 
    Since $\hrOmega{k}{\Lambda_0} (\Lambda_0 (G\ltimes G/K)) = 0$ for $k \geq 1$, the claim follows
    in the case  $V_p= \{ p\}$. Now consider the case $M = G/K \times V_p$ with $V_p$ an arbitrary manifold
    on which $K$ acts trivially. Observe that in this situation
    \[ \Omega^k_{G\ltimes M\to G} (G\ltimes M)\cong \bigoplus_{0\leq l \leq k}
      \Omega^l_{G\ltimes G/K \to G} (G\ltimes G/K)\hatotimes\Omega^{k-l} (V_p) \]
    and that $Y_G$ acts, near its zero set $S = \Lambda_0(G\ltimes M)$, only on the first components
    \[ \Omega^l_{G\ltimes G/K \to G} (G\ltimes G/K)\ .\]
    Hence the chain complex \eqref{eq:quasiisomorphic-equivariant-chain-cplxc}
    is then quasi-isomorphic to the chain complex
    \[
     \calC^\infty(\Lambda_0 (G\ltimes G/K )) \hatotimes \Omega^\bullet (V_p)
    \]
    with zero differntial. But since 
    \[ \hrOmega{k}{\Lambda_0} (\Lambda_0 (G\ltimes M)) \cong
       \calC^\infty(\Lambda_0 (G\ltimes G/K )) \hatotimes \Omega^k (V_p) \]
    the claim is now proved.  
\end{proof}

\begin{conjecture}[Brylinski {\cite[Prop.~3.4]{BryAAGAH}\& \cite[p.~24, Prop.]{BryCHET}}]\label{BryConjBHF}
  Let $M$ be $G$-manifold and regard $\hrOmega{\bullet}{\Lambda_0} \big(\Lambda_0 (G\ltimes M)\big)$
  as a chain complex endowed with the zero differential. Then the chain map 
   \[
      \Phi_{\bullet,M/G} : C_\bullet  \big( \calC^\infty (M) , \calA (M/G)\big)
      \to \hrOmega{\bullet}{\Lambda_0} \big(\Lambda_0 (G\ltimes M)\big)
  \]
  is a quasi-isomorphism.
\end{conjecture}

\begin{remark}
  Proposition \ref{prop:manifold-one-isotropy-type} shows that Brylinski's
  conjecture holds true for
  $G$-manifolds having only one isotropy type. Corollary \ref{cor:finitegroup}
  tells that Brylinski's conjecture is true for finite group actions.
  In the following section we will verify it for circle actions. 
\end{remark}


\section{The circle action case}
\label{sec:circleaction}
\subsection{Rotation in a plane}
Let us consider the case of the natural $S^1$-action on $\R^2$ by rotation. 
First we describe the ideal sheaf $\calJ \subset \calC^\infty_{S^1 \ltimes \R^2} $
which consists of smooth functions on open sets of $S^1\times \R^2$ vanishing on 
$\Lambda_0(S^1 \ltimes \R^2)$.  
To this end denote by $x_j : S^1 \times \R^2 \to \R$, $j=1,2$, the projection onto the first respectively second 
coordinate of $\R^2$ and by
$\tau :  S^1\setminus \{ - 1 \} \times \R^2  \to (-\pi,\pi)$ the coordinate map 
$(g,v) \mapsto \operatorname{Arg} (g)$. By  $r = \sqrt{x_1^2 + x_2^2}$ we denote the radial coordinate
and by $B_\varrho (v)$ the open disc of radius $\varrho >0$ around a point $v \in \R^2$.
Note that the loop space $\Lambda_0(S^1 \ltimes \R^2)$ is the disjoint union of the strata
$\{ (1,0) \}$, $ \{ 1\} \times (\R^2\setminus \{0\})$, and $( S^1\setminus \{1\}) \times \{0\}$
and that the loop space is smooth outside the singular point $(1,0)$. 

\begin{proposition}\label{prop:vanishingideal}
  Around the point $(1,0)$, the vanishing ideal $ \calJ \big( (S^1\setminus \{ -1 \})  \times B_\varrho(0) \big) $  
  consists of all smooth $f : (S^1\setminus \{ -1 \})  \times B_\varrho(0)   \to \R$ 
  which can be written in the form
\begin{equation}
  \label{eq:expansion-functions-vanishing-ideal}
  f = f_1 \tau x_1 + f_2 \tau x_2 , \quad \text{where } f_1,f_2 \in \calC^\infty\big( (S^1\setminus \{ -1 \}) 
  \times B_\varrho(0) \big) \ . 
\end{equation}
  Around the stratum $ \{ 1\} \times (\R^2\setminus \{0\})$, a function $f \in \calC^\infty \big( (S^1\setminus \{ -1 \}) \times (\R^2\setminus \{0\})\big)$
  lies in the ideal  $ \calJ \big( ( S^1\setminus \{ -1 \})  \times (\R^2\setminus \{0\})\big)$ if and only if 
  $f$ is of the form $h \tau $ for some  $h \in \calC^\infty \big( (S^1\setminus \{ -1 \})  \times (\R^2\setminus \{0\})\big)$.
  Finally, around the stratum $( S^1\setminus \{1\}) \times \{0\}$, a function
  $f \in \calC^\infty \big( (S^1 \setminus \{ 1\}) \times \R^2\big)$ vanishes on  $\Lambda_0(S^1 \ltimes \R^2)$
  if and only if it is of the form  $f_1 x_1 + f_2 x_2$ with  $f_1,f_2 \in \calC^\infty  \big( (S^1 \setminus \{ 1\}) \times \R^2\big)$. 
\end{proposition}

\begin{proof}
 Since the loop space is smooth at points of the strata  $ \{ 1\} \times (\R^2\setminus \{0\})$ and $( S^1\setminus \{1\}) \times \{0\}$,
 only the case where $f$ is defined on a neighborhood of the singular point $(1,0)$ is non-trivial. 
 So let us assume that  $ f \in \calC^\infty \big( (S^1\setminus \{ -1 \})  \times B_\varrho(0) \big) $   
 vanishes on  $\Lambda_0(S^1 \ltimes \R^2)$. 
 Using the coordinate functions we can consider $f$ as a function of $t \in (-\pi,\pi)$ and $x \in \R^2$.
 By the Malgrange preparation theorem one then has an expansion
 \[
   f(t, x) + t = c(t,x) (t + a_0 (x)) ,
 \]
where $c$ and $a_0$ are smooth and $a_0(0)=0$. Since $t = c(t,0)t$ for all $t\in (-\pi,\pi)$, one has 
$c (t,0)=1$. Putting $t=0$ gives $0= c(0,x)a_0(x)$ for all $x \in B_\varrho(0)$. Since $c(0,0)=1$, one obtains 
$a_0(x)= 0$ for all $x$ in a neighborhood of the origin. After possibly shrinking $B_\varrho (0)$ we can assume
that $a_0=0$. Hence 
\begin{equation}
  \label{eq:intermediate-expansion}
  f(t, x) = (c(t,x) - 1 ) t \ .
\end{equation}
Parametric Taylor expansion of $c(t,x)-1$ gives 
\[
  c(t,x) - 1 = x_1 r_1 (t,x) +  x_2 r_2(t,x), \quad \text{where }  
 r_j (t,x) = \int_0^1 (1-s) \, \partial_j c (t, sx) \, ds , \: j=1,2 \ . 
\] 
Since the functions $r_j$ are smooth, this expansion together with \eqref{eq:intermediate-expansion}
entails \eqref{eq:expansion-functions-vanishing-ideal}.
\end{proof}

\begin{lemma}\label{lem:coordinate-representations-vector-fields}
 The vector fields 
 \begin{equation*}
   \begin{split}
      Y =  Y_{S^1\ltimes \R^2} :\: S^1\times \R^2 \to \R^2 , (g,x) \mapsto x - g x \quad \text{and} \quad
      Z = Z_{S^1\ltimes \R^2} :\: S^1\times \R^2 \to \R^2 , (g,x) \mapsto \frac{x + gx}{2} 
   \end{split}
 \end{equation*}
 have coordinate representations  $Y  =  Y_1  \frac{\partial}{\partial x_1}  + Y_2  \frac{\partial}{\partial x_2}$
 and   $Z  =  Z_1  \frac{\partial}{\partial x_1}  + Z_2  \frac{\partial}{\partial x_2}$ with coefficients given by 
 \begin{equation}
 \label{eq:expansionY} 
 \begin{split}
     Y_1  = x_1 (1-\cos \tau) - x_2 \sin \tau 
     \quad \text {and} \quad 
     Y_2  = x_2 (1 -\cos \tau ) + x_1 \sin \tau 
 \end{split}
 \end{equation}
 respectively by 
 \begin{equation}
   \label{eq:expansionZ}
   \begin{split}
     Z_1   = x_1 (1+\cos \tau) + x_2 \sin \tau  \quad \text {and} \quad 
     Z_2   =  x_2 (1+\cos \tau )- x_1 \sin \tau \ .
   \end{split}
 \end{equation}
 Moreover, the vector fields $Y$ and $Z$ have square norms
 \begin{equation}
   \label{eq:squarenorms}
     \| Y \|^2 = 2 r^2 \, (1 - \cos \tau ) = r^2\tau^2 (\xi \circ\tau)  \quad
     \text{and} \quad 
     \| Z \|^2  = 2 r^2 \, (1+\cos \tau ) \ ,
 \end{equation}
 where $\xi$ is holomorphic with positive values over $(-\pi,\pi)$ and value $1$ at the origin. 
\end{lemma}
\begin{proof}
 The representations  
\begin{equation*}
\begin{split}
   Y|_{(S^1 \setminus \{ -1\})\times \R^2}  & 
    = \left( x_1 (1-\cos \tau) - x_2 \sin \tau \right) \frac{\partial}{\partial x_1} 
           + \left( x_2 (1 -\cos \tau ) + x_1 \sin \tau  \right) \frac{\partial}{\partial x_2}   \quad \text{and} \\
   Z|_{(S^1 \setminus \{ -1\})\times \R^2} &
    = \left( x_1 (1+\cos \tau) + x_2 \sin \tau  \right) \frac{\partial}{\partial x_1} 
           + \left( x_2 (1+\cos \tau )- x_1 \sin \tau \right) \frac{\partial}{\partial x_2}
\end{split}
\end{equation*}
are immediate by definition of $Y$ and $Z$ and since $S^1$ acts by rotation. 
Note that these formulas still hold true when extending $\tau$ to the whole circle by putting 
$\tau (-1) = \pi$. At $g=-1$ the extended $\tau$ is not continuous then, but compositions with 
the trigonometric functions $\cos$ and $\sin$ are smooth on $S^1$.
For the norms of $Y$ and $Z$ one now obtains
\[
  \| Y \|^2 = x_1^2 (1 -\cos \tau)^2  +  x_2^2  \sin^2 \tau \,
           +  x_2^2 (1-\cos \tau)^2  + x_1^2  \sin^2 \tau  = 2 r^2 \, (1 - \cos \tau ) 
\]
and 
\[
  \| Z \|^2 = x_1^2 (1+\cos \tau)^2  +  x_2^2  \sin^2 \tau 
           +  x_2^2 (1+\cos \tau)^2 + x_1^2  \sin^2 \tau  = 2 r^2 \, (1+\cos \tau ) \ .
\] 
By power series expansion of $1 -\cos t$ one obtains the statement about $\xi$. 
\end{proof}

\begin{lemma}
 For all open subsets $U$ of the loop space $\Lambda_0 = \Lambda_0 (S^1 \ltimes \R^2)$ and all $k\in \N$ the map 
 \[
   \Theta^k_U : \rOmega{k}{\Lambda_0}(U) \to \Gamma^\infty (U,\wedge^k F) 
 \]
 from Prop.~\ref{prop:factorization-relative-forms} is injective. 
\end{lemma}

\begin{proof}
Since $\rOmega{0}{\Lambda_0}(U) = \calC^\infty (U) = \Gamma^\infty (U,\wedge^0 F)$
and $\Theta^0_U = \id$, 
we only need to prove the claim for $k\geq 1$.   
To this end we have to show that for 
$\omega \in \Gamma^\infty (\widetilde{U}, \wedge^k s^* T^*M)$ with 
$[\omega]_F=0$ the relation $[\omega]_{\Lambda_0} =0$ holds true. 
Here, as before, $\widetilde{U}\subset S^1\times \R^2$ is an 
open subset such that $U = \widetilde{U} \cap \Lambda_0(S^1\ltimes \R^2)$.
In other words we have to show that each such $\omega$ has the form
\[
  \omega = \sum_{l \in L} f_l \omega_l +
  \sum_{j\in J} d_\textup{rel} h_j \wedge \eta_j \ ,
\]
where $L,J$ are finite index sets, $f_l,h_j \in \calJ (\widetilde{U})$, 
$\omega_l \in \Gamma^\infty (\widetilde{U},\wedge^k s^* T^*M)$, 
and $\eta_j \in \Gamma^\infty (\widetilde{U},\wedge^{k-1}s^* T^*M)$. 
Since the involved sheaves are fine, we need to show the claim only locally.
So let $(g,v) \in \Lambda_0(S^1\ltimes \R^2)$. Choose 
$\varrho>0$ and $\varepsilon >0$ with $\varepsilon < \pi $ such that 
$0 \notin B_\varrho(v) $ if  $v \neq 0$ 
and such that 
$e^{\sqrt{-1} t}g \neq 1 $ for all $t$ with $|t|< \varepsilon$ if $g \neq 1$. 
Let
\[
  \widetilde{U}= \left\{  (e^{\sqrt{-1} t}g, w) \in S^1\times \R^2\mid
  |t| < \varepsilon \: \& \: \| v - w \| < \varrho \right\} \ .
\]
Using the coordinate maps $\tau,x_1,x_2$ we now consider three cases.

\textit{1.~Case:}  $g=1$ and $v=0$. Then $F_{(1,w)} = T^*_w \R^2$,
  hence $\omega_{(1,w)} =0$ for all $w$ such that 
  $(1,w) \in \widetilde{U} \cap \Lambda_0$.
  Hence 
  \[
    \omega = \tau \sum_{1\leq i_1< \ldots < i_k \leq 2} 
    \omega_{i_1,\ldots , i_k} dx_1\wedge \ldots \wedge dx_{i_k} 
  \]
  with $\omega_{i_1,\ldots , i_k} \in \calC^\infty (\widetilde{U})$. 
  Now observe that $\tau x_j \in  \calJ (\widetilde{U})$ for $j=1,2$ 
  and that $d_\textup{rel} (\tau x_j) = \tau dx_j$. Therefore 
  $\omega \in d_\textup{rel} \calJ (\widetilde{U}) \wedge \Gamma^\infty (\widetilde{U}, \wedge^{k-1} s^* T^*M)$. 

\textit{2.~Case:} $g \neq 1$ and $v=0$.   
  Then $F_{(h,0)} = 0$ for all $h \in S^1$ with 
  $(h,0) \in \widetilde{U}\cap \Lambda_0$. Hence $\omega$ can be any $k$-form 
  on $\widetilde{U}$. But over $\widetilde{U}$ one has 
  $x_1, x_2\in \calJ(\widetilde{U})$ which entails that 
  \[
    \omega = \sum_{1\leq i_1< \ldots < i_k \leq 2} 
    \omega_{i_1,\ldots , i_k} dx_{i_1}\wedge \ldots \wedge dx_{i_k} \in 
     d_\textup{rel} \calJ (\widetilde{U}) \wedge
     \Gamma^\infty (\widetilde{U}, \wedge^{k-1} s^* T^*M). 
  \]

\textit{3.~Case:}  $g=1$ and $v \neq 0$. 
  Then $F_{(1,w)} = T^*\R^2$ for all $w$ such that $(1,w) \in \widetilde{U}\cap \Lambda_0$. Hence
  \[
    \omega = \tau \sum_{1\leq i_1< \ldots < i_k \leq 2} 
    \omega_{i_1,\ldots , i_k} dx_1\wedge \ldots \wedge dx_{i_k} 
  \]
  with $\omega_{i_1,\ldots , i_k} \in \calC^\infty (\widetilde{U})$. 
  Since $\tau\in \calJ(\widetilde{U})$ one obtains 
  $\omega \in  \calJ(\widetilde{U}) 
   \Gamma^\infty (\widetilde{U}, \wedge^k s^* T^*M)$. 

So in all three cases $\omega$ is in the differential graded ideal
\[ 
\calJ(\widetilde{U}) \Gamma^\infty (\widetilde{U}, \wedge^k s^* T^*M) + 
d_\textup{rel} \calJ(\widetilde{U}) \wedge \Gamma^\infty (\widetilde{U}, \wedge^{k-1} s^* T^*M) 
\]
and $[\omega]_{\Lambda_0} =0$. Hence $\Theta^k_U$ is injective.
\end{proof}

\begin{lemma}
 For every $S^1$-invariant open $V \subset \R^2$ the restriction morphism 
\[ 
  [-]_{\Lambda_0}: \Omega^\bullet_{S^1 \ltimes V \to S^1} (S^1 \ltimes V ) \to 
  \rOmega{\bullet}{\Lambda_0} \big(\Lambda_0(S^1 \ltimes V )\big)
\]
 maps the space of cycles $Z_k \big( \Omega^\bullet_{S^1 \ltimes V \to S^1} (S^1 \ltimes V ) , Y\lrcorner \big)$ 
 onto the space $\hrOmega{k}{\Lambda_0} \big(\Lambda_0(S^1 \ltimes V )\big)$
 of horizontal relative forms. 
\end{lemma}

\begin{proof}
Since the sheaf $\hrOmega{\bullet}{\Lambda_0}$ is fine it suffices to verify 
this claim for $V \subset \R^2$ of the form 
$V = B_\varrho(0) $ or $ V = B_\varrho (0) \setminus \overline{B}_\sigma (0)$ where
$0 < \sigma < \varrho$. So assume that $1\leq k\leq 2$ and
$[\omega]_{\Lambda_0} \in \hrOmega{k}{\Lambda_0} \big(\Lambda_0(S^1 \ltimes V )\big)$ for some 
relative form $\omega \in \Omega^k_{S^1 \ltimes V \to S^1} (S^1 \ltimes V )$. 
Now observe that $N_v^* = \R dr$ for all $v \in \R^2\setminus \{ 0 \}$ where  
$dr = \frac{1}{r} (dx_1 + dx_2)$. Hence, $\omega|_{\{ 1\} \times V} = 0$ if $k=2$ 
and  $\omega|_{\{ 1\} \times (V\setminus \{ 0\})} = \varphi \, dr$ with 
$\varphi \in \calC^\infty (V \setminus \{ 0 \}) $ if $k=1$. Since the claim for $k=2$ therefore has been 
proved, we assume from now on that $k=1$.
In cartesian coordinates, $\omega = \omega_1 dx_1 + \omega_2 dx_2$ 
with $\omega_j \in \calC^\infty \big(S^1 \times (V \setminus \{ 0 \})\big)$, 
$j=1,2$. Comparing with the expansion in polar coordinates gives 
the following equality over $V \setminus \{ 0 \}$
\begin{equation}
  \label{eq:comparison-polar-coordinates}
  \omega_j ( 1, - ) =  \frac{\varphi}{r} x_i \text{ for } j=1,2  \ .
\end{equation}
Note that if the origin is an element of $V$, then 
$\omega_{(1,0)} =0$, hence $(\omega_j)_{(1,0)} =0$, $j=1,2$. Choose a smooth cut-off function 
$\chi : S^1 \to [0,1]$ such that $\chi$ is identical $1$ on a neighborhood of $1$ and identical
$0$ on a neighborhood of $-1$. 
Now define the $k$-form $\widehat{\omega} \in 
\Omega^k ( S^1 \times V  ) $ by
\[
 \widehat{\omega}_{(g,x)} =  
 \begin{cases}
    \frac{\chi (g) \, \varphi (x)}{\| Z(g,x)\|}  \, \left\langle Z(g,x) ,- \right\rangle :
   \R^2 \to \R & \text{for } 
   g \in \supp \chi \text{ and } x \in V \setminus \{ 0  \}   \ ,
   \\
   0  & \text{for } g \in S^1 \setminus \supp \chi   \text{ or }  x \in V \cap \{  0 \}  \ . 
 \end{cases}
\]
where $\langle - , - \rangle $ is the euclidean inner product on $\R^2$. It needs to be verified 
that $\widehat{\omega}$ is smooth on a neighborhood of $S^1\times \{ 0 \}$ in case the origin is in $V$. 
To simplify notation we denote the composition of a function $f : V \to \R$ with the projection
$S^1 \times V \to$ again by $f$ and likewise for a function $\widetilde{f}: S^1 \to \R$. 
With this notational agreement the formula for $Z$ in \eqref{eq:expansionZ} entails
by \eqref{eq:comparison-polar-coordinates} over $(S^1 \setminus \{ -1 \}) \times (V \setminus \{ 0 \})$
\begin{equation*}
  \begin{split}
    \widehat{\omega}&|_{(S^1 \setminus \{ -1 \})  \times (V \setminus \{ 0 \})} \,  =  \\
    & = \frac{\chi \, \varphi}{r \, \sqrt{2(1+\cos \tau)}} \, 
    \Big( \left( (1+\cos \tau)x_1 + \sin \tau \, x_2 \right) dx_1
    + \left( (1+\cos \tau)x_2 - \sin \tau \, x_1 \right) dx_2 \Big) = \\
    & =  \frac{\chi}{\sqrt{2(1+\cos \tau)}} \, 
     \Big( \left( (1+\cos \tau)\omega_1 + \sin \tau \, \omega_2 \right) dx_1
    + \left( (1+\cos \tau) \omega_2 - \sin \tau \, \omega_1 \right) dx_2 \Big) \ .
  \end{split}
\end{equation*}
The right hand side can be extended by $0$ to a smooth form on $S^1 \times V$, hence
$ \widehat{\omega}$ is smooth. 
Moreover, the restriction of $ \widehat{\omega}$ to $\{ 1 \} \times V$ coincides with the 
restriction $\omega|_{\{ 1 \} \times V}$. Finally check that for $x \neq 0$ and $g\in S^1\setminus \{ -1 \}$
\[
  Y (g,x) \lrcorner \, \widehat{\omega}_{(g,x)} =  
  \frac{\chi (g) \, \varphi (x)}{\|x + gx\|}  \, \left\langle x + gx , x - gx  \right\rangle 
  = 0 \ . 
\]
Hence $\widehat{\omega} \in Z_k \big( \Omega^\bullet_{S^1 \ltimes V \to S^1} (S^1 \ltimes V ) , Y\lrcorner \big)$
and $[ \widehat{\omega}]_{\Lambda_0} = [ \omega]_{\Lambda_0}$. 
\end{proof}

\begin{proposition}
  For each  $S^1$-invariant open $V \subset \R^2$ the chain map 
  \[
    \big( \Omega^\bullet_{S^1 \ltimes V \to S^1} (S^1 \ltimes V ), Y \lrcorner \big) 
    \to
    \big( \hrOmega{\bullet}{\Lambda_0} (\Lambda_0(S^1 \ltimes V )), 0 \big)
  \]
  is a quasi-isomorphism. 
\end{proposition}
\begin{proof}
It remains to prove that every 
$\omega \in Z_k \big( \Omega^\bullet_{S^1 \ltimes V \to S^1} (S^1 \ltimes V ), Y \lrcorner \big)$ 
which satisfies the condition $ [ \omega]_{\Lambda_0} = 0$ is of the form 
$\omega =  Y \lrcorner \, \eta $ for some 
$\eta \in  \Omega^{k+1}_{S^1 \ltimes V \to S^1} (S^1 \ltimes V )$. 
Let us show this. We consider the three non-trivial cases $k=0,2,1$ separately.
  
\textit{1.~Case:} $k=0$. Then $\omega$ is a smooth function on $S^1 \ltimes V$ vanishing
on $\Lambda_0$.  By Prop.~\ref{prop:vanishingideal}, the function $\omega$ can  be expanded over $S^1\setminus \{ -1\} \times V$ 
in the form
\[
 \omega|_{S^1\setminus \{ -1\} \times V} = \omega_1\tau x_1 + \omega_2\tau x_2 \ ,\quad\text{where } \omega_1,\omega_2 
  \in \calC^\infty (S^1\setminus \{ -1\} \times V) \ .
\]
Moreover, the interior product of a form $\eta = \eta_1 dx_1 + \eta_2 dx_2 \in\Omega^1_{S^1 \ltimes V \to S^1} (S^1 \ltimes V )$ 
with the vector field $Y$ has the form
\[
  Y \lrcorner\, \eta = Y_1 \eta_1 + Y_2\eta_2 = (x_1(1- \cos\tau)-x_2\sin \tau)\eta_1 + (x_2(1-\cos\tau)-x_1\sin \tau)\eta_2 \ .
\]
This means that it suffices to find $\eta_1,\eta_2 \in \calC^\infty (S^1 \ltimes V) $  which solve the system of equations  
\begin{equation}
\label{eq:system-equation-coefficient-fcts}
\begin{split}
   \omega_1\tau& =  (1- \cos\tau) \eta_1 + (\sin \tau) \eta_2 \ , \\ 
   \omega_2\tau & =  - (\sin \tau) \eta_1 + (1 -\cos \tau) \eta_2 \ . 
\end{split}
\end{equation}
The $1$-form $\eta = \eta_1 dx_1 + \eta_2 dx_2$ will then satisfy  $Y \lrcorner\, \eta = \omega$ which will prove the first case. 
The functions
\begin{equation*}
\label{eq:linear-equation-coefficient-fcts}
\begin{split}
   \eta_1& = \frac{\tau(1-\cos \tau)}{(1-\cos \tau)^2+\sin^2\tau} \omega_1 
             - \frac{\tau\sin \tau}{(1-\cos \tau)^2+\sin^2\tau} \omega_2 =
               \frac{\tau}{2} \omega_1 
             -   \frac{\tau\sin \tau}{2 (1-\cos \tau)} \omega_2 \  \\
   \eta_2 & = \frac{\tau \sin \tau}{(1-\cos \tau)^2+\sin^2\tau} \omega_1 
              +  \frac{\tau (1-\cos \tau)}{(1-\cos \tau)^2+\sin^2\tau} \omega_2 =
               \frac{\tau\sin \tau}{2 (1-\cos \tau)} \omega_1 
             +   \frac{\tau}{2} \omega_2 
\end{split}
\end{equation*}
now are well-defined and smooth over $(S^1 \times V )\setminus (\{ 1 \} \times \R^2)$. 
They also solve \eqref{eq:system-equation-coefficient-fcts}. We are done when we can show that they can be extended smoothly to 
the whole domain $S^1 \times V$.  But this is clear since the function 
$(-\pi, \pi)\setminus \{ 0 \} \to \R$, $t\mapsto  \frac{t \sin t}{2 (1-\cos t)}$ has a holomorphic
extension near the origin as one verifies by power series expansion. 

\textit{2.~Case:} $k=2$.
Let $\omega \in \Omega^2_{S^1 \ltimes V \to S^1} (S^1 \ltimes V )$ and $Y \lrcorner \, \omega =0$. 
Then $\omega = \varphi dx_1 \wedge dx_2$ for some smooth function $\varphi \in S^1 \ltimes V \to S^1$. 
Now compute using \eqref{eq:expansionY}
\begin{equation*}
  \begin{split}
    0&=Y\lrcorner\,\omega=\varphi\cdot (Y_1-Y_2)=
    \varphi\cdot\big( x_1(1-\cos\tau)-x_2\sin\tau -x_2(1-\cos\tau)  - x_1\sin\tau \big)=\\
    &=\varphi\cdot (x_1-x_2) \cdot (1-\cos\tau-\sin\tau)\ .
  \end{split}
\end{equation*}
Hence $\varphi=0$ and $\omega=0$.

\textit{3.~Case:} $k=1$.
 Observe that in this case $\omega$ can be written in the form $\omega= \omega_1 dx_1 + \omega_2 dx_2$ with
 $\omega_1,\omega_2 \in \mathcal{J} ( S^1 \times V ) \subset \calC^\infty( S^1 \times V )$. 
 By Lemma \ref{eq:expansion-functions-vanishing-ideal},
 $\omega_j|_{(S^1 \setminus \{ -1\}) \times V} = \tau \Omega_j$ for $j=1,2$ and functions 
 $\Omega_j \in \calC^\infty((S^1 \setminus \{ -1\}) \times V) $. 
The condition $Y \lrcorner \, \omega = 0 $ implies 
 \begin{equation}
   \label{eq:cycle-condition-one-form}
   Y_1 \Omega_1 + Y_2 \Omega_2 = Y_1 \omega_1 + Y_2 \omega_2 = 0 \ . 
 \end{equation}
Now define the function $\varphi : (S^1 \times V )\setminus \Lambda_0 \to \R$ by
$ \varphi  = \left. \frac{1}{\|Y\|^2} (-Y_2 \omega_1 + Y_1 \omega_2 )\right|_{(S^1\times V) \setminus \Lambda_0} $. 
Since $\|Y\|^2 = 2 r^2 (1-\cos \tau)$, the vector field $Y$ vanishes nowhere on $ (S^1 \times V )\setminus \Lambda_0$,
so $\varphi$ is well-defined and smooth.
By \eqref{eq:cycle-condition-one-form} one computes 
\begin{equation*}
  \varphi (g,x) = 
  \begin{cases}
    \frac{ \omega_2}{Y_1}  (g,x)& \text{ if } g \neq 1, x \neq 0 \text{ and } Y_1(g,x) \neq 0 \ , \\
    \frac{- \omega_1}{Y_2}  (g,x) & \text{ if } g \neq 1, x \neq 0 \text{ and } Y_2(g,x) \neq 0 \ .
  \end{cases}
\end{equation*} 
Assume that $\varphi$ can be extended smoothly to $S^1 \times V $. 
Then $\eta = \varphi dx_1 \wedge dx_2 $ is a smooth form on $S^1 \times V$ which satisfies
\[
   Y \lrcorner \eta = \varphi ( Y_1 dx_2 - Y_2 dx_1 ) =  \omega \ . 
\]
So it remains to verify that $\varphi$ can be smoothly extended to $S^1 \times V$. 
To this end we use the complex coordinate $z= x_1 + \sqrt{-1} x_2$ of $V$ and introduce the 
complex valued function $\Omega = \Omega_1 + \sqrt{-1} \Omega_2$. Moreover, we define 
$y:  S^1 \times V \to \C $, $(g,z) \mapsto z - gz $. 
Then 
\begin{equation}
  \label{eq:representation-vector-field}
  y = (1 - e^{\sqrt{-1}\tau})z = Y_1 + \sqrt{-1}Y_2
\end{equation}
and, by Eq.~\ref{eq:cycle-condition-one-form}, 
\begin{equation}
  \label{eq:complex-cycle-condition}
   \frac{1}{2}( y\overline{\Omega}+\overline{y}\Omega) =  Y_1\Omega_1+Y_2\Omega_2 = 0 \ .
\end{equation}
 Next observe that $1-e^{\sqrt{-1}\tau} = -\sqrt{-1}\tau \big(1-\sqrt{-1}\tau (\zeta \circ \tau)\big)$ for some holomorphic $\zeta :\C \to \C$ 
 which fulfills $\zeta (0)=\frac 12$.  Then Eq.~\ref{eq:complex-cycle-condition} entails
 \[
   \big(1-\sqrt{-1}\tau (\zeta \circ \tau)\big) z \overline{\Omega} = \big(1 + \sqrt{-1}\tau (\overline{\zeta} \circ \tau)\big)  \overline{z} \Omega \ .
 \]
 By power series expansion it follows that 
 $\left. \frac{\partial\Omega}{\partial \overline{z}}  \right|_{z=0} = 0$ for all $k \in \N$. Hence, by Taylor's Theorem 
 $\Omega = z \Phi $ for some smooth $\Phi : S^1 \times V \to \C$. 
 Since by Lemma \ref{lem:coordinate-representations-vector-fields} $\|Y\|^2 = r^2 \tau^2 (\xi \circ \tau)$ for some 
 holomorphic  function $\xi$ not vanishing on $(-\pi,\pi)$ the following equality holds over 
 $(S^1 \setminus \{ \pm 1\}) \times (V \setminus \{ 0\})$
 \begin{equation*}
   \begin{split}
     \varphi &  =  
     \frac{1}{\tau r^2(\xi \circ \tau)} (- Y_2 \Omega_1 + Y_1 \Omega_2)
     = \frac{\sqrt{-1}}{2\tau r^2(\xi \circ \tau)}\big(y\overline{\Omega}-\overline{y}\Omega \big)  = \\
     & =  \frac{1}{2r^2(\xi \circ \tau)} 
     \Big( \big(1-\sqrt{-1}\tau (\zeta \circ \tau)\big) z \overline{z} \overline{\Phi} + \big(1+\sqrt{-1}\tau (\overline{\zeta} \circ \tau)\big) z \overline{z} \Phi
     \Big) = \\ & =  \left.\frac{1}{(\xi \circ \tau)}  
     \big(1-\sqrt{-1}\tau (\zeta \circ \tau)\big) \overline{\Phi} \right|_{(S^1 \setminus \{ \pm 1\}) \times (V \setminus \{ 0\})}  \ .
   \end{split}
 \end{equation*}
 Since the right hand side has a smooth extension to $S^1\setminus \{ -1 \} \times V$,
 the function  $\varphi$ can be smoothly extended to $S^1 \times V$ and the claim is proved. 
\end{proof}

\subsection{$S^1$ rotation in $\mathbb{R}^{2m}$}  
In this subsection, we work with complex-valued functions, and differential
forms over complex numbers. Since tensoring an $\mathbb{R}$-vector
space  with $\mathbb{C}$ is a faithfully flat functor, our results  in this
section still hold true for the algebra of real-valued functions. 

We consider a linear representation of $S^1$ on $\mathbb{R}^{2m}$. We identify $\mathbb{R}^{2m}$ with $\mathbb{C}^m$, and decompose $\mathbb{C}^m$ into the following two subspaces, i.e.
\begin{equation}\label{eq:action}
\mathbb{C}^m=V_0 \oplus V_1, 
\end{equation}
where $V_0$ is the subspace of $\mathbb{C}^m$ on which  $S^1$ acts trivially, and $V_1$ is the $S^1$-invariant subspace of $\mathbb{C}^m$ orthogonal to $V_0$ with respect to an $S^1$-invariant hermitian metric on $\mathbb{C}^m$. Furthermore, $V_1$ is decomposed into irreducible unitary representations of $S^1$, i.e. 
\[
V_1=\bigoplus_{j=1}^{t} \mathbb{C}_{w_j},
\]
where  $\mathbb{C}_{w_j}$ is an irreducible representation $\rho_{w_j}$ of $S^1$ with the weight $0\neq w_j\in \mathbb{Z}$, i.e. 
\[
\rho_{w_j}(\exp(2\pi \sqrt{-1})t)\big(z\big):=\exp(2w_j \pi \sqrt{-1}t)z.
\]
We observe that $\calC^\infty(\mathbb{C}^m)\rtimes S^1$ is isomorphic to $\big(\calC^\infty(V_0)\otimes \calC^\infty(V_1)\big)\rtimes S^1$. As $S^1$ acts on $V_0$ trivially, we have
\[
\calC^\infty(\mathbb{C}^m)\rtimes S^1\cong \calC^\infty(V_0)\otimes \big( \calC^\infty(V_1)\rtimes S^1 \big).
\]
The K\"unneth formula for Hochschild homology \cite[Theorem 4.2.5]{LodCH} gives 
\[
HH_{\bullet}\Big( \calC^\infty(\mathbb{C}^m)\rtimes S^1\Big)=HH_\bullet\big(  \calC^\infty(V_0)\big)\otimes HH_\bullet \big( \calC^\infty(V_1)\rtimes S^1 \big). 
\]
The Hochschild-Kostant-Rosenberg theorem shows $HH_\bullet\big(  \calC^\infty(V_0)\big)=\Omega^\bullet (V_0)$. Hence, we have reduced  the computation of $HH_\bullet \big( \calC^\infty(\mathbb{C}^m)\rtimes S^1 \big)$ to $HH_\bullet \big( \calC^\infty(V_1)\rtimes S^1 \big)$. Without loss of generality, we assume in the left of this subsection that $\mathbb{C}^m=V_1$, i.e. 
\[
\mathbb{C}^m=\bigoplus_{j=1}^{m} \mathbb{C}_{w_j},\ 0\neq w_j\in \mathbb{Z}. 
\]

Let $w$ be the lowest common multiplier of $w_1, ..., w_m$. We observe that for $t\in [0, 1)$, if $t \neq \frac{j}{w}, j=0, ..., w-1$, the fixed point subspace of $t$ is $\{0\}$; if $t=\frac{j}{w}$, the fixed point subspace of $t$ is 
\[
\mathbb{C}_{w_{k_1}}\oplus \cdots \oplus \mathbb{C}_{w_{k_l}}, 
\]
for $w_{k_1}, ..., w_{k_l}$ that $w$ divides $jw_{k_1},\cdots, jw_{k_l}$. Hence the loop space $\Lambda_0(S^1\ltimes \mathbb{C}^m)$ has the following form, 
\[
\begin{split}
\Lambda_0(S^1\ltimes \mathbb{C}^m)=\Big\{ \big(\exp(2\pi \sqrt{-1}t), & (0,\cdots, z_{w_{k_1}}, \cdots, z_{w_{k_l}},0,\cdots) \big)\Big{|} \\
&(0,\cdots, z_{w_{k_1}}, \cdots, z_{w_{k_l}},0,\cdots)\in \mathbb{C}^m, t w_{k_1}, \cdots, tw_{k_l}\in \mathbb{Z} w\Big\}. 
\end{split}
\]

Let $\sigma: \Lambda_0(S^1\ltimes \mathbb{C}^m)\to S^1$ be the forgetful map mapping $(\exp(2\pi \sqrt{-1}t), z)\in \Lambda_0(S^1\ltimes \mathbb{C}^m)$ to $\exp(2\pi \sqrt{-1}t)$. 

Following Proposition \ref{prop:chain-map-horizontal-relative-forms} and Eq. (\ref{eq:ConvolutionConnesKoszulChainCpl}), we can compute the Hochschild homology of $\calC^\infty(\mathbb{C}^m)\rtimes S^1$ is computed by  the $S^1$-invariant part of the cohomology of the following Koszul type complex, 
\begin{equation}\label{eq:S1koszul}
       \Omega^{2m}_{S^1\ltimes \mathbb{C}^m \to S^1} (S^1\ltimes \mathbb{C}^m)  \overset{i_{Y_{S^1\ltimes \mathbb{C}^m}}}{\longrightarrow} \ldots 
       \overset{i_{Y_{S^1 \ltimes \mathbb{C}^m }}}{\longrightarrow}   \Omega^1_{S^1 \ltimes \mathbb{C}^m \to S^1} (S^1\ltimes \mathbb{C}^m)
       \overset{i_{Y_{S^1 \ltimes \mathbb{C}^m}}}{\longrightarrow}  \calC^\infty (S^1 \ltimes \mathbb{C}^m) \longrightarrow 0 \ ,
\end{equation}
where $Y_{S^1 \ltimes \mathbb{C}^m } : S^1\ltimes \mathbb{C}^m \to  s^*T\mathbb{C}^m$ is defined by $Y_{S^1\ltimes \mathbb{C}^m}(g,v) =v -gv$. Below, we sometimes abbreviate the symbol $Y_{S^1 \ltimes \mathbb{C}^m }$ by $Y$ by abusing the notations. Fix a choice of coordinates $(z_1, \cdots, z_m)$ for $z_j\in \mathbb{C}_{w_j}$. The vector field $Y:=Y_{S^1\ltimes \mathbb{C}^m}(\exp(2\pi \sqrt{-1}t),z)$ is written as
\[
Y:=Y_{S^1\ltimes \mathbb{C}^m}(\exp(2\pi \sqrt{-1}t),z)=\sum_{k=1}^m \big(\exp(2\pi \sqrt{-1} w_kt)-1\big)z_k\frac{\partial}{\partial z_k}+ \big(\exp(-2\pi \sqrt{-1} w_kt)-1\big)\bar{z}_k\frac{\partial}{\partial \bar{z}_k} \ .
\]
Define an analytic function $a(z)$ on $\mathbb{C}$ by  
\[
a(z):=\frac{\exp(2\pi \sqrt{-1}z)-1}{z} \ . 
\]
Then  we have 
\[
\begin{split}
\exp(2\pi \sqrt{-1}w_k t)-1&= w_k t a(w_k t),\\
\exp(-2\pi \sqrt{-1}w_k t)-1&=w_k t \bar{a}(w_k t).
\end{split}
\]
Observe that for $t\in \mathbb{R}$, $a(t)=\bar{a}(t)$, and $a(t)\ne 0$ for all $t$ sufficiently close to $0$. For a sufficiently small $\epsilon$, the vector field $Y$ on $(-\epsilon, \epsilon)\times \mathbb{C}^m$ is of the following form
\[
  Y= t \sum_{k=1}^m  w_k \left( a(w_k t) z_k \frac{\partial}{\partial z_k} +
    \overline{a(w_k t)}\bar{z}_k \frac{\partial }{\partial \bar{z}_k}\right).
\]
This leads to the following property of the vector field $Y$. 

\begin{lemma}\label{lem:vectorfieldY} The vector field $Y: S^1\times \mathbb{C}^m\to \mathbb{C}^m,\ (g, z)\mapsto z-gz$ has a coordinate representation $Y=\sum_{k=1}^m Y^k z_k \frac{\partial }{\partial z_k}+\overline{Y}^k \bar{z}_k\frac{\partial }{\partial \bar{z}_k}$ with coefficients given by
\[
Y^k\big(\exp(2\pi \sqrt{-1}t)\big)=\exp (2\pi \sqrt{-1}w_k t)-1. 
\]
Denote $w=l.c.m.(w_1, \cdots, w_m)$. When $t_0=\frac{j}{w}$, for $0\leq j<w$, there is a sufficiently small $\epsilon>0$ such that on $ (\frac{j}{w}-\epsilon, \frac{j}{w}+\epsilon)$, $Y^k$ is of the following form,
\[
Y^k\big( \exp(2\pi \sqrt{-1}t)\big)= w_k (t-\frac{j}{w}) a\big(w_k( t-\frac{j}{w})\big),\ \text{for}\ w_k j \in \mathbb{Z}w,
\]  
where $a\big(w_k( t-\frac{j}{w})\big)\ne 0$ for all $t\in (\frac{j}{w}-\epsilon, \frac{j}{w}+\epsilon)$. And for $k$ with $w_k j \notin \mathbb{Z}w$, $Y^k\big(\exp(2\pi \sqrt{-1}t)\big)\ne 0$ for all $t\in (\frac{j}{w}-\epsilon, \frac{j}{w}+\epsilon)$. 

When $t_0\neq\frac{j}{w}$, there is a sufficiently small $\epsilon>0$ such that on $(t_0-\epsilon, t_0+\epsilon)$, $Y^k\big(\exp(2\pi \sqrt{-1}t)\big)\ne 0$ for all $t\in(t_0-\epsilon, t_0+\epsilon) $. 
\end{lemma}

Analogous to the expression of the vector field $Y$, we study in the following lemma the local expression of the vanishing ideal $\calJ$ of the loop space $\Lambda_0(S^1\ltimes \mathbb{C}^m)$ for the $S^1$ action on $\mathbb{C}^m$ defined by Equation (\ref{eq:action}). 

\begin{lemma}
\label{lem:vanishingideal}
The vanishing ideal $\calJ$ of $\Lambda_0(S^1\ltimes \mathbb{C}^m)$ has the following local form.
\begin{itemize}
\item Near $\big(\exp(2\pi \sqrt{-1}\frac{j}{w}), 0\big)\in S^1\times \mathbb{C}^m$, the vanishing ideal $\calJ\big( (\frac{j}{w}-\epsilon, \frac{j}{w}+\epsilon)\times B_{\varrho}(0)\big)$ for a sufficiently small $\epsilon>0$ and a ball $B_{\varrho}(0)\subset \mathbb{C}^m$ centered at $0$ with a sufficiently small radius $\varrho>0$ consists of all smooth functions $f\in \calC^\infty\big( (\frac{j}{w}-\epsilon, \frac{j}{w}+\epsilon)\times B_{\varrho}(0)\big)$ which can be written in the form 
\[
f=(t-\frac{j}{w})\sum _{k, w_k j\in w\mathbb{Z} }  (z_k f_k+\bar{z}_k g_k)+\sum_{k, w_kj\notin w\mathbb{Z}} (z_k f_k+\bar{z}_k g_k),
\]
for $f_k, g_k\in  \calC^\infty\big( (\frac{j}{w}-\epsilon, \frac{j}{w}+\epsilon)\times B_{\varrho}(0)\big)$.
\item Near $\big(\exp(2\pi \sqrt{-1}\frac{j}{w}), Z\big)\in S^1\times \mathbb{C}^m$ with $Z\ne 0$ and $\exp(2\pi \sqrt{-1}\frac{j}{w}) Z=Z$, the vanishing ideal $\calJ\big( (\frac{j}{w}-\epsilon, \frac{j}{w}+\epsilon)\times B_{\varrho}(Z)\big)$ for a sufficiently small $\epsilon>0$ and a ball $B_{\varrho}(Z)\subset \mathbb{C}^m$ centered at $Z$ with a sufficiently small radius $\varrho>0$ consists of all smooth functions $f\in \calC^\infty\big( (\frac{j}{w}-\epsilon, \frac{j}{w}+\epsilon)\times B_{\varrho}(Z)\big)$ which can be written in the form 
\[
f=(t-\frac{j}{w})f+\sum_{k, w_kj\notin w\mathbb{Z}} (z_k f_k+\bar{z}_k g_k),
\]
for $f, f_k, g_k\in \calC^\infty\big( (\frac{j}{w}-\epsilon, \frac{j}{w}+\epsilon)\times B_{\varrho}(Z)\big)$. 
\item Near $\big(\exp(2\pi \sqrt{-1}t_0), 0\big)\in S^1\times \mathbb{C}^m$  such that $t_0\ne \frac{j}{w}$ for all $j$ and $0\in \mathbb{C}^m$,  the vanishing ideal $\calJ\big( (t_0-\epsilon, t_0+\epsilon)\times B_{\varrho}(0)\big)$ for a sufficiently small $\epsilon>0$ and a ball $B_{\varrho}(0)\subset \mathbb{C}^m$ centered at $0$ with a sufficiently small radius $\varrho>0$ consists of all smooth functions $f\in \calC^\infty\big( (t_0-\epsilon, t_0+\epsilon)\times B_{\varrho}(0)\big)$ which can be written in the form 
\[
f=\sum_{k=1}^m (z_k f_k+\bar{z}_k g_k), 
\]
for $f_k, g_k\in \calC^\infty\big( (t_0-\epsilon, t_0+\epsilon)\times B_{\varrho}(0)\big)$. 
\end{itemize}
\end{lemma}

\begin{proof} We will prove the case around the most singular point $(1,0)\in S^1\times \mathbb{C}^m$. A similar proof works for the other points. We leave the details to the reader. 

For $(1, 0)\in S^1\times \mathbb{C}^m$, choose a sufficiently small $\epsilon>0$ such that there is no other point in the interval $(-\epsilon, \epsilon)$ of the form $\frac{j}{w}$ for an integer $0<j<w$.  We identify $(-\epsilon, \epsilon)$ with a neighborhood of $1$ in $S^1$ via the exponential map. For a positive $\varrho$, the loop space $\Lambda_0(S^1\ltimes \mathbb{C}^m)$ in $ (-\epsilon, \epsilon)\times B_{\varrho}(0)$ is of the form 
\[
\Lambda_0(S^1\times \mathbb{C}^m)_{(0, 0)}=\{ (0, z)|z\in B_{\varrho}(0)\}\cup \{(t, 0)\}.
\]
A smooth function $f$ on $(-\epsilon, \epsilon)\times B_{\varrho}(0)$ belongs to $\calJ\big((-\epsilon, \epsilon)\times B_{\varrho}(0)\big)$ if and only if
\[
f(0,z)=f(t,0)=0. 
\]
We consider $f$ as a function of $t\in (-\epsilon, \epsilon)$. By the Malgrange preparation theorem, we have the expansion
\[
f(t,z)+t=c(t, z)(t+a_0(z)),
\]
where $c(t, z)$ and $a_0(z)$ are smooth and $a_0(0)=0$.  Since $t=c(t,0)t$ for all $t\in (-\epsilon, \epsilon)$, $c(t,0)=1$. Putting $t=0$ gives $0=c(0,z)a_0(z)$ for all $z\in B_{\varrho}(0)$. Recall that $c(0,0)=1$. Therefore, $a_0(z)=0$ for all $z$ in a neighborhood of $0$. After possibly shrinking $\varrho$, we can assume that $a_0(z)=0$ on $B_{\rho}(0)$. Hence, we conclude that 
\[
f(t,z)=t(c(t,z)-1).
\]
Taking the parametric Taylor expansion of $c(t,z)-1$ gives
\[
c(t,z)-1=\sum_{j=1}^m z_jf_j(t,z)+\bar{z}_j g_j(t,z),
\]
where $f_j$ and $g_j$ are smooth functions on $(-\epsilon, \epsilon)\times B_{\varrho}(0)$. 
\end{proof}

In the following, we compute the cohomology of the complex (\ref{eq:S1koszul}). We observe that the complex $(\Omega^\bullet_{S^1\ltimes \mathbb{C}^m\to S^1}(S^1\ltimes \mathbb{C}^m), i_{Y})$ for $Y:=Y_{S^1 \ltimes \mathbb{C}^m } $ forms a sheaf of complexes over $S^1$ via the map $\sigma:  \Lambda_0(S^1\ltimes \mathbb{C}^m)\to S^1$.  Accordingly, we compute the cohomology $\big( \Omega^{\bullet}_{S^1\ltimes \mathbb{C}^m \to S^1}(S^1 \ltimes \mathbb{C}^m ), i_{Y}\big)$ as a sheaf over $S^1$. 

\begin{proposition}
\label{prop:rel-F} For all open subsets $U$ of the loop space $\Lambda_0 = \Lambda_0 (S^1 \ltimes \mathbb{C}^m)$ and all $k\in \N$ the map 
 \[
   \Theta^k_U : \rOmega{k}{\Lambda_0}(U) \to \Gamma^\infty (U,\bigwedge^k F) 
 \]
 from Prop.~\ref{prop:factorization-relative-forms} is injective. 
\end{proposition}

\begin{proof}
We will prove the case around the most singular point $(1,0)\in S^1\times \mathbb{C}^m$. A similar proof works for the other points. We leave the detail to the reader. 

Recall that we show in Lemma \ref{lem:vanishingideal} that near $(1,0)$, the vanishing ideal $\calJ\big( (-\epsilon, \epsilon)\times B_{\varrho}(0)\big)$ for a sufficiently small $\epsilon>0$ and a ball $B_{\varrho}(0)\subset \mathbb{C}^m$ centered at $0$ with  a sufficiently small radius $\varrho>0$ consists of all smooth functions $f\in \calC^\infty\big( (-\epsilon, \epsilon)\times B_{\varrho}(0)\big)$ which can be written in the form 
\[
f=t\sum _{k=1}^m  (z_k f_k+\bar{z}_k g_k),
\]
for $f_k, g_k\in  \calC^\infty\big( (-\epsilon, \epsilon)\times B_{\varrho}(0)\big)$.
Recall that by definition, $\rOmega{p}{\Lambda_0} \big( (-\epsilon, \epsilon)\times B_{\varrho}(0) \big)$
is the quotient 
\[
  \Omega^p_{S^1\ltimes \mathbb{C}^m\to S^1} \big( (-\epsilon, \epsilon)\times B_{\varrho} (0) \big) /
  \calJ \Omega^p_{S^1\ltimes \mathbb{C}^m\to S^1} + d\calJ \wedge \Omega^p_{S^1\ltimes \mathbb{C}^m\to S^1}
  \big( (-\epsilon, \epsilon)\times B_{\varrho}(0) \big) \ .
\]
In the following, we will discribe
$\rOmega{p}{\Lambda_0} \big( (-\epsilon, \epsilon)\times B_{\varrho}(0) \big)$
in more detail and, for ease of notation, will use the symbols
$\Omega^p_{S^1\ltimes \mathbb{C}^m\to S^1}$ and $\rOmega{p}{\Lambda_0}$ to stand for
$\Omega^p_{S^1\ltimes \mathbb{C}^m\to S^1}\big((-\epsilon, \epsilon)\times B_{\varrho}(0) \big)$ and
$\rOmega{p}{\Lambda_0}\big((-\epsilon, \epsilon)\times B_{\varrho}(0) \big)$, respectively, and
$\calJ$ for the vanishing ideal $\calJ\big( (-\epsilon, \epsilon)\times B_{\varrho}(0)\big)$.

In degree $p=0$, $\rOmega{0}{\Lambda_0}$ coincides with the quotient of
$\calC^\infty\big((-\epsilon, \epsilon)\times B_{\varrho}(0)\big)$ by
$\calJ\big( (-\epsilon, \epsilon)\times B_{\varrho}(0) \big)$.

In degree $p=1$, we know by Lemma \ref{lem:vanishingideal} that $d\calJ$ consists of $1$-forms
which can be expressed as follows:
\[
t\sum_{k=1}^m (f_k dz_k+g_k d\bar{z}_k),\ f_k, g_k\in \calC^\infty\big( (-\epsilon, \epsilon)\times B_{\varrho}(0)\big). 
\]
Hence, $d\calJ$ is of the form $t \Omega^1_{S^1\ltimes \mathbb{C}^m\to S^1}$,
which contains $\calJ \Omega^1_{S^1\ltimes \mathbb{C}^m\to S^1}$. 
Notice that for $(0, z)\in S^1\times \mathbb{C}^m$,  $F_{(0, z)}$ coincides with $T^*_z\mathbb{C}^m$.
For $\omega=\sum_{k=1}^m f_k dz_k+g_k d\bar{z}_k \in \rOmega{1}{\Lambda_0}$, if $\Theta(\omega)=0$, then $f_k(0,z)=g_k(0,z)=0$ for $1\leq k\leq m$. Therefore, taking the parametric Talyor expansion of  $f_k, g_k$ at $(0,z)$, we have that there are $\tilde{f}_k$ and $\tilde{g}_k$ in $\calC^\infty\big((-\epsilon, \epsilon)\times B_{\varrho}(0) \big)$ such that $f_k=t\tilde{f}_k$ and $g_k=t\tilde{g}_k$. Hence, $\omega=t\sum_{k=1}^m \tilde{f}_k dz_k +\tilde{g}_k d\bar{z}_k\in d\calJ$ and $[\omega]=0$ in $\rOmega{1}{\Lambda_0}$. 

In degree $p>1$, the above description of $\rOmega{1}{\Lambda_0}$ generalizes with the
above expression for $d\calJ$. As $\Omega^k_{S^1\ltimes \mathbb{C}^m\to S^1}$ is of the form 
\[
\sum_{j}dz_j\wedge \Omega^{k-1}_{S^1\ltimes \mathbb{C}^m\to S^1}+d\bar{z}_j \wedge \Omega^{k-1}_{S^1\ltimes \mathbb{C}^m\to S^1}, 
\]
we conclude that $d\calJ \wedge\Omega^{k-1}_{S^1\ltimes \mathbb{C}^m\to S^1}$ can be identified as $t\Omega^{k}_{S^1\ltimes \mathbb{C}^m\to S^1}$,
which contains $\calJ \Omega^{k-1}_{S^1\ltimes \mathbb{C}^m\to S^1}$ as a subspace. 

We notice that at $(0,z)\in S^1\times \mathbb{C}^m$, $\bigwedge^k F_{(0,z)}$ is $\bigwedge^k T^*_{(0,z)}\mathbb{C}^m$. For $\omega=\sum_{I,J}f_{I,J}dz_{I_1}\wedge\cdots \wedge dz_{I_s}\wedge d\bar{z}_{J_{s+1}}\wedge \cdots \wedge d\bar{z}_{J_{k}}$, with $1\leq I_1<\cdots <I_s\leq m$ and $1\leq J_{s+1}<\cdots<J_{k}\leq m$, if $\Theta(\omega)=0$, we then get $f_{I,J}(0,z)=0$ for all $I,J$. And we can conclude from the Taylor expansion that there exists $\tilde{f}_{I,J}$ such that $f_{I,J}=t\tilde{f}_{I,J}$, and $\omega=t\sum_{I, J} \tilde{f}_{I,J}dz_{I_1}\wedge\cdots \wedge dz_{I_s}\wedge d\bar{z}_{J_{s+1}}\wedge \cdots \wedge d\bar{z}_{J_{k}}$ which is an element in $d\calJ \wedge \Omega^{k-1}_{S^1\ltimes \mathbb{C}^m\to S^1}$. Therefore, $[\omega]=0$ in $\rOmega{k}{\Lambda_0}$. 
\end{proof}

\begin{proposition}\label{prop:equivariant-koszul}
  For each  $S^1$-invariant open $V \subset \mathbb{C}^m$ the chain map 
  \[
    \mathfrak{R}: \big( \Omega^\bullet_{S^1 \ltimes V \to S^1} (S^1 \ltimes V ), Y \lrcorner \big) 
    \to
    \big( \hrOmega{\bullet}{\Lambda_0} (\Lambda_0(S^1 \ltimes V )), 0 \big)
  \]
  is a quasi-isomorphism. 
\end{proposition}

\begin{proof} We consider both sides as sheaves over $S^1$, and prove that $\mathfrak{R}$
  is a quasi-isomorphism of sheaves over $S^1$. 
  It is sufficient to prove the quasi-isomorphism $\mathfrak{R}$ at each stalk.
  We split our proof into two parts according to the point $t_0$ in $S^1$,
\begin{enumerate} 
\item at $\exp(2\pi \sqrt{-1}t_0)$ with $t_0\ne \frac{j}{w}$ for $0\leq j<w$ and $t\in [0,1)$, 
\item at $\exp(2\pi \sqrt{-1}\frac{j}{w})$ for $0\leq j<w$. 
\end{enumerate}

\noindent{\textit{Case (1)}}.
We prove that
\[
  \mathfrak{R}_{\exp(2\pi \sqrt{-1}t_0)}: \big( \Omega^\bullet_{S^1 \ltimes V \to S^1, \exp(2\pi \sqrt{-1}t_0)} (S^1 \ltimes V ), Y \lrcorner \big)\to  {\hrOmega{\bullet}{\Lambda_0 ,\exp(2\pi \sqrt{-1}t_0)}} (\Lambda_0(S^1 \ltimes V ))
\]
is a quasi-isomorphism for $t_0\ne \frac{j}{w}$ for $0\leq j<w$ and $t_0\in [0,1)$. It is crucial to observe that for a sufficiently small $\epsilon>0$, on $(t_0-\epsilon, t_0+\epsilon)\times \mathbb{C}^m$, the vector field $Y$ is of the  form
\[
Y=\sum_{j=1}^m \big(\exp(2\pi \sqrt{-1}w_j t)-1\big)z_j\frac{\partial}{\partial z_j}+ \big(\exp(-2\pi \sqrt{-1}w_j t)-1\big)z_j\frac{\partial}{\partial \bar{z}_j}.
\]
Observe that the vector field $Y$ vanishes exactly at $(t, 0)$.
Morover,
\[ \big( \Omega^\bullet_{S^1 \ltimes V \to S^1, \exp(2\pi \sqrt{-1}t_0)} \big((t_0-\epsilon, t_0+\epsilon)\times \mathbb{C}^m \big), Y \lrcorner \big)
\]
is a smooth family of generalized Koszul complexes over $t\in (t_0-\epsilon, t_0+\epsilon)$. Its cohomology can be computed using Proposition \ref{prop:parametrizedkoszul} as
\[
H^\bullet \big(  \Omega^\bullet_{S^1 \ltimes V \to S^1, \exp(2\pi \sqrt{-1}t_0)} \big((t_0-\epsilon, t_0+\epsilon)\times \mathbb{C}^m \big), Y \lrcorner \big)=\left\{
\begin{array}{ll}
\calC^\infty\big(t_0-\epsilon, t_0+\epsilon\big),& \bullet=0,\\
0,&\text{otherwise}.
\end{array}
\right.
\]
At the same time, for every $t$ in $(t_0-\epsilon, t_0+\epsilon)$, the fixed point of $\exp(2\pi \sqrt{-1}t)$ is $0$ in $\mathbb{C}^m$. And therefore, the complex $ \hrOmega{\bullet}{\Lambda_0}\big( (t_0-\epsilon, t_0+\epsilon)\times \mathbb{C}^m\big)$ identified with $\Gamma^\infty\big((t_0-\epsilon, t_0+\epsilon)\times \{0\}, \bigwedge^\bullet F\big)$ is computed as follows,
\[
\Gamma^\infty\big((t_0-\epsilon, t_0+\epsilon)\times \{0\}, {\bigwedge}^\bullet F\big)=\left\{
\begin{array}{ll}
\calC^\infty\big(t_0-\epsilon, t_0+\epsilon\big),& \bullet=0,\\
0,&\text{otherwise}.
\end{array}
\right.
\]
From the above computation, it is straight forward  to conclude that $\mathfrak{R}_{\exp(2\pi \sqrt{-1}t_0)}$ is a quasi-isomorphism. 

\noindent{\textit{Case (2)}}. We prove that at $\exp(2\pi \sqrt{-1}\frac{j}{w})$, the morphism $\mathfrak{R}_{\exp(2\pi \sqrt{-1}\frac{j}{w})}$ is a quasi-isomorphism. 
Following Lemma \ref{lem:vectorfieldY}, we write  the vector field $Y$ as a sum of two components
\[
\begin{split}
Y&=Y_1+Y_2\\
Y_1&=\sum_{k, kj\notin w\mathbb{Z}} Y^k z_k\frac{\partial}{\partial z_k}+\overline{Y}^k \bar{z}_k \frac{\partial}{\partial \bar{z}_k}\\
Y_2&=(t-\frac{j}{w}) \sum_{k, kj\in w\mathbb{Z}} w_k(a_k z_k\frac{\partial}{\partial z_k}+\bar{a}_k \bar{z}_k\frac{\partial}{\partial z_k}),
\end{split}
\] 
where $a_k=a\big(w_k(t-\frac{j}{w})\big)$. Define $\widetilde{Y}_2$ to be $\sum_{k, kj\in w\mathbb{Z}} w_k(a_k z_k\frac{\partial}{\partial z_k}+\bar{a}_k \bar{z}_k\frac{\partial}{\partial z_k})$. Then we have the following expression for $Y$,
\[
Y=Y_1+(t-\frac{j}{w})\widetilde{Y}_2. 
\]
Accordingly, we can decompose $\mathbb{C}^m$ as a direct sum of two subspaces, that is write 
$\mathbb{C}^m=S_1\times  S_2$ with 
\[
\begin{split}
S_1&:=\bigoplus _{k, kj\notin w\mathbb{Z}} \mathbb{C}_{w_k},\\
S_2&:=\bigoplus_{k, kj\in w\mathbb{Z}} \mathbb{C}_{w_k}. 
\end{split}
\]
Both $S_1$ and $S_2$ are equipped with $S^1$-actions such that the above decomposition of $\mathbb{C}^m$ is $S^1$-equivariant. As our argument is local, we can assume to work with an open set $V$, which is of the product form $V=V_1\times V_2$ such that $V_1$ (and $V_2$) is an $S^1$-invariant neighborhood of $0$ in $S_1$ (and $S_2$). 

We consider $\left(\Omega^\bullet_{S^1 \ltimes V_l \to S^1}\left( \big(\frac{j}{w}-\epsilon, \frac{j}{w}+\epsilon\big) \times V_l \right), i_{Y_l}\right)$ for $l=1,2$. Observe that each complex $\Omega^\bullet_{S^1 \ltimes V_l \to S^1}\left( \big(\frac{j}{w}-\epsilon, \frac{j}{w}+\epsilon\big) \times V_l \right)$ is a $\calC^\infty\big(\frac{j}{w}-\epsilon, \frac{j}{w}+\epsilon\big)$-module, and their tensor product over the algebra $\calC^\infty\big(\frac{j}{w}-\epsilon, \frac{j}{w}+\epsilon \big)$ defines a bicomplex 
\[
\Omega^p_{S^1 \ltimes V_1 \to S^1} \left( \big(\frac{j}{w}-\epsilon, \frac{j}{w}+\epsilon\big) \times V_1 \right)\otimes_{\calC^\infty\big(\frac{j}{w}-\epsilon, \frac{j}{w}+\epsilon\big)} \Omega^q_{S^1 \ltimes V_2 \to S^1} \left( \big(\frac{j}{w}-\epsilon, \frac{j}{w}+\epsilon\big) \times V_2 \right)
\]
with $i_{Y_1}\otimes 1$ being the horizontal differential and $1\otimes i_{Y_2}$ being the vertical one. The total complex of this double complex is exactly 
\[
  \Omega^\bullet_{S^1 \ltimes V \to S^1} \left(
  \big( \frac{j}{w}-\epsilon, \frac{j}{w}+\epsilon\big) \times V \right)
\]
with the differential $i_Y=i_{Y_1}\otimes 1+1\otimes i_{Y_2}$. 
The $E_1$-page of the spectral sequence associated to the bicomplex
\[
  \Omega^\bullet_{S^1 \ltimes V_1 \to S^1}
 \left( \big(\frac{j}{w}-\epsilon, \frac{j}{w}+\epsilon\big) \times V_1 \right)
 \otimes_{\calC^\infty\big(\frac{j}{w}-\epsilon, \frac{j}{w}+\epsilon\big)}
 \Omega^\bullet_{S^1 \ltimes V_2 \to S^1}
 \left(\big(\frac{j}{w}-\epsilon, \frac{j}{w}+\epsilon\big) \times V_2\right)
\] is 
\[
H^\bullet\left(\Omega^\bullet_{S^1 \ltimes V_1\to S^1} \big(\big(\frac{j}{w}-\epsilon, \frac{j}{w}+\epsilon\big) \times V_1\big), i_{Y_1}\right)\otimes _{\calC^\infty\big(\frac{j}{w}-\epsilon, \frac{j}{w}+\epsilon\big)}  \Omega^q_{S^1 \ltimes V_2 \to S^1} \big(\big(\frac{j}{w}-\epsilon, \frac{j}{w}+\epsilon\big) \times V_2\big),
\]
with the differential $1\otimes i_{Y_2}$. We observe that $Y_1$ vanishes only at $0$ for every fixed $t$. Therefore, $\left(\Omega^\bullet_{S^1 \ltimes V_1\to S^1} \big(\big(\frac{j}{w}-\epsilon, \frac{j}{w}+\epsilon\big) \times V_1\big), i_{Y_1}\right)$ is a smooth family of generalized Koszul complexes. Its cohomology is computed by Proposition \ref{prop:parametrizedkoszul} as follows,
\[
  H^\bullet\left(\Omega^\bullet_{S^1 \ltimes V_1\to S^1} \big(\big(\frac{j}{w}-\epsilon, \frac{j}{w}+\epsilon\big) \times V_1\big), i_{Y_1}\right)=\left\{
    \begin{array}{ll}
      \calC^\infty(\frac{j}{w}-\epsilon, \frac{j}{w}+\epsilon)&\bullet=0\\
      0&\bullet\neq 0 \ .
    \end{array}
  \right. 
\]
Therefore, we get the following expression of $E_1^{p,q}$, 
\[
E_1^{p,q}=\left\{\begin{array}{ll}
																							 \Omega^q_{S^1 \ltimes V_2 \to S^1} \big(\big(\frac{j}{w}-\epsilon, \frac{j}{w}+\epsilon\big) \times V_2\big),&p=0\\
																				0,&p\neq0. 
																							\end{array}
																						\right. 
\]

Next we compute the cohomology of $(E_1^{0,q}, i_{Y_2})$. Recall by Lemma \ref{lem:vectorfieldY} that $Y_2$ has the form $Y_2=(t-\frac{j}{w})\widetilde{Y}_2$, where $\widetilde{Y}_2$ vanishes exactly at $0$ for every fixed $t\in (\frac{j}{w}-\epsilon, \frac{j}{w}+\epsilon)$. At degree $q$, we notice that if an element $\omega\in \Omega^q_{S^1 \ltimes V_2 \to S^1} ((\frac{j}{w}-\epsilon, \frac{j}{w}+\epsilon) \times V_2)$ belongs to $\ker(i_{Y_2})$, $(t-\frac{j}{w})i_{\widetilde{Y}_2}\omega=0$. Hence $\omega$ belongs to $\ker(i_{\widetilde{Y}_2})$. Hence, we have reached the following equation
\[
\ker(i_{Y_2})=\ker(i_{\widetilde{Y}_2}). 
\]
It is also easy to check that 
\[
  i_{Y_2}\Omega^{q+1}_{S^1 \ltimes V_2 \to S^1}
  \left(\big(\frac{j}{w}-\epsilon, \frac{j}{w}+\epsilon\big) \times V_2\right)=
  \big(t-\frac{j}{w}\big)i_{\widetilde{Y}_2}
  \Omega^{q+1}_{S^1 \ltimes V_2 \to S^1}
  \left( \big(\frac{j}{w}-\epsilon, \frac{j}{w}+\epsilon\big) \times V_2 \right)
\]
We conclude that the quotient
$\ker(i_{Y_2})/ i_{Y_2}\Omega^{q+1}_{S^1 \ltimes V_2 \to S^1} \left(\big(\frac{j}{w}-\epsilon, \frac{j}{w}+\epsilon\big) \times V_2\right)$ is isomorphic to 
\[
\ker (i_{\widetilde{Y}_2})/(t-\frac{j}{w})i_{\widetilde{Y}_2} \Omega^{q+1}_{S^1 \ltimes V_2 \to S^1} \left(\big(\frac{j}{w}-\epsilon, \frac{j}{w}+\epsilon\big) \times V_2\right)
\]
Recall that  the cohomology of $\left(\Omega^{\bullet}_{S^1 \ltimes V_2 \to S^1} \big(\big(\frac{j}{w}-\epsilon, \frac{j}{w}+\epsilon\big) \times V_2\big), i_{\widetilde{Y}_2}\right)$ is computed as follows, 
\[
H^q\left( \Omega^{\bullet}_{S^1 \ltimes V_2 \to S^1} \big(\big(\frac{j}{w}-\epsilon, \frac{j}{w}+\epsilon\big) \times V_2\big), i_{\widetilde{Y}_2}\right)=\left\{
																									\begin{array}{ll}
																									\calC^\infty(\frac{j}{w}-\epsilon, \frac{j}{w}+\epsilon),&q=0\\
																									0,&q\neq 0. 
																									\end{array}
																									\right.
\]
Therefore, for all $q$, we conclude that
\[
i_{\widetilde{Y}_2} \Omega^{q+1}_{S^1 \ltimes V_2 \to S^1} \left(\big(\frac{j}{w}-\epsilon, \frac{j}{w}+\epsilon\big) \times V_2\right)=\ker(i_{\widetilde{Y}_2}), 
\]
and the quotient $\ker(i_{Y_2})/ i_{Y_2}\Omega^{q+1}_{S^1 \ltimes V_2 \to S^1} \left(\big(\frac{j}{w}-\epsilon, \frac{j}{w}+\epsilon\big) \times V_2\right)$ is isomorphic to
\[
\ker(i_{\widetilde{Y}_2})/\big(t-\frac{j}{2}\big) \ker(i_{\widetilde{Y}_2}).
\]

As the $E_2$ page has only nonzero component when $p=0$, the spectral sequence collapses at the $E^2$ page, and we conclude that the cohomology of the total complex, which is the cohomology of $ \Omega^\bullet_{S^1 \ltimes V \to S^1} \left(\big(\frac{j}{w}-\epsilon, \frac{j}{w}+\epsilon\big) \times V\right)$ with the differential $i_{Y_1}\otimes 1+1\otimes i_{Y_2}$, is equal to the quotient
\[
\ker(i_{\widetilde{Y}_2})/\big(t-\frac{j}{2}\big) \ker(i_{\widetilde{Y}_2}) 
\]
for the contraction $i_{\widetilde{Y}_2}$ on $\Omega^\bullet_{S^1 \ltimes V_2 \to S^1} \left(\big(\frac{j}{w}-\epsilon, \frac{j}{w}+\epsilon\big) \times V_2\right)$. 

We now prove that the morphism 
\[
 \mathfrak{R}: \left( \Omega^q_{S^1 \ltimes V \to S^1} (S^1 \ltimes V ), Y \lrcorner \right) 
    \to
    \big( \hrOmega{\bullet}{\Lambda_0} (\Lambda_0(S^1 \ltimes V )), 0 \big)
\]
is a quasi-isomorphism. The above discussion  and description of
$\Lambda_0((\frac{j}{w}-\epsilon, \frac{j}{w}+\epsilon) \times V)$ reduces us to prove
that  the morphism
\[
\mathfrak{R}_2: \left( \Omega^\bullet_{S^1 \ltimes V_2 \to S^1} \big(\big(\frac{j}{w}-\epsilon, \frac{j}{w}+\epsilon\big) \times V_2 \big), Y_2 \lrcorner \right) 
    \to
    \left( \hrOmega{\bullet}{\Lambda_0} \left(\Lambda_0\big(\big(\frac{j}{w}-\epsilon, \frac{j}{w}+\epsilon\big) \times V_2 \big)\right), 0 \right)
\]
is a quasi-isomorphism. 
We prove this by examination of $\mathfrak{R}_2$ in degree $q$. Hereby, we will work with
$\bigwedge^\bullet F$ as it is isomorphic to $\rOmega{\bullet}{\Lambda_0}$ by Proposition \ref{prop:rel-F}. \\

\noindent{ $\bullet\ q\geq 1$}. Recall that $\Gamma^\infty\big((\frac{j}{w}-\epsilon, \frac{j}{w}+\epsilon) \times V_2 , \bigwedge^q F\big)$ is $\bigwedge^q F_{(\frac{j}{w}, z)}$. We observe that the vector field $\widetilde{Y}_2$ at $t=\frac{j}{w}$ coincides with the fundamental vector field of the $S^1$ action on $V_2$. Hence, if $\phi\in \bigwedge^q F_{(\frac{j}{w}, z)}$ is horizontal, $\phi$ satisfies the equation $i_{\widetilde{Y}_2(\frac{j}{w},z)}\phi=0$. As the cohomology of the $(\Omega^\bullet(V_2), i_{\widetilde{Y}_2(0,z)})$ at degree $q$ vanishes, there is a degree $q+1$ form $\psi \in \Omega^\bullet(V_2)$ such that $i_{\widetilde{Y}_2(\frac{j}{w},z)}\psi=\phi$. Define $\omega\in \Omega^\bullet_{S^1 \ltimes V_2 \to S^1} ((\frac{j}{w}-\epsilon, \frac{j}{w}+\epsilon) \times V_2 )$ by $\omega:=i_{\tilde{Y}_2}\psi$, where $\psi$ is viewed as an element in  $\Omega^\bullet_{S^1 \ltimes V_2 \to S^1} ((\frac{j}{w}-\epsilon, \frac{j}{w}+\epsilon) \times V_2 )$ constant along the $t$ direction. Then we can easily check that $\omega $ belongs to the kernel of $i_{\widetilde{Y}_2}$ and $\mathfrak{R}_2(\psi)=\phi$. We conclude that $\mathfrak{R}_2$ is surjective. 

For the injectivity of $\mathfrak{R}_2$, we suppose that $\omega\in \ker(i_{\widetilde{Y}_2})$. Hence, $\mathfrak{R}_2(\omega)(\frac{j}{w},z)=\omega(\frac{j}{w},z)=0$. Then by the parametrized Taylor expansion, we can find a form $\tilde{\omega}\in \Omega^\bullet_{S^1 \ltimes V_2 \to S^1} ((\frac{j}{w}-\epsilon, \frac{j}{w}+\epsilon) \times V_2 )$ such that $\omega=(t-\frac{j}{w})\tilde{\omega}$. As $0=i_{\widetilde{Y}_2}\omega=(t-\frac{j}{w})i_{\widetilde{Y}_2}\tilde{\omega}_2$, $i_{\widetilde{Y}_2}\tilde{\omega}=0$. Hence $\omega=(t-\frac{j}{w})\tilde{\omega}$ belongs to $(t-\frac{j}{w})\ker{i_{\widetilde{Y}_2}}$, and $[\omega]$ is zero in the cohomology of $i_{Y_2}$. 

\noindent{$\bullet\ q=0$}. Recall that $\widetilde{Y}_2$  is of the form $\sum_{k} w_k\big(a(w_k(t-\frac{j}{w})) z_k\frac{\partial}{\partial z_k}+\bar{a}(w_k(t-\frac{j}{w}))\big)\bar{z}_k\frac{\partial}{\partial \bar{z}_k}$, where $a(w_k(t-\frac{j}{w}))\neq 0$ for all $t\in (\frac{j}{w}-\epsilon, \frac{j}{w}+\epsilon)$. Therefore, the space $(t-\frac{j}{w})i_{\widetilde{Y}_2}$ is of the form
\[
\big( t-\frac{j}{w}  \big)\sum_k z_k f_k +\bar{z}_k g_k, 
\]
which is exactly the vanishing ideal $\calJ\big((\frac{j}{w}-\epsilon, \frac{j}{w}+\epsilon) \times V_2 \big)$. This shows that the cohomology of $ \big( \Omega^\bullet_{S^1 \ltimes V_2 \to S^1} ((\frac{j}{w}-\epsilon, \frac{j}{w}+\epsilon) \times V_2 ), Y_2 \lrcorner \big)$ at degree 0 coincides with $\calC^\infty\big(\Lambda_0(S^1\ltimes V_2)\big)|_{(\frac{j}{w}-\epsilon, \frac{j}{w}+\epsilon) \times V_2}$. One concludes that $\mathfrak{R}_2$ is an isomorphism in degree $0$. 

\end{proof}
\subsection{Stitching it all together}

We are now in a position to prove the Conjecture \ref{BryConjBHF} in the case of circle actions:
\begin{theorem}
Let $M$ be an $S^1$-manifold and regard $\hrOmega{\bullet}{\Lambda_0} \big(\Lambda_0 (S^1\ltimes M)\big)$
  as a chain complex endowed with the zero differential. Then the chain map 
   \[
      \Phi_{\bullet,M/S^1} : C_\bullet  \big( \calC^\infty (M) , \calA (M/S^1)\big)
      \to \hrOmega{\bullet}{\Lambda_0} \big(\Lambda_0 (S^1\ltimes M)\big)
  \]
  is a quasi-isomorphism.
\end{theorem}
\begin{proof}
Since $\Phi_{\bullet,M/S^1}$ is the global sections of a morphism of fine sheaves on $M/S^1$, it suffices to prove that 
\[
    \Phi_\bullet: \sheafC_\bullet  ( \calC^\infty_M , \calA)\big)
    \to \pi_* (s_{|\Lambda_0})_* \rOmega{\bullet}{\Lambda_0} \ , 
\] 
is a quasi-isomorphism, i.e., that the induced map on the stalks $\Phi_{\bullet,\calO}$ is. Now there are two cases, depending on the isotropies
of the orbit $\calO$: when the isotropy subgroup $\Gamma_x\subset S^1$ of a point $x\in S^1$ is a finite group, this follows from the (proof of) 
Corollary \ref{cor:finitegroup}. When the isotropy group is $S^1$ itself, it follows from Proposition \ref{prop:equivariant-koszul}.
\end{proof}


\appendix

\section{Tools from singularity theory}
\label{AppTollsSingularityTheory}
\subsection{Differentiable stratified spaces}
\label{AppDifferentiableStratifiedSpaces}
Recall that for every locally closed subset $X\subset \R^n$ of euclidean space
the sheaf $\calC_X^\infty$ of smooth functions on $X$ is defined as the quotient sheaf 
$\calC_U^\infty/\calJ_{X,U}$, where $U\subset \R^n$ is an open subset such that $X \subset U$ 
is relatively closed, $\calC_U^\infty$ is the sheaf of smooth functions on $U$, and 
$\calJ_{X,U}$ the ideal sheaf of smooth functions on open subsets of $U$ vanishing on $X$.
Note that $\calC_X^\infty$ does not depend on the particular choice of the ambient open
subset $U\subset \R^n$.
\begin{definition}
  A commutative locally ringed space $(A,\calO)$ is called an \emph{affine differentiable space}
  if there is a closed subset $X\subset \R^n$ and an isomorphism of ringed  spaces
  $(f,F) : (A,\calO) \to (X,\calC_X^\infty)$.

  By a \textit{differentiable stratified space} we understand a  commutative locally ringed space
  $(X,\calC^\infty)$ consisting of a separable locally compact topological Hausdorff space $X$
  equipped with a stratification $\calS$ on $X$ in the sense of Mather \cite{MatSM} (cf.~also 
  \cite[Sec.~1.2]{PflAGSSS}) and a sheaf $\calC^\infty$  of commutative local $\C$-rings
  on $X$ such that for every point $x\in X$ there is an open neighborhood $U$ together with 
  $\varphi_1, \ldots , \varphi_n \in \calC^\infty (U) $ having the following properties:
\begin{enumerate}[(DS1)]
\item The map $\varphi: U \rightarrow \R^n$, 
      $y \mapsto (\varphi_1 (y) , \ldots , \varphi_n (y) ) $ is a homeomorphism 
      onto a locally closed subset $\widetilde U := \varphi (U) \subset \R^n$.
      and induces an isomorphism of ringed  spaces
      $\varphi: (U ,\calC^\infty_{|U}) \to (\widetilde U,\calC^\infty_{\widetilde U})$.
\item The map $\varphi$ endows $(U ,\calC^\infty_{|U})$ with the structure of an affine
      differentiable space which means that
      $(\varphi,\varphi^*): (U ,\calC^\infty_{|U}) \to (\widetilde U, \calC^\infty_{\widetilde U})$
      is an isomorphism of ringed  spaces, where $\calC^\infty_{\widetilde U}$ denotes the sheaf
      of smooth functions on $\widetilde U$ as defined above. 
\item For each stratum $S\subset U$, $\varphi_{| S\cap U}$ is a diffeomorphism of 
      $S\cap U$ onto a submanifold $\varphi (S\cap U) \subset \R^n$.
\end{enumerate}
A map $\varphi: U \rightarrow \R^n$ fulfilling the axioms (DS1) to (DS3) will often be called 
a \emph{singular chart} of $X$ (cf.~\cite[Sec.~1.3]{PflAGSSS}).  
\end{definition}

A differentiable stratified space  is in particular a reduced differentiable space 
in the sense of Spallek \cite{SpaDR} or Gonz\'ales--de Salas \cite{GonSalDS}. Moreover, 
differentiable stratified spaces defined as above coincide with the stratified spaces with 
smooth structure as in \cite{PflAGSSS}.

\begin{proposition}[cf.~{\cite[Thm.~1.3.13]{PflAGSSS}}]
  The structure sheaf of a differentiable stratified space is fine. 
\end{proposition}

 To formulate the next result, we introduce the commutative ringed space 
 $(\R^\infty, \calC^\infty_{\R^\infty})$. It is defined as the limit of the direct system of ringed spaces 
 $\big( (\R^n , \calC^\infty_{\R^n}) , \iota_{nm} \big)_{n,m\in \N, \, n\leq m} $, 
 where $\iota_{nm} :\R^n \hookrightarrow \R^m$ is the embedding given by
 \[ 
  \iota_{nm} (v_1,\cdots , v_n) =(v_1,\cdots v_n,0,\cdots ,0).
 \]
 Note that for each open set $U\subset \R^\infty$ the section space 
 $\calC^\infty_{\R^\infty} (U)$ coincides with the inverse limit of the 
 projective system of nuclear Fr\'echet algebras 
 $\big( \calC^\infty_{\R^n} (U\cap \R^n) , \iota_{nm}^* \big)_{n,m\in \N, \, n\leq m} $.
 Hence the $\calC^\infty_{\R^\infty} (U)$  and in particular 
 $\calC^\infty_{\R^\infty} (\R^\infty)$ are  nuclear Fr\'echet algebras  by 
 \cite[Prop.~50.1]{TreTVSDK}.
\begin{proposition}
\label{ThmEmbAppendix} 
  For every differentiable stratified space $(X,\calC^\infty)$  there exists a proper 
  embedding $\varphi: (X,\calC^\infty) \hookrightarrow (\R^\infty, \calC^\infty_{\R^\infty})$.
\end{proposition}
\begin{proof}
  Since $X$ is separable and locally compact there exists a compact exhaustion that is a family
  $(K_k)_{k\in \N}$ of compact subsets $K_k\subset X$ such that $K_k\subset K_{k+1}^\circ$ for all
  $k\in \N$ and such that $\bigcup_{k\in \N}K_k = X$. By \cite[Lem.~1.3.17]{PflAGSSS} there
  then exists an inductively embedding atlas that is a family $(\varphi_k)_{k\in \N}$ of
  singular charts $\varphi_k : K_{k+1}^\circ \to \R^{n_k}$ together with a
  family $(U_k)_{k\in \N}$ of relatively compact open subsets $U_k \subset\subset K_{k+1}^\circ$
  such that $K_k \subset U_k$ and
  $\varphi_{k+1}|_{U_k}  = \iota_{n_k n_{k+1}}\circ  \varphi_k|_{U_k}$
  for all $k\in \N$. Now define $\varphi : X \to \R^\infty $ by
  $\varphi (x) = \varphi_k (x)$ whenever $x \in U_k$. 
  Then $\varphi$ is well defined and an embedding  by construction.
  By a straightforward partition of unity argument one constructs a smooth function
  $\psi:X \to \R$ such that $\psi(x)\geq k $ for all $x \in K_{k+1} \setminus K_k^\circ$.
  The embedding $(\varphi,\psi):X \to \R^\infty \times \R \cong \R^\infty$
  then is proper. 
\end{proof}

\begin{corollary}
  \label{Cor:SmoothMetricAppendix}
  $(X,\calC^\infty)$ be a differential stratified space. Then there exists a
  complete metric $d:X\times X \to \R$ such that $d^2 \in \calC^\infty (X \times X)$. 
\end{corollary}

\begin{proof}
  The euclidean inner product  $\langle -,- \rangle_{\R^n}$ extends in a unique way to an
  inner product $\langle -,- \rangle_{\R^\infty}$ on $\R^\infty$ such that
  $\langle j_n (x),j_n(y) \rangle_{\R^\infty}  = \langle x,y \rangle_{\R^n}$ for all
  $n\in \N $ and $x,y\in \R^n$,
  where $j_n: \R^n \hookrightarrow \R^\infty$ is the canonical embedding 
  $(x_1,\ldots ,x_n)  \mapsto (x_1,\ldots ,x_n, 0, \ldots,0,\ldots )$.
  The associated metric  $d_{\R^\infty}: \R^\infty \times \R^\infty \to \R $, $(x,y)
  \mapsto \sqrt{\langle x-y,x-y \rangle_{\R^\infty}}$
  then is related to the euclidean metric $d_{\R^n}$ by 
  $d_{\R^\infty} \big( j_n (x),j_n(y) \big)  = d_{\R^n}(x,y)$ for $x,y\in \R^n$.
  Now choose a proper embedding $X \hookrightarrow  \R^\infty$ and denote the restriction of
  $d_{\R^\infty}$ to $X$ by $d$. By construction, $d^2$ then is smooth. Moreover,
  $d$ is a complete metric since the embedding is proper and each of the metrics $d_{\R^n}$
  is complete. 
\end{proof}



%
%
\section{The cyclic homology  of  bornological algebras} 
\label{AppCyHomBornAlg}
\subsection{Bornological vector spaces and tensor products}
\label{AppBornVecSp}
We recall some basic notions from the theory of bornological vector spaces 
and their tensor products. For details we refer to \cite{HogBFA} and
\cite[Chap.~1]{MeyLACH}.
\begin{definition}[cf.~{\cite[Chap.~I, 1:1 Def.]{HogBFA}}]
  By a \emph{bornology} on a  set $X$ one understands a set $\scrB$ of subset 
  of $X$ such that the following conditions hold true:
  \begin{enumerate}[(BS)]
  \item $\scrB$ is a covering of $X$, $\scrB$ is hereditary under inclusions,
        and $\scrB$ is stable under finite unions.
   \end{enumerate}

  A map $f: X \to Y$ from a set $X$ with bornology $\scrB$ to a set $Y$ carrying
  a bornology $\scrD$ is called \emph{bounded}, if the following is satisfied:
  \begin{enumerate}[(BM)]
  \item 
    The map $f$ preserves the bornologies, i.e.~$f(B) \in \scrD$ for all 
    $B \in \scrB$.    
  \end{enumerate}
  
  If $V$ is a vector space over $\Bbbk = \R$ or $\Bbbk = \C$, 
  a bornology $\scrB$ is called a \emph{convex vector bornology} on $V$, if 
  the following additional properties hold true:
  \begin{enumerate}[(BV)]
  \item The bornology $\scrB$ is stable under addition,
        under scalar multiplication, under forming balanced hulls, and
        finally under forming convex hulls.
  \end{enumerate}
 
  A set together with a bornology is called a \emph{bornological set}, a 
  vector space with a convex vector bornology a \emph{bornological vector space}. 
  For clarity, we sometimes denote a bornological vector space as a pair $(V,\scrB)$,
  where $V$ is the underlying vector space, and $\scrB$ the corresponding convex vector 
  bornology.  

  A bornological vector space $(V,\scrB)$ is called \emph{separated},
  if the condition (S) below is satisfied. 
  If in addition condition (C) holds true as well,  $(V,\scrB)$ is called
  \emph{complete}.  
  \begin{enumerate}
  \item[(S)] The subspace $\{ 0 \}$ is the only bounded subvector space of $V$.
  \item[(C)] Every bounded set is contained in a completant bounded disk, 
             where a disk $D \subset V$ is called \emph{completant}, if the space 
             $V_D$ spanned by $D$ and semi-normed by the gauge of $D$ is a Banach space.
  \end{enumerate}
\end{definition}

As for the category of topological vector spaces there exist functors of separation and 
completion within the category of bornological vector spaces. 

\begin{example}
  Let $V$ be a locally convex topological vector space. The \emph{von Neumann bornology}
  on $V$ consists of all (von Neumann) bounded subsets of $V$, ie.~of all 
  $B\subset V$ which are absorbed by every $0$-neighborhood. One immediately checks that
  the von Neumann bornology is a convex vector bornology on $V$. We sometimes denote 
  this bornology by $\scrB_\textup{vN}$.
\end{example}

Similarly to the topological case, the 
bornological tensor product is defined by a universal property. 

\begin{definition}
\label{DefProjBorTensorProd}
The (\emph{projective}) \emph{bornological tensor product} of two bornological spaces 
vector spaces $\big( V_1 , \calB_1 \big)$ and $\big( V_2 , \calB_2 \big)$ is defined as 
the up to isomorphism uniquely defined bornological vector space 
$\big( V_1 \otimes V_2 , \scrB_1 \otimes \scrB_2 \big)$ together with a 
bounded map $V_1 \times V_2 \to V_1 \otimes V_2$  such that  for each 
bornological vector space $( W , \scrD ) $ and bounded bilinear map 
$\lambda : V_1 \times V_2 \to W$ there is a unique bounded map 
$\overline{ \lambda}:V_1 \otimes V_2 \to W$ making the diagram 
\begin{displaymath}
  \xymatrix{
  V_1 \otimes V_2 \ar[r]^\lambda\ar[d] & W \\ 
  V_1\otimes V_2 \ar[ur]_{\overline{\lambda}}
}
\end{displaymath}
commute. The completion of the bornological tensor product will be denoted 
by $V_1 \hat{\otimes} V_2$.  
\end{definition}

\begin{remark}
  \begin{enumerate}
  \item Note that the underlying vector space of the bornological tensor product coincides with 
        the algebraic tensor product of the vector spaces $V_1 \otimes V_2$.
  \item
     Since tensor products of topological vector spaces are also needed in this paper, let us briefly recall that 
     the  complete projective (resp.~inductive) topological tensor product $\hat{\otimes}_\pi$ (resp.~$\hat{\otimes}_\iota$) 
     can be defined as the (up to isomorphism) unique bifunctor on the category of complete locally convex  
     topological vector spaces which is universal with respect to jointly (resp.~separately) 
     continuous bilinear maps with values in complete locally convex topological vector spaces.
     For Fr\'echet spaces, the complete projective and complete inductive tensor products coincide, since
     separately continuous bilinear maps on Fr\'echet spaces are automatically jointly continuous. 
     See \cite{GroPTTEN} and \cite{MeyLACH} for details. 
  \end{enumerate}
\end{remark}
\subsection{The Hochschild chain complex}
In this section we recall  the construction of the cyclic bicomplex associated 
to a complete bornological algebra $A$ which not necessarily is assumed to be 
unital. 
To this end observe first that the space of Hochschild $k$-chains 
$C_k (A) :=  A^{\hat{\otimes} (k+1)} $ is defined using the complete projective bornological tensor 
product $\hat{\otimes}$. Together with the face maps 
\[
  b_{k,i} : C_k (A) \to C_{k-1} (A), \: a_0 \otimes \ldots \otimes a_k \mapsto  
  \begin{cases}
    a_0 \otimes \ldots \otimes a_ia_{i+1} \otimes \ldots \otimes a_k, & \text{if $0\leq i<k$},\\
    a_ka_0 \otimes \ldots \otimes a_{k-1}, & \text{if $i =k$},
  \end{cases}
\]
and the cyclic operators 
\[
  t_k : C_k (A) \to C_k (A), \: a_0 \otimes \ldots \otimes a_k \mapsto  (-1)^k \,
  a_k \otimes a_0 \otimes \ldots \otimes a_{k-1}
\] 
the graded linear space of Hochschild chains $C_\bullet ( A ) := \big( C_k (A) \big)_{k\in \N}$ then
becomes a pre-cyclic object (see for example \cite{LodCH} for the precise commutation relations
of the face and cyclic operators).

From the pre-cyclic structure one obtains two boundary maps, namely the one of the Bar complex 
$b' :  C_k (A) \to C_{k-1} (A)$, $b' := := \sum_{i=0}^{k-1} (-1)^i b_i$ 
and the Hochschild boundary $b :  C_k (A) \to C_{k-1} (A)$, $b := b' + (-1)^k b_k$. 
The commutation relations for the $b_i$ immediately entail $b^2 = (b')^2 = 0$.
This gives rise to the following two-column bicomplex.
\begin{displaymath}
  \xymatrix{ 
  \vdots \ar[d] & \vdots \ar[d] \\
  C_2 (A) \ar[d]_b & \ar[l]_{1-t} C_2(A) \ar[d]_{-b'}\\
  C_1(A)  \ar[d]_b & \ar[l]_{1-t} C_1(A) \ar[d]_{-b'} \\
  C_0(A) & \ar[l]_{1-t} C_0(A)
  }
\end{displaymath}
We will denote this two-column bicomplex by $C_{\bullet,\bullet} (A)^{\{2\}}$. By definition, the homology of
its total complex is the Hochschild homology 
\begin{equation}
  \label{eq:DefHochschildHom}
  HH_\bullet (A ) := H_\bullet \big(\Tot_\bullet\big( C_{\bullet,\bullet}(A)^{\{2\}}\big)\big) \: .
\end{equation}
\subsection{A twisted version of the theorem by Hochschild--Kostant--Rosenberg and Connes}
The classical theorem by Hochschild--Kostant--Rosenberg identifies the Hochschild homology of 
the algebra of regular functions on a smooth affine variety with the graded module of 
K\"ahler forms of that algebra \cite{HocKosRosDFRAA}. In his seminal paper \cite{ConNDG}, Connes proved 
that for compact smooth manifolds an analogous result holds true that is the (continuous) Hochschild
homology of the algebra of smooth functions on a manifold coincides naturally with the 
complex of differential forms over the manifold (see \cite{PflCHHCG} for the non-compact case of that result).
Here we show a twisted version of this theorem. That result appears to be folklore, cf.~also \cite{BroDavNis}. 

Assume that $h$ is an  orthogonal transformation acting on some euclidean space  $\R^d$.
Let $V$ be an open ball around the origin of $\R^d$. 
Then we denote by $\ltwist{\calC^\infty (V)}{h}$ the space $\calC^\infty (V)$  with the $h$-twisted 
$\calC^\infty (V)$-bimodule structure
\begin{displaymath}
  \calC^\infty (V) \hatotimes \, \ltwist{\calC^\infty (V)}{h} \hatotimes  \calC^\infty (V) \to 
  \ltwist{\calC^\infty(V)}{h} , \:
  f \otimes a \otimes f' \mapsto \Big( V \ni v \mapsto f(h v) \, a(v) f'(v) \in \R \Big) \ .
\end{displaymath}
In the following we compute the \emph{twisted} Hochschild homology $H_\bullet \big(\calC^\infty (V), \ltwist{\calC^\infty(V)}{h}\big)$. 
Denote by $\langle -,-\rangle$ the euclidean inner product on $\R^d$. By the orthogonality assumption  
$\langle -,-\rangle$ is $G$-invariant, hence $V$ is so, too. 
Recall that for every topological projective resolution $R_\bullet \to \calC^\infty (V)$ of $\calC^\infty (V)$
as a $\calC^\infty (V)$-bimodule the Hochschild homology groups $H_k (\calC^\infty (V), \ltwist{\calC^\infty(V)}{h}$
are naturally isomorphic to the homology groups $H_k \big( R_\bullet , \ltwist{\calC^\infty(V)}{h} \big)$,
see \cite{HelHBTA}.
Recall further that a topological projective resolution of the $\calC^\infty (V)$-bimodule  
$\calC^\infty (V)$ is given by the Connes-Koszul resolution \cite[p.~127ff]{ConNDG}
\begin{equation}
  \label{eq:ResSmoothFunc}
  \Gamma^\infty (V \times V , E_d ) \overset{i_Y}{\longrightarrow} \ldots \overset{i_Y}{\longrightarrow}
  \Gamma^\infty (V \times V , E_1 ) \overset{i_Y}{\longrightarrow} \calC^\infty (V \times V) 
  \longrightarrow \calC^\infty (V) \longrightarrow 0,
\end{equation} 
where $E_k$ is the pull-back bundle $\prtwo^* \big(\Lambda^k T^*\R^d\big) $ along the 
projection $\prtwo : \R^d \times \R^d \to \R^d$, $(v,w) \mapsto w$, and $i_Y$ denotes contraction with the vector 
field $Y : V \times V \rightarrow \prtwo^* (T\R^d)$, $(v,w) \mapsto w-v$. By tensoring the Connes-Koszul 
resolution with $\ltwist{\calC^\infty(V)}{h}$ one obtains the chain complex
\begin{equation}
  \label{eq:TwistedChainCpl}
       \Omega^d(V)  \overset{i_{Y_h}}{\longrightarrow} \ldots 
       \overset{i_{Y_h}}{\longrightarrow}  \Omega^1(V)
       \overset{i_{Y_h}}{\longrightarrow}  \calC^\infty (V) \longrightarrow 0 ,
\end{equation} 
where the vector field $Y_h : V \rightarrow T\R^d$ is given by $Y_h(v)= v -hv$. Denote by 
$V^h$ the fixed point set of $h$ in $V$, let $\iota_h : V^h \hookrightarrow V$ be the canonical embedding, 
and $\pi_h : V \rightarrow V^h$ the restriction of the orthogonal projection onto the fixed point space 
$(\R^d)^h$. One obtains the following commutative diagram.
\begin{equation}
  \label{eq:DiagQuismsTwisted}
  \vcenter{\vbox{
  \xymatrix{ \Omega^d(V)  \ar[r]^{\hspace{2em}i_{Y_h}} \ar[d]_{\iota_h^*}  & \ldots \ar[r]^{\hspace{-1em}i_{Y_h}}   &  
       \Omega^1(V) \ar[r]^{i_{Y_h}} \ar[d]_{\iota_h^*} & \calC^\infty (V) \ar[d]_{\iota_h^*}  \\
       \Omega^d(V^h) \ar[r]^{\hspace{2em}0} \ar[d]_{\pi_h^*} & \ldots \ar[r]^{\hspace{-1em}0} &  \Omega^1(V^h) \ar[r]^{0}\ar[d]_{\pi_h^*} & 
       \calC^\infty (V^h) \ar[d]_{\pi_h^*}\\
       \Omega^d(V) \ar[r]^{\hspace{2em}i_{Y_h}} & \ldots \ar[r]^{\hspace{-1em}i_{Y_h}}   &  
       \Omega^1(V) \ar[r]^{i_{Y_h}}& \calC^\infty (V)  
  }}}
\end{equation}
\begin{proposition}\label{prop:homotopy}
  The chain maps $\iota_h^*$ and $\pi_h^*$ are quasi-isomorphisms. 
\end{proposition}
\begin{proof}
  Since the restriction of the vector field $Y_h$ to $V^h$ vanishes, the diagram \eqref{eq:DiagQuismsTwisted} 
  commutes, and the $\iota_h^*$ and $\pi_h^*$ are chain maps indeed. 
  Let $W$ be the orthogonal complement of $(\R^d)^h$ in $\R^d$, $m = \dim W$, and $\pi_W := \id_V -\pi_h$ 
  the orthogonal projection onto $W$.
  Since the $h$-action on   $W$ is orthogonal and has as only fixed point  the origin, there exists an orthonormal basis 
  $w_1, \ldots ,w_m$ of $W$, a natural $l\leq \frac m2$, and 
  $\theta_1, \ldots , \theta_l \in (-\pi,\pi) \setminus \{ 0 \}$ 
  such that the following holds: 
  \[
    hw_k =
    \begin{cases}
      \cos \theta_i \, w_{2i-1} + \sin \theta_i \, w_{2i} & \text{if $ k = 2i -1 $ with $i \leq l$}, \\
      - \sin \theta_i \, w_{2i-1} + \cos \theta_i \, w_{2i} & \text{if $k=2i$ with $i \leq l$}, \\
      - w_k & \text{if $2l < k \leq m$}.
    \end{cases}
  \]
  Denote by $\varphi_t : \R^d \to \R^d$, $t \in \R$ the flow of the complete vector field $Y_h$ or in other words 
  the solution of the initial value problem $\frac{d}{dt} \varphi_t  = ( \id_V - h )  \varphi_t $,
  $\varphi_0  = \id_V$. Then $\varphi_t v = v$ for all $v \in (\R^d)^h$, and 
  \begin{equation}
    \label{eq:FundSolutions}
    \varphi_t (w_k) =
    \begin{cases}
      e^{(1-\cos \theta_i)t}\big(  \cos ( t \sin \theta_i ) \, w_{2i-1} + \sin ( t \sin \theta_i ) \, w_{2i} \big),
      & \text{if $k=2i-1$ with $i \leq l$}, \\ 
      e^{(1-\cos \theta_i)t}\big( - \sin ( t \sin \theta_i ) \, w_{2i-1} + \cos ( t \sin \theta_i ) \, w_{2i} \big),
      & \text{if $ k = 2i $ with $i \leq l$}, \\
      e^{2t} w_k, & \text{if $2l < k \leq m$}.
    \end{cases}
  \end{equation}
  Now let $v_1,\ldots , v_n$ be a basis of $V^h$, and denote by $v^1,\ldots,v^n,w^1, \ldots ,w^m$ the basis of $V'$
  dual to $v_1,\ldots , v_n, w_1, \ldots ,w_m$. Then every $k$-form $\omega$ on $V$ is the sum 
  of monomials $dv^{i_1} \wedge \ldots \wedge dv^{i_l} \wedge \omega_{i_1 , \ldots , i_l}$, where 
  $1 \leq i_1 < \ldots < i_l \leq n$ and 
  $\omega_{i_i , \ldots , i_l} = i_{v_{i_1} \wedge \ldots \wedge v_{i_l}} \omega \in \Gamma^\infty \big(\pi_W^* \Lambda^{k-l} T^*W\big)$. 
  Let $d_W$ be the restriction of the exterior differential to $\Gamma^\infty \big(\pi_W^* \Lambda^\bullet T^*W\big)$ and
  define $S: \Omega^k (V) \to \Omega^{k+1}(V) $ by its action on the monomials:
  \[
    S\omega = \sum_{l=0}^k \sum_{1 \leq i_1 < \ldots < i_l \leq n} dv^{i_1} \wedge \ldots \wedge dv^{i_l} \wedge 
    \int_{-\infty}^0 \varphi_t^* (d_W\omega_{i_1 , \ldots , i_l}) \, dt . 
  \]
  Note that  the integral is well-defined since 
  $\varphi_t (  V ) \subset V$ for all $t \leq 0$ by Eq.~\eqref{eq:FundSolutions}.   
  Observe that  ${\varphi_t}_* Y_h = Y_h$ by construction of $\varphi_t$ and 
  that  the fibers of the projection 
  $\pi_h$ are left invariant by $\varphi_t$. Hence one concludes by Cartan's magic formula
  \begin{equation}
  \label{eq:HomProp}
  \begin{split}  
     (S i_{Y_h} +  i_{Y_h} S) \omega = &\,  \sum_{l=0}^k\sum_{1 \leq i_1 < \ldots < i_l \leq n} dv^{i_1} \wedge \ldots \wedge dv^{i_l} \wedge 
     \int_{-\infty}^0 ( d_W i_{Y_h}  +  i_{Y_h} d_W) \varphi_t^* \omega_{i_1 , \ldots , i_l} \, dt =\\
     = &\,  \sum_{l=0}^k\sum_{1 \leq i_1 < \ldots < i_l \leq n} dv^{i_1} \wedge \ldots \wedge dv^{i_l} \wedge 
     \int_{-\infty}^0 \calL_{Y_h}  \varphi_t^* \omega_{i_1 , \ldots , i_l} \, dt = \\
     = &\,  \sum_{l=0}^k\sum_{1 \leq i_1 < \ldots < i_l \leq n} dv^{i_1} \wedge \ldots \wedge dv^{i_l} \wedge 
     \int_{-\infty}^0 \frac{d}{dt} \varphi_t^* \omega_{i_1 , \ldots , i_l} \, dt = \\
     = &\, \sum_{l=0}^{k-1}\sum_{1 \leq i_1 < \ldots < i_l \leq n} dv^{i_1} \wedge \ldots \wedge dv^{i_l} \wedge \omega_{i_1 , \ldots , i_l} \: +\\
     &  + \sum_{1 \leq i_1 < \ldots < i_k \leq n} dv^{i_1} \wedge \ldots \wedge dv^{i_k} \wedge 
        (\omega_{i_1 , \ldots , i_k} -  \pi_h^* \iota_h^* \omega_{i_1 , \ldots , i_k}   ) =\\
     = & \, \omega - \pi_h^* \iota_h^* \omega .  
  \end{split}
  \end{equation}
  To verify the second last equality observe that the $\omega_{i_1 , \ldots , i_k}$ are smooth functions which satisfy
  $$
     \lim_{t \rightarrow -\infty} \varphi_t^* \omega_{i_1 , \ldots , i_k} =  \pi_h^* \iota_h^* \omega_{i_1 , \ldots , i_k} .
  $$
  Eq.~\eqref{eq:HomProp} proves the claim. 
\end{proof}

The proposition immediately entails the following twisted version of the theorem by Hochschild--Kostant--Rosenberg and Connes. 
\begin{theorem}\label{thm:twisted-hochschild-homology} 
   Let $h:\R^d \to\R^d$ be an orthogonal linear transformation and $V \subset \R^d$ an open ball around the origin.
   Then the  Hochschild homology $H_\bullet \big(\calC^\infty (V), \ltwist{\calC^\infty(V)}{h}\big)$
   is naturally isomorphic to $\Omega^\bullet (V^h)$, where $V^h$ is the fixed point manifold of $h$ in $V$. 
   A quasi-isomorphism inducing this identification is given by 
   \[
    \ltwist{\calC^\infty(V)}{h} \hatotimes \, C_k(\calC^\infty (V) ) \to 
    \Omega^k (V^h), \: f_0 \otimes f_1 \otimes \ldots \otimes f_k 
    \mapsto {f_0}_{|V^h} \, d{f_1}_{|V^h} \wedge \ldots \wedge d {f_k}_{|V^h}  \ . 
   \]
\end{theorem}

We consider a finite subgroup $\Gamma$ of the orthogonal linear transformation group of $\R^d$.  Let $V\subset \R^d$ be an open ball around
the origin that is invariant with respect to the $\Gamma$ action on $\R^d$.  We can apply the quasi-isomorphism from
Section \ref{sec:group-manifold-case} to compute $HH_\bullet \big(\calC^\infty(V)\rtimes \Gamma)$ by the  homology of the complex
$ C^\Gamma_\bullet(\calC^\infty(V))$. Since $\Gamma$ is a finite group, the homology of  $ C^\Gamma_\bullet(\calC^\infty(V))$ is computed by 
 \[
\Big(\bigoplus_{\gamma\in \Gamma} H_\bullet \big(\calC^\infty (V), \ltwist{\calC^\infty(V)}{\gamma}\big)\Big)^\Gamma. 
 \]
 As a corollary to Theorem \ref{thm:twisted-hochschild-homology}, we thus obtain the following computation of the Hochschild homology of
 $\calC^\infty(V)\rtimes \Gamma$. 
 \begin{corollary}\label{cor:finitegroup}
   The Hochschild homology $HH_\bullet \big(\calC^\infty(V)\rtimes \Gamma)$ is naturally
   isomorphic to 
\[
  \left(\bigoplus_{\gamma \in \Gamma} \Omega^\bullet (V^\gamma) \right)^\Gamma,
\]
  where $\Gamma$ acts on the disjoint union  $\coprod_{\gamma \in \Gamma} V^\gamma$
  by $\gamma'(\gamma,x)=(\gamma' \gamma(\gamma')^{-1}, \gamma' x)$. 
\end{corollary}

In the case of a smooth affine algebraic variety, Corollary \ref{cor:finitegroup} is proved
by \cite[Thm.~2.19]{BroDavNis}. We refer the reader to \cite{BryNis, FeiTsy,Wassermann, Ponge}
for related developments.

We end with a generalisation of Proposition \ref{prop:homotopy} which is a useful tool in our computations.
Observe that in the complex (\ref{eq:TwistedChainCpl})
\[
       \Omega^d(V)  \overset{i_{Y_h}}{\longrightarrow} \ldots 
       \overset{i_{Y_h}}{\longrightarrow}  \Omega^1(V)
       \overset{i_{Y_h}}{\longrightarrow}  \calC^\infty (V) \longrightarrow 0 ,
\]
the vector field $Y_h$ can be extended to be a more general linear vector field
$Y_H: \mathbb{R}^n\to T\mathbb{R}^d$ of the form $Y_H(v)=H(v) \in T_v \mathbb{R}^d$ where
$H: \mathbb{R}^d\to \mathbb{R}^d$ is a diagonalizable linear map. A construction similar to the homotopy operator $S$ in the proof of Proposition \ref{prop:homotopy} (see also \cite{Wassermann}) computes the homology of
$(\Omega^\bullet(V), i_{Y_H})$ to be $(\Omega^\bullet(V^H), 0)$ where $V^H=\ker(H)$. Furthermore, if $H:S\to \End(\mathbb{R}^d)$ is a smooth family of diagonalizable linear operators parametrized by a smooth manifold $S$, $H$ is called regular if $H$ satisfies the following properties:
\begin{enumerate}
\item the kernel $\ker(H):=\{\ker(H(s))\}_{s\in S}\subset S\times \mathbb{R}^d$ is a smooth subbundle of the trivial vector bundle $S\times \mathbb{R}^d$;
\item near every $s_0\in S$, there is a  local frame of $S\times \mathbb{R}^d$ on a neighborhood $U_{s_0}$ of $s_0$ in $S$ consisting of  $\xi_1, \cdots, \xi_d$ such that 
\begin{itemize}
\item the collection $\{\xi_1, \cdots, \xi_k\}$ is a local frame of the subbundle $\ker(H)$ on $U_{s_0}$,
\item for every $j=k+1,\cdots, d$, there is a smooth eigenfunction $\lambda_j(s)$ defined on $U_{s_0}$ satisfying $H(s)\xi_j(s)=\lambda_j(s)\xi_j(s)$ and $\lambda_j(s)\ne 0,\ \forall s\in U_{s_0}$. 
\end{itemize}
\end{enumerate}
The proof of Proposition \ref{prop:homotopy} generalizes to the following result.    
\begin{proposition}\label{prop:parametrizedkoszul} Let $H: S\to \End(\mathbb{R}^d)$ be a smooth family of diagonalizable linear operators parametrized by a smooth manifold $S$. Assume that $H$ is regular. Let $i_{\ker(H)}: \ker(H)\to S\times \mathbb{R}^d$ be the canonical embedding, and $\Omega^\bullet \big(\ker(H)\big)$ the restriction of $\calC^{\infty}(S, \Omega^\bullet(V))$ to $\ker(H)$ along $i_{\ker(H)}$. Then the restriction map $R_{\ker(H)}:\big(\calC^{\infty}(S, \Omega^\bullet(V)), i_{Y_H}\big)\to \Big(\Omega^\bullet \big( \ker(H)\big), 0\Big)$ is a quasi-isomorphism. 
\end{proposition}

In a certain sense, the final result is  variant of the latter. To formulate it
recall that by an Euler-like vector field for an embedded smooth manifold  $S \hookrightarrow M$
one understands a vector field $Y:M\to TM$ such that $S$ is the zero set of $Y$ and such that
for each $f\in \calC^\infty (M)$ vanishing on $S$ the function $Yf -f$ vanishes to second order on $S$;  
cf.~\cite[Def.~1.1]{SadHigEulerLike}.

\begin{proposition}\label{prop:parametrizedkoszuleulervectorfield}
  Let $M$ be a smooth manifold of dimension $d$, $S\hookrightarrow M$ an embedded submanifold and
  $Y :M\to TM$a smooth vector field which is Euler like with respect to $S$. Then the complex
  \begin{equation}\label{eq:parametrizedkoszuleulervectorfield}
     \Omega^d(M)  \overset{i_{Y}}{\longrightarrow} \ldots 
     \overset{i_{Y}}{\longrightarrow}  \Omega^1(M)
     \overset{i_{Y}}{\longrightarrow}  \calC^\infty (M) \longrightarrow \calC^\infty (S) \longrightarrow 0 ,
  \end{equation}
  is exact and will be  called the \emph{parametrized Koszul resolution} of $\calC^\infty (S)$.
\end{proposition}

\begin{proof}
  The claim is an immediate consequence of the Koszul resolution as for example stated
  in \cite[]{Wassermann}.
\end{proof}


\bibliographystyle{amsalpha}    
\bibliography{ncgeometry&groupoids}

\end{document}